\documentclass{article}

\usepackage{times} 

\usepackage[all]{xy}
\xyoption{ps}      
\xyoption{dvips}   
\CompileMatrices   
\usepackage{llamp}

\input graphical.sty

\title{A survey of graphical languages for monoidal categories}
\author{Peter Selinger}
\date{Dalhousie University}

\begin{document}
\maketitle

\begin{abstract}
  This article is intended as a reference guide to various notions of
  monoidal categories and their associated string diagrams. It is
  hoped that this will be useful not just to mathematicians, but also
  to physicists, computer scientists, and others who use diagrammatic
  reasoning. We have opted for a somewhat informal treatment of
  topological notions, and have omitted most proofs.  Nevertheless,
  the exposition is sufficiently detailed to make it clear what is
  presently known, and to serve as a starting place for more in-depth
  study. Where possible, we provide pointers to more rigorous
  treatments in the literature. Where we include results that have
  only been proved in special cases, we indicate this in the form of
  caveats.
\end{abstract}

\tableofcontents

\section{Introduction}

There are many kinds of monoidal categories with additional structure
--- braided, rigid, pivotal, balanced, tortile, ribbon, autonomous,
sovereign, spherical, traced, compact closed, *-autonomous, to name a
few. Many of them have an associated graphical language of ``string
diagrams''. The proliferation of different notions is often confusing
to non-experts, and occasionally to experts as well. To add to the
confusion, one concept often appears in the literature under multiple
names (for example, ``rigid'' is the same as ``autonomous'',
``sovereign'' is the same as ``pivotal'', and ``ribbon'' is the same
as ``tortile''). 

In this survey, I attempt to give a systematic overview of the main
notions and their associated graphical languages. My initial intention
was to summarize, without proof, only the main definitions and
coherence results that appear in the literature.  However, it quickly
became apparent that, in the interest of being systematic, I had to
include some additional notions. This led to the sections on spacial
categories, and planar and braided traced categories.

Historically, the terminology was often fixed for special cases before
more general cases were considered. As a result, some concepts have a
common name (such as ``compact closed category'') where another name
would have been more systematic (e.g.~``symmetric autonomous
category'').  I have resisted the temptation to make major changes to
the established terminology. However, I propose some minor tweaks that
will hopefully not be disruptive. For example, I prefer ``traced
category'', which can be combined with various qualifying adjectives,
to the longer and less flexible ``traced monoidal category''.

Many of the coherence results are widely known, or at least presumed
to be true, but some of them are not explicitly found in the
literature. For those that can be attributed, I have attempted to do
so, sometimes with a caveat if only special cases have been proved in
the literature. For some easy results, I have provided proof sketches.
Some unproven results have been included as conjectures.

While the results surveyed here are mathematically rigorous, I have
shied away from giving the full technical details of the definitions
of the graphical languages and their respective notions of equivalence
of diagrams.  Instead, I present the graphical languages somewhat
informally, but in a way that will be sufficient for most
applications. Where appropriate, full mathematical details can be
found in the references.

Readers who want a quick overview of the different notions are
encouraged to first consult the summary chart at the end of this
article.

An updated version of this article will be maintained on the ArXiv, so
I encourage readers to contact me with corrections, literature
references, and updates.

\paragraph{Graphical languages: an evolution of notation.}

The use of graphical notations for operator diagrams in physics goes
back to Penrose {\cite{Pen71}}. Initially, such notations applied to
multiplications and tensor products of linear operators, but it became
gradually understood that they are applicable in more general
situations.

To see how graphical languages arise from matrix multiplication,
consider the following example. Let $M:A\ii B$, $N:B\x C\ii D$,
and $P:D\ii E$ be linear maps between finite dimensional vector spaces
$A,B,C,D,E$. These maps can be combined in an obvious way to obtain a
linear map $F:A\x C\ii E$. In functional notation, the map $F$
can be written
\begin{equation}\label{eqn-functional}
  F = P\cp N\cp (M\x \id_C).
\end{equation}
The same can be expressed as a summation over matrix indices, relative
to some chosen basis of each space. In mathematical notation, suppose
$M=(m_{j,i})$, $N=(n_{l,jk})$, $P=(p_{m,l})$, and $F=(f_{m,ik})$,
where $i,j,k,l,m$ range over basis vectors of the respective spaces.
Then
\begin{equation}\label{eqn-math}
  f_{m,ik} = \sum_j\sum_l p_{m,l} n_{l,jk}   m_{j,i}.
\end{equation}
In physics, it is more common to write column indices as superscripts
and row indices as subscripts. Moreover, one can drop the summation
symbols by using Einstein's summation convention.
\begin{equation}\label{eqn-physics}
  F_{m}^{ik} = P_{m}^{l} N_{l}^{jk} M_{j}^{i}.
\end{equation}
In (\ref{eqn-math}) and (\ref{eqn-physics}), the order of the
factors in the multiplication is not relevant, as all the information
is contained in the indices. Also note that, while the notation
mentions the chosen bases, the result is of course basis independent.
This is because indices occur in pairs of opposite variance (if on the
same side of the equation) or equal variance (if on opposite sides of
the equation). It was Penrose {\cite{Pen71}} who first pointed out
that the notation is valid in many situations where the indices are
purely formal symbols, and the maps may not even be between vector
spaces.

Since the only non-trivial information in (\ref{eqn-physics}) is in
the pairing of indices, it is natural to represent these pairings
graphically by drawing a line between paired indices. Penrose
{\cite{Pen71}} proposed to represent the maps $M,N,P$ as boxes, each
superscript as an incoming wire, and each subscript as an outgoing
wire. Wires corresponding to the same index are connected.  Thus, we
obtain the graphical notation:
\begin{equation}
  \vcenter{\wirechart{@R-.2cm}{
      *{}\wireright{r}{k}&\blank\\
      &\blank\nnbox{[u].[d]}{F}\wireright{r}{m}&*{}\\
      *{}\wireright{r}{i}&\blank\\
      }}
  =
  \vcenter{\wirechart{@R-.2cm}{
      *{}\wireright{rr}{k}&&\blank\\
      &&\blank\nnbox{[u].[d]}{N}\wireright{r}{l}&\blank\nnbox{[]}{P}\wireright{r}{m}&*{}\\
      *{}\wireright{r}{i}&\blank\nnbox{[]}{M}\wireright{r}{j}&\blank
      }}
\end{equation}
Finally, since the indices no longer serve any purpose, one may omit
them from the notation. Instead, it is more useful to label each wire
with the name of the corresponding space.
\begin{equation}\label{eqn-cat-diagram}
  \vcenter{\wirechart{@R-.2cm}{
      *{}\wireright{r}{C}&\blank\\
      &\blank\nnbox{[u].[d]}{F}\wireright{r}{E}&*{}\\
      *{}\wireright{r}{A}&\blank\\
      }}
  =
  \vcenter{\wirechart{@R-.2cm}{
      *{}\wireright{rr}{C}&&\blank\\
      &&\blank\nnbox{[u].[d]}{N}\wireright{r}{D}&\blank\nnbox{[]}{P}\wireright{r}{E}&*{}\\
      *{}\wireright{r}{A}&\blank\nnbox{[]}{M}\wireright{r}{B}&\blank
      }}
\end{equation}
In the notation of monoidal categories, (\ref{eqn-cat-diagram}) can be
expressed as a commutative diagram
\begin{equation}
  \xymatrix{
    A\x C \ar[r]^<>(.5){F} \ar[d]_{M\x\id_C} & E \\
    B\x C \ar[r]^<>(.5){N} & D \ar[u]_{P},
    }
\end{equation}
or simply:
\begin{equation}
  F = P\cp N\cp (M\x \id_C).
\end{equation}
Thus, we have completed a full circle and arrived back at the notation
(\ref{eqn-functional}) that we started with.

\paragraph{Organization of the paper.}

In each of the remaining sections of this paper, we will consider a
particular class of categories and its associated graphical language.

\paragraph{Acknowledgments.}

I would like to thank Gheorghe {\Stefanescu} and Ross Street for their
help in locating hard-to-obtain references, and for providing some
background information. Thanks to Fabio Gadducci, Chris Heunen, and
Micah McCurdy for useful comments on an earlier draft.

\section{Categories}\label{sec-categories}

We only give the most basic definitions of categories, functors, and
natural transformations. For a gentler introduction, with more details
and examples, see e.g.~Mac~Lane {\cite{ML71}}.

\begin{definition}
  A {\em category} $\Cc$ consists of:
  \begin{itemize}
  \item a class $\obj{\Cc}$ of {\em objects}, denoted $A$, $B$, $C$,
    \ldots;
  \item for each pair of objects $A,B$, a set $\hom_{\Cc}(A,B)$ of
    {\em morphisms}, which are denoted $f:A\ii B$;
  \item {\em identity morphisms} $\id_A:A\ii A$ and the operation of
    {\em composition}: if $f:A\ii B$ and $g:B\ii C$, then 
    \[ g\cp f:A\ii C,
    \]
  \end{itemize}
  subject to the three equations
  \[ \id_B\cp f = f, \sep f\cp\id_A = f, \sep (h\cp g)\cp f = h\cp(g\cp f)
  \]
  for all $f:A\ii B$, $g:B\ii C$, and $h:C\ii D$.
\end{definition}

The terms ``map'' or ``arrow'' are often used interchangeably with
``morphism''.

\begin{examples}
  Some examples of categories are: the category $\Set$ of sets (with
  functions as the morphisms); the category $\Rel$ of sets (with
  relations as the morphisms); the category $\Vect$ of vector spaces
  (with linear maps); the category $\Hilb$ of Hilbert spaces (with bounded
  linear maps); the category $\UHilb$ of Hilbert spaces (with unitary
  maps); the category $\Top$ of topological spaces (with continuous
  maps); the category $\Cob$ of $n$-dimensional oriented manifolds
  (with oriented cobordisms).  Note that in each case, we need to
  specify not only the objects, but also the morphisms (and
  technically the composition and identities, although they are often
  clear from the context).
  
  Categories also arise in other sciences, for example in logic (where
  the objects are propositions and the morphisms are proofs), and in
  computing (where the objects are data types and the morphisms are
  programs).
\end{examples}

Many concepts associated with sets and functions, such as {\em
  inverse}, {\em monomorphism} (injective map), {\em idempotent}, {\em
  cartesian product}, etc., are definable in an arbitrary category.

\paragraph{Graphical language.}

In the graphical language of categories, objects are represented as
{\em wires} (also called {\em edges}) and morphisms are represented as
{\em boxes} (also called {\em nodes}).  An identity morphisms is
represented as a continuing wire, and composition is represented by
connecting the outgoing wire of one diagram to the incoming wire of
another. This is shown in Table~\ref{tab-graphical-cats}.

\begin{table}
  \begin{center}
    \begin{tabular}{@{}llc@{}}
      Object& $A$ &
      $\wirechart{@C=.7cm}{*{}\wireright{rr}{A}&&}$ \\\\
      Morphism &$f:A\ii B$ &
      $\wmetamorph{f}{A}{B}$\\\\
      Identity & $\id_A:A\ii A$ &
      $\wirechart{@C=.7cm}{*{}\wireright{rr}{A}&&}$ \\\\
      Composition & $t\cp s$ &
      $\wirechart{}{*{}\wireright{r}{A}&
        \wwblank{8mm}\ulbox{[]}{s}\wireright{r}{B}&
        \wwblank{8mm}\ulbox{[]}{t}\wireright{r}{C}&
        }$
    \end{tabular}
  \end{center}
  \caption{The graphical language of categories}
  \label{tab-graphical-cats}
\end{table}

\paragraph{Coherence.}

Note that the three defining axioms of categories (e.g., $\id_B\cp
f=f$) are automatically satisfied ``up isomorphism'' in the graphical
language. This property is known as {\em soundness}. A converse of
this statement is also true: every equation that holds in the
graphical language is a consequence of the axioms. This property is
called {\em completeness}. We refer to a soundness and completeness
theorem as a {\em coherence theorem}.

\begin{theorem}[Coherence for categories]\label{thm-coherence-categories}
  A well-formed equation between two morphism terms in the language of
  categories follows from the axioms of categories if and only if it
  holds in the graphical language up to isomorphism of diagrams.
\end{theorem}

Hopefully it is obvious what is meant by {\em isomorphism of
  diagrams}: two diagrams are isomorphic if the boxes and wires of the
first are in bijective correspondence with the boxes and wires of the
second, preserving the connections between boxes and wires.

Admittedly, the above coherence theorem for categories is a
triviality, and is not usually stated in this way. However, we have
included it for sake of uniformity, and for comparison with the less
trivial coherence theorems for monoidal categories in the following
sections.  The proof is straightforward, since by the associativity
and unit axioms, each morphism term is uniquely equivalent to a term
of the form $((f_n\cp \ldots)\cp f_2)\cp f_1$ for $n\geq 0$, with
corresponding diagram
\[
\wirechart{}{*{}  \wireright{r}{}&
  \blank\ulbox{[]}{f_1}  \wireright{r}{}&
  \blank\ulbox{[]}{f_2}  \wireright{r}{}&
  *+{\cdots}  \wireright{r}{}&
  \blank\ulbox{[]}{f_n}  \wireright{r}{}& *{.}
  }
\]

\begin{remark}
  We have equipped wires with a left-to-right arrow, and boxes with a
  marking in the upper left corner. These markings are of no use at
  the moment, but will become important as we extend the language in
  the following sections.
\end{remark}

\subsection*{Technicalities}

\paragraph{Signatures, variables, terms, and equations.}

So far, we have not been very precise about what the wires and boxes
of a diagram are labeled with. We have also glossed over what was
meant by ``a well-formed equation between morphism terms in the
language of categories''. We now briefly explain these notions,
without giving all the formal details. For a more precise mathematical
treatment, see e.g.~Joyal and Street {\cite{JS91}}.

The wires of a diagram are labeled with {\em object variables}, and
the boxes are labeled with {\em morphism variables}. To understand
what this means, consider the familiar language of arithmetic
expressions. This language deals with {\em terms}, such as
$(x+y+2)(x+3)$, which are built up from {\em variables}, such as $x$
and $y$, {\em constants}, such as $2$ and $3$, by means of {\em
  operations}, such as addition and multiplication. Variables can be
viewed in three different ways: first, they can be viewed as {\em
  symbols} that can be compared (e.g.~the variable $x$ occurs twice in
the given term, and is different from the variable $y$). They can also
be viewed as placeholders for arbitrary {\em numbers}, for example
$x=5$ and $y=15$. Here $x$ and $y$ are allowed to represent different
numbers or the same number; however, the two occurrences of $x$ must
denote the same number. Finally, variables can be viewed as
placeholders for arbitrary {\em terms}, such as $x=a+b$ and $y=z^2$.

The formal language of category theory is similar, except that we
require two sets of variables: object variables (for labeling wires)
and morphism variables (for labeling boxes). We must also equip each
morphism variable with a specified domain and codomain. The following
definition makes this more precise.

\begin{definition}
  A {\em simple (categorical) signature} $\Sigma$ consists of a set
  $\Sigma_0$ of {\em object variables}, a set $\Sigma_1$ of {\em
    morphism variables}, and a pair of functions
  $\dom,\cod:\Sigma_1\ii\Sigma_0$. Object variables are usually
  written $A,B,C,\ldots$, morphism variables are usually written
  $f,g,h,\ldots$, and we write $f:A\ii B$ if $\dom(f)=A$ and
  $\cod(f)=B$.
\end{definition}

Given a simple signature, we can then build {\em morphism terms},
such as $f\cp(g\cp \id_A)$, which are built from morphism variables
(such as $f$ and $g$) and morphism constants (such as $\id_A$), via
operations (i.e., composition). Each term is recursively equipped with
a domain and a codomain, and we must require compositions to respect
the domain and codomain information. A term that obeys these rules is
called {\em well-formed}. Finally, an equation between terms is called
a {\em well-formed equation} if the left-hand side and right-hand side
are well-formed terms that moreover have equal domains and equal
codomains.

The graphical language is also relative to a given signature. The
wires and boxes are labeled, respectively, with object variables and
morphism variables from the signature, and the labeling must respect
the domain and codomain information. This means that the wire entering
(respectively, exiting) a box labeled $f$ must be labeled by the
domain (respectively, codomain) of $f$.

The above remark about the different roles of variables in arithmetic
also holds for the diagrammatic language of categories. On the one
hand, the labels can be viewed as formal symbols. This is the view
used in the coherence theorem, where the formal labels are part of the
definition of equivalence (in this case, isomorphism) of diagrams.

The labels can also be viewed as placeholders for specific objects and
morphisms in an actual category. Such an assignment of objects and
morphisms is called an {\em interpretation} of the given signature.
More precisely, an interpretation $i$ of a signature $\Sigma$ in a
category $\Cc$ consists of a function $i_0:\Sigma_0\ii\obj{\Cc}$, and
for any $f\in\Sigma_1$ a morphism $i_1(f):i_0(\dom f)\ii i_0(\cod f)$.
By a slight abuse of notation, we write $i:\Sigma\ii\Cc$ for such an
interpretation.

Finally, a morphism variable can be viewed as a placeholder for an
arbitrary (possibly composite) diagram. We occasionally use this
latter view in schematic drawings, such as the schematic
representation of $t\cp s$ in Table~\ref{tab-graphical-cats}. We then
label a box with a morphism term, rather than a formal variable, and
understand the box as a short-hand notation for a possibly composite
diagram corresponding to that term.

\paragraph{Functors and natural transformations.}

\begin{definition}
  Let $\Cc$ and $\Dd$ be categories.  A {\em functor} $F:\Cc\ii\Dd$
  consists of a function $F:\obj{\Cc}\ii\obj{\Dd}$, and for each pair
  of objects $A,B\in\obj{\Cc}$, a function
  $F:\hom_{\Cc}(A,B)\ii\hom_{\Dd}(FA,FB)$, satisfying $F(g\cp
  f)=F(g)\cp F(f)$ and $F(\id_A)=\id_{FA}$.
\end{definition}

\begin{definition}
  Let $\Cc$ and $\Dd$ be categories, and let $F,G:\Cc\ii\Dd$ be
  functors. A {\em natural transformation} $\natt:F\ii G$ consists of
  a family of morphisms $\natt_A:FA\ii GA$, one for each object
  $A\in\obj{\Cc}$, such that the following diagram commutes for all
  $f:A\ii B$:
  \[\xymatrix{
    FA\ar[r]^{\natt_A} \ar[d]_{Ff} & GA\ar[d]^{Gf} \\
    FB\ar[r]^{\natt_B} & GB. \\
    }
  \]
\end{definition}

\paragraph{Coherence and free categories.}

Most coherence theorems are proved by characterizing the {\em free}
categories of a certain kind. 

\begin{definition}
  We say that a category $\Cc$ is {\em free} over a signature $\Sigma$
  if it is equipped with an interpretation $i:\Sigma\ii\Cc$, such that
  for any category $\Dd$ and interpretation $j:\Sigma\ii\Dd$, there
  exists a unique functor $F:\Cc\ii\Dd$ such that $j=F\cp i$.
\end{definition}

\begin{theorem}
  The graphical language of categories over a signature $\Sigma$, with
  identities and composition as defined in
  Table~\ref{tab-graphical-cats}, and up to isomorphism of diagrams,
  forms the free category over $\Sigma$.
\end{theorem}

Theorem~\ref{thm-coherence-categories} is indeed a consequence of this
theorem: by definition of freeness, an equation holds in all
categories if and only if it holds in the free category. By the
characterization of the free category, an equation holds in the free
category if and only if it holds in the graphical language.

\section{Monoidal categories}\label{sec-progressive}

In this section, we consider various notions of monoidal categories.
We sometimes refer to these notions as ``progressive'', which means
they have graphical languages where all arrows point left-to-right.
This serves to distinguish them from ``autonomous'' notions, which
will be discussed in Section~\ref{sec-autonomous}, and ``traced''
notions, which will be discussed in Section~\ref{sec-traced}.

\subsection{(Planar) monoidal categories}\label{subsec-planar-monoidal}

A {\em monoidal category} (also sometimes called {\em tensor
  category}) is a category with an associative unital tensor product.
More specifically:

\begin{definition}[\cite{ML71,JS93}]
  A {\em monoidal category} is a category with the following
  additional structure:
  \begin{itemize}
  \item a new operation $A\x B$ on objects and a new object constant $I$;
  \item a new operation on morphisms: if $f:A\ii C$ and $g:B\ii D$, then
    \[ f\x g:A\x B\ii C\x D;
    \]
  \item and isomorphisms
    \[ \begin{array}{ll}
      \alpha_{A,B,C}:&(A\x B)\x C\catarrow{\iso} A\x(B\x C),\\
      \lambda_A:&I\x A\catarrow{\iso} A,\\
      \rho_A:&A\x I\catarrow{\iso} A,\\
    \end{array}
    \]
  \end{itemize}
  subject to a number of equations:
  \begin{itemize}
  \item $\x$ is a bifunctor, which means $\id_A\x\id_B = \id_{A\x
      B}$ and $(k\x h)\cp(g\x f) = (k\cp g)\x(h\cp f)$;
  \item $\alpha$, $\lambda$, and $\rho$ are natural transformations,
    i.e., $(f\x (g\x h))\cp\alpha_{A,B,C} = \alpha_{A',B',C'}\cp((f\x
    g)\x h)$, $f\cp\lambda_{A} = \lambda_{A'}\cp(\id_I\x f)$, and
    $f\cp\rho_{A} = \rho_{A'}\cp(f\x\id_I)$;
  \item plus the following two coherence axioms, called the ``pentagon
    axiom'' and the ``triangle axiom'':
    \[\xymatrix@R-5mm@C-20mm{
      &(A\x (B\x C))\x D
      \ar[rr]\ar@{}@<1ex>[rr]^{\alpha_{A,B\x C,D}}&&
      A\x((B\x C)\x D)
      \ar[dr]^<>(.9){A\x\alpha_{B,C,D}}
      \\
      ((A\x B)\x C)\x D
      \ar[ur]^<>(.1){\alpha_{A,B,C}\x D}
      \ar[drr]_<>(.3){\alpha_{A\x B,C,D}}&&&&
      A\x(B\x (C\x D))\\
      &&(A\x B)\x (C\x D)
      \ar[rru]_<>(.7){\alpha_{A,B,C\x D}}
      \\
      }
    \]
    \[\xymatrix@R-5mm@C-10mm{
      (A\x I)\x B\ar[dr]_{\rho_A\x\id_B} \ar[rr]^{\alpha_{A,I,B}}
      && A\x (I\x B)\ar[dl]^{\id_A\x\lambda_B} \\
      & A\x B \\
      }
    \]
  \end{itemize}
\end{definition}

When we specifically want to emphasize that a monoidal category is not
assumed to be braided, symmetric, etc., we sometimes also refer to it
as a {\em planar monoidal category}.

\begin{examples}
  Examples of monoidal categories include: the category $\Set$ (of
  sets and functions), together with the cartesian product $\times$;
  the category $\Set$ together with the disjoint union operation $+$;
  the category $\Rel$ with either $\times$ or $+$; the category
  $\Vect$ (of vectors spaces and linear functions) with either
  $\oplus$ or $\x$; the category $\Hilb$ of Hilbert spaces with either
  $\oplus$ or $\x$; the categories $\Top$ and $\Cob$ with disjoint
  union $+$. Note that in each case, we need to specify a category and
  a tensor product (in general there are multiple choices).
  Technically, we should also specify associativity maps etc., but
  they are usually clear from the context.
\end{examples}

\paragraph{Graphical language.}

We extend the graphical language of categories as follows. A tensor
product of objects is represented by writing the corresponding wires
in parallel. The unit object is represented by zero wires. A morphism
variable $f:A_1 \x \ldots\x A_n \ii B_1\x \ldots\x B_m$ is represented
as a box with $n$ input wires and $m$ output wires. A tensor product
of morphisms is represented by stacking the corresponding diagrams.
This is shown in Table~\ref{tab-graphical-monoidal}.

\begin{table}
  \begin{center}
    \begin{tabular}{@{}llc@{}}
      Tensor product& $S\x T$ &
      $\vcenter{\wirechart{@C=1.5cm@R=0.4cm}{\wireright{rr}{T}&&\\\wireright{rr}{S}&&}}$\\\\[-.5ex]
      Unit object& $I$ &
      (empty)\\\\[-.5ex]
      Morphism &$f:$ {\small $A_1\,{\x}\,\ldots\,{\x}\, A_n\,{\ii}\, B_1\,{\x}\,\ldots\,{\x}\, B_m$} &
      $\wirechart{@C=1.5cm@R=0.4cm}{
        *{}\wireright{r}{A_n}&
        \blank\wireright{r}{B_m}&
        \\
        *{}\wireright{r}{A_1}^<>(.8){\vdots}&
        \blank\ulbox{[].[u]}{f}\wireright{r}{B_1}^<>(.2){\vdots}&
        \\
        }$\\\\
      Tensor product &$s\x t$ &
      $\vcenter{\wirechart{@C=1.5cm@R=0.4cm}{
          \vsblank\wireright{r}{C}&\blank\ulbox{[]}{t}\wireright{r}{D}&\\
          \vsblank\wireright{r}{A}&\blank\ulbox{[]}{s}\wireright{r}{B}&\\
          }}$
    \end{tabular}
  \end{center}
  \caption{The graphical language of monoidal categories}
  \label{tab-graphical-monoidal}
\end{table}

Note that it is our convention to write tensor products in the
bottom-to-top order. Similar conventions apply to objects as to
morphisms: thus, a single wire is labeled by an {\em object variable}
such as $A$, while a more general object such as $A\x B$ or $I$ is
represented by zero or more wires. For more details, see ``Monoidal
signatures'' below.

\paragraph{Coherence.}

It is easy to check that the graphical language for monoidal
categories is sound, up to deformation of diagrams in the plane.  As
an example, consider the following law, which is a consequence of
bifunctoriality:
\[ (\id_C\x g)\cp(f\x\id_B) =
(f\x\id_D)\cp(\id_A\x g).
\]
Translated into the graphical language, this becomes
\[\vcenter{\wirechart{@C-2mm@R6mm}{
    *{}\wireright{rr}{B}&&\blank\ulbox{[]}{g}\wireright{r}{D}&*{}\\
    *{}\wireright{r}{A}&\blank\ulbox{[]}{f}\wireright{rr}{C}&&*{}\\
    }}
\ssep=\ssep
\vcenter{\wirechart{@C-2mm@R7mm}{
    *{}\wireright{r}{B}&\blank\ulbox{[]}{g}\wireright{rr}{D}&&*{}\\
    *{}\wireright{rr}{A}&&\blank\ulbox{[]}{f}\wireright{r}{C,}&*{}\\
    }}
\]
which obviously holds up to deformation of diagrams. We have the
following coherence theorem:

\begin{theorem}[Coherence for planar monoidal categories {\cite[Thm.~1.5]{JS88}}, {\cite[Thm.~1.2]{JS91}}]\label{thm-coherence-planar}
  A well-formed equation between morphism terms in the language of
  monoidal categories follows from the axioms of monoidal categories
  if and only if it holds, up to planar isotopy, in the graphical
  language.
\end{theorem}

Here, by ``planar isotopy'', we mean that two diagrams, drawn in a
rectangle in the plane with incoming and outgoing wires attached to
the boundaries of the rectangle, are equivalent if it is possible to
transform one to the other by continuously moving around boxes in the
rectangle, without allowing boxes or wires to cross each other or to
be detached from the boundary of the rectangle during the moving.  To
make these notions mathematically precise, it is usually easier to
represent morphism as points, rather than boxes.  For precise
definitions and a proof of the coherence theorem, see
Joyal and Street {\cite{JS88,JS91}}.

\begin{caveat}
  Technically, Joyal and Street's proof in {\cite{JS88,JS91}} only
  applies to planar isotopies where each intermediate diagram during
  the deformation remains progressive, i.e., with all arrows oriented
  left-to-right.  Joyal and Street call such an isotopy ``recumbent''.
  We conjecture that the result remains true if one allows arbitrary
  planar deformations. Similar caveats also apply to the coherence
  theorems for braided and balanced monoidal categories below.
\end{caveat}

The following is an example of two diagrams that are not isomorphic in
the planar embedded sense:
\begin{equation}\label{eqn-non-spacial}
  \vcenter{\wirechart{@R+0.1cm}{
      \blank\wireright{rr}{B}&&\blank\\
      \ulbox{[u].[d]}{f} & \blank\ulbox{[]}{h} & \ulbox{[u].[d]}{g} \\
      \blank\wireright{rr}{A}&&\blank\\
      }}
  \sep\neq\sep
  \vcenter{\wirechart{@R+0.1cm}{
      \blank\wireright{rr}{B}&&\blank\\
      \blank\ulbox{[u].[]}{f}\wireright{rr}{A}&&\blank\ulbox{[u].[]}{g},\\
      & \blank\ulbox{[]}{h} & \\
      }}
\end{equation}
where $f:I\ii A\x B$, $g:A\x B\ii I$, and $h:I\ii I$. And indeed, the
corresponding equation $g\cp ((\rho_A\cp(\id_A\x
h)\cp\rho\inv_A)\x\id_B) \cp f = g\cp
((\lambda_A\cp(h\x\id_A)\cp\lambda\inv_A)\x\id_B) \cp f$ {\em does
  not follow} from the axioms of monoidal categories. This is an easy
consequence of soundness.

Note that because of the coherence theorem, it is not actually
necessary to memorize the axioms of monoidal categories: indeed, one
could use the coherence theorem as the {\em definition} of monoidal
category! For practical purposes, reasoning in the graphical language
is almost always easier than reasoning from the axioms. On the other
hand, the graphical definition is not very useful when one has to
check whether a given category is monoidal; in this case, checking
finitely many axioms is easier.

\paragraph{Relationship to traditional coherence theorems.}

Many category theorists are familiar with coherence theorems of the
form ``all diagrams of a certain type commute''. Mac~Lane's
traditional coherence theorem for monoidal categories {\cite{ML63}} is
of this form. It states that all diagrams built from only $\alpha$,
$\lambda$, $\rho$, $\id$, $\cp$, and $\x$ commute.
  
The coherence results in this paper are of a more general form (cf.
Kelly {\cite[p.~107]{K72b}}). Here, the object is to characterize {\em
  all} formal equations that follow from a given set of axioms. We
note that the traditional coherence theorem is an easy consequence of
the general coherence result of Theorem~\ref{thm-coherence-planar}:
namely, if a given well-formed equation is built only from $\alpha$,
$\lambda$, $\rho$, $\id$, $\cp$, and $\x$, then both the left-hand
side and right-hand side denote identity diagrams in the graphical
language.  Therefore, by Theorem~\ref{thm-coherence-planar}, the
equation follows from the axioms of monoidal categories. Analogous
remarks hold for all the coherence theorems of this article.

\subsection*{Technicalities}

\paragraph{Monoidal signatures.}

To be precise about the labels on diagrams of monoidal categories, and
about the meaning of ``well-formed equation'' in the coherence
theorem, we introduce the concept of a monoidal signature. This
generalizes the simple signatures introduced in
Section~\ref{sec-categories}. Monoidal signatures were introduced
under the name {\em tensor schemes} by Joyal and Street
{\cite{JS88,JS91}}. We give a non-strict version of the definition.

\begin{definition}[{\cite[Def.~1.4]{JS91}}, {\cite[Def.~1.6]{JS88}}]
  Given a set $\Sigma_0$ of {\em object variables}, let
  $\MonTerm(\Sigma_0)$ denote the free $(\x,I)$-algebra generated by
  $\Sigma_0$, i.e., the set of {\em object terms} built from object
  variables and $I$ via the operation $\x$. For example, if
  $A,B\in\Sigma_0$, then the term $(A\x B)\x(I\x A)$ is an element of
  $\MonTerm(\Sigma_0)$.
  
  A {\em monoidal signature} consists of a set $\Sigma_0$ of object
  variables, a set $\Sigma_1$ of {\em morphism variables}, and a pair
  of functions $\dom,\cod:\Sigma_1\ii\MonTerm(\Sigma_0)$.
\end{definition}

The concept of well-formed morphism terms and equations (in the
language of monoidal categories) is defined relative to a given
monoidal signature. In the graphical language, wires and boxes are
labeled by object variables and morphism variables as before. An
object term expands to zero or more parallel wires, by the rules of
Table~\ref{tab-graphical-monoidal}. As before, the labellings must
respect the domain and codomain information, which now involves
possibly multiple wires connected to a box. Just as we sometimes label
a box by a morphism term in schematic drawings to denote a possibly
composite diagram, we sometimes label a wire by an object term, such
as $S$ and $T$ in Table~\ref{tab-graphical-monoidal}. In this case, it
is a short-hand notation for zero or more parallel wires.

Given a monoidal signature $\Sigma$ and a monoidal category $\Cc$, an
{\em interpretation} $i:\Sigma\ii\Cc$ consists of an object function
$i_0:\Sigma_0\ii\obj{\Cc}$, which then extends in a unique way to
$\hat i_0: \MonTerm(\Sigma_0)\ii\obj{\Cc}$ such that $\hat i_0(A\x
B)=\hat i_0(A)\x\hat i_0(B)$ and $\hat i_0(I)=I$, and for any
$f\in\Sigma_1$ a morphism $i_1(f):i_0(\dom f)\ii i_0(\cod f)$.

The remaining graphical languages in this
Section~\ref{sec-progressive} are all given relative to a monoidal
signature.

\paragraph{Monoidal functors and natural transformations.}

\begin{definition}
  A {\em strong monoidal functor} (also sometimes called a {\em tensor
    functor}) between monoidal categories $\Cc$ and $\Dd$ is a functor
  $F:\Cc\ii\Dd$, together with natural isomorphisms $\phi^2:FA\x FB\ii
  F(A\x B)$ and $\phi^0: I\ii FI$, such that the following diagrams
  commute:
  \[ \xymatrix@C-2ex{
    (FA\x FB)\x FC \ar[rr]^{\phi^2\x \id} \ar[d]_{\alpha} &&
    F(A\x B)\x FC \ar[r]^{\phi^2} & 
    F((A\x B)\x C) \ar[d]^{F(\alpha)} \\
    FA\x (FB\x FC) \ar[rr]^{\id\x\phi^2} && 
    FA\x F(B\x C) \ar[r]^{\phi^2} & 
    F(A\x (B\x C)) \\
    }
  \]
  \[ \xymatrix{
    FA\x I \ar[r]^{\rho} \ar[d]_{\id\x\phi^0} & 
    FA \ar[d]^{F(\rho)} \\
    FA\x FI \ar[r]^{\phi^2} & 
    F(A\x I) 
    }
  \sep
  \xymatrix{
    I\x FA \ar[r]^{\lambda} \ar[d]_{\phi^0\x\id} &
    FA \ar[d]^{F(\lambda)} \\
    FI\x FA \ar[r]^{\phi^2} &
    F(I\x A)
    }
  \]
\end{definition}

\begin{definition}
  Let $\Cc$ and $\Dd$ be monoidal categories, and let $F,G:\Cc\ii\Dd$
  be strong monoidal functors. A natural transformation $\natt:F\ii
  G$ is called {\em monoidal} (or a {\em tensor transformation}) if
  the following two diagrams commute for all $A,B$:
  \[ \xymatrix{
    FA\x FB \ar[r]^{\phi^2} \ar[d]_{\natt_A\x\natt_B} & 
    F(A\x B) \ar[d]^{\natt_{A\x B}} \\
    GA\x GB \ar[r]^{\phi^2} & 
    G(A\x B)
    }
  \]
\end{definition}

\paragraph{Coherence and free monoidal categories.}

Similarly to what we stated for categories, the coherence theorem for
monoidal categories is a consequence of a characterization of the free
monoidal category. However, due to the extra coherence conditions in
the definition of a strong monoidal functor, the definition of
freeness is slightly more complicated. 

\begin{definition}
  A monoidal category $\Cc$ is a {\em free monoidal category} over a
  monoidal signature $\Sigma$ if it is equipped with an interpretation
  $i:\Sigma\ii\Cc$ such that for any monoidal category $\Dd$ and
  interpretation $j:\Sigma\ii\Dd$, there exists a strong monoidal
  functor $F:\Cc\ii\Dd$ such that $j=F\cp i$, and $F$ is unique up to
  a unique monoidal natural isomorphism.
\end{definition}

As before, the coherence theorem can be re-formulated as a freeness
theorem.

\begin{theorem}
  The graphical language of monoidal categories over a monoidal
  signature $\Sigma$, with identities, composition, and tensor as
  defined in Tables~\ref{tab-graphical-cats} and
  {\ref{tab-graphical-monoidal}}, and up to planar isotopy of
  diagrams, forms a free monoidal category over $\Sigma$.
\end{theorem}

Most of the coherence theorems (and conjectures) of this article can
be similarly formulated in terms of freeness. An exception to
this are the traced categories without braidings in
Sections~\ref{subsec-right-traced}--\ref{subsec-spacial-traced} and
{\ref{subsec-dagger-traced}}, as explained in
Remark~\ref{rem-planar-not-free}. From now on, we will only mention
freeness when it is not entirely automatic, such as in
Section~\ref{subsec-planar-autonomous}.

\subsection{Spacial monoidal categories}

\begin{definition}
  A monoidal category is {\em spacial} if it satisfies the additional
  axiom
  \begin{equation}\label{eqn-spacial}
    \rho_A\cp(\id_A\x h)\cp\rho\inv_A=\lambda_A\cp(h\x\id_A)\cp
    \lambda\inv_A,
  \end{equation}
  for all $h:I\ii I$.
\end{definition}

In the graphical language, this means that
\[
  \vcenter{\wirechart{@R+0.3cm}{
      \blank \\
      \blank & \blank\ulbox{[]}{h} & \blank \\
      \blank\wireright{rr}{A}&&\blank\\
      }}
  =
  \vcenter{\wirechart{@R+0.3cm}{
      \blank\wireright{rr}{A}&&\blank,\\
      \blank & \blank\ulbox{[]}{h} & \blank \\
      \blank \\
      }},
\]
so in particular, it implies that the two terms in
(\ref{eqn-non-spacial}) are equal. The author does not know whether
the concept of a spacial monoidal category appears in the literature,
or if it does, under what name.

\paragraph{Graphical language.}

The graphical language for spacial monoidal categories is the same as
that for monoidal categories, except that planarity is dropped from
the notion of diagram equivalence, i.e., diagrams are considered up to
isomorphism. Obviously the axioms are sound; we conjecture that they
are also complete.

\begin{conjecture}[Coherence for spacial monoidal categories]
  A well-formed equation between morphism terms in the language of
  spacial monoidal categories follows from the axioms of spacial
  monoidal categories if and only if it holds, up to isomorphism of
  diagrams, in the graphical language.
\end{conjecture}

Note that, in the case of planar diagrams, the notion of isomorphism
of diagrams coincides with ambient isotopy in 3 dimensions. This
explains the term ``spacial''.

\subsection{Braided monoidal categories}

\begin{definition}[\cite{JS93}]
  A {\em braiding} on a monoidal category is a natural family of
  isomorphisms $\sym_{A,B}:A\x B\ii B\x A$, satisfying the
  following two ``hexagon axioms'':
  \[\xymatrix@R-3mm@C-5mm{
    &(B\x A)\x C\ar[rr]^{\alpha_{B,A,C}}&&
    B\x(A\x C) \ar[dr]^{\id_B\x\sym_{A,C}}\\
    (A\x B)\x C
    \ar[ur]^{\sym_{A,B}\x\id_C}
    \ar[dr]_{\alpha_{A,B,C}}&&&&
    B\x(C\x A).\\
    &A\x(B\x C)\ar[rr]^{\sym_{A,B\x C}}&&
    (B\x C)\x A \ar[ur]_{\alpha_{B,C,A}}\\
    }
  \]
  \[\xymatrix@R-3mm@C-5mm{
    &(B\x A)\x C\ar[rr]^{\alpha_{B,A,C}}&&
    B\x(A\x C) \ar[dr]^{\id_B\x\symi_{C,A}}\\
    (A\x B)\x C
    \ar[ur]^{\symi_{B,A}\x\id_C}
    \ar[dr]_{\alpha_{A,B,C}}&&&&
    B\x(C\x A).\\
    &A\x(B\x C)\ar[rr]^{\symi_{B\x C,A}}&&
    (B\x C)\x A \ar[ur]_{\alpha_{B,C,A}}\\
    }
  \]
\end{definition}

Note that every braided monoidal category is spacial; this follows
from the naturality (in $I$) of $\sym_{A,I}:A\x I\ii I\x A$.

A braided monoidal functor between braided monoidal categories is a
monoidal functor that is compatible with the braiding in the following
sense:
\[ \xymatrix{
  FA\x FB \ar[r]^{\phi^2}\ar[d]_{\sym_{FA,FB}} &
  F(A\x B)\ar[d]^{F\sym_{A,B}} \\
  FB\x FA \ar[r]^{\phi^2} &
  F(B\x A).
}
\]

\paragraph{Graphical language.}

One extends the graphical language of monoidal categories with the {\em
  braiding}:

\begin{center}
  \begin{tabular}{@{}llc@{}}
    Braiding & $\sym_{A,B}$ &
    $\vcenter{\wirechart{@C=1.5cm@R=0.8cm}{
        *{}\wireright{r}{B}&\blank\wirecross{d}\wireright{r}{A}&\\
        *{}\wireright{r}{A}&\blank\wirebraid{u}{.3}\wireright{r}{B}&
        }}$ \\
  \end{tabular}
\end{center}

In general, if $A$ and $B$ are composite object terms, the braiding
$\sym_{A,B}$ is represented as the appropriate number of wires
crossing each other.

Note that the braiding satisfies
$\sym_{A,B}\cp\symi_{A,B}=\id_{A\x B}$, but not
$\sym_{A,B}\cp\sym_{B,A}=\id_{A\x B}$. Graphically:

\[\vcenter{\wirechart{@C=1.5cm@R=0.8cm}{
    *{}\wireright{r}{B}&
    \blank\wirecross{d}\wireright{r}{A}&
    \blank\wirebraid{d}{.3}\wireright{r}{B}&
    \\
    *{}\wireright{r}{A}&
    \blank\wirebraid{u}{.3}\wireright{r}{B}&
    \blank\wirecross{u}\wireright{r}{A}&
    \\
    }}
= \id_{A\x B},
\]
\[\vcenter{\wirechart{@C=1.5cm@R=0.8cm}{
    *{}\wireright{r}{B}&
    \blank\wirecross{d}\wireright{r}{A}&
    \blank\wirecross{d}\wireright{r}{B}&
    \\
    *{}\wireright{r}{A}&
    \blank\wirebraid{u}{.3}\wireright{r}{B}&
    \blank\wirebraid{u}{.3}\wireright{r}{A}&
    \\
    }}
\neq \id_{A\x B}.
\]

\begin{example}
  The hexagon axiom translates into the following in the graphical
  language:

  \[ (\id_B\x\sym_{A,C}) \cp \alpha_{B,A,C} \cp (\sym_{A,B}\x\id_C)
  = \alpha_{B,C,A} \cp (\sym_{B,C\x A}) \cp \alpha_{A,B,C}
  \]
  \[\vcenter{\wirechart{@R=6mm}{
    \wireright{r}{C}&
    \wireid\wireright{r}{C}&
    \blank\wirecross{d}{.3}\wireright{r}{A}&
    \\
    \wireright{r}{B}&
    \blank\wirecross{d}{.3}\wireright{r}{A}&
    \blank\wirebraid{u}{.3}\wireright{r}{C}&
    \\
    \wireright{r}{A}&
    \blank\wirebraid{u}{.3}\wireright{r}{B}&
    \wireid\wireright{r}{B}&
    \\
    }}
  \ssep=\ssep
  \vcenter{\wirechart{@R=6mm}{
    \wireright{r}{C}&
    \blank\wirecross{d}{.3}\wireright{r}{A}&
    \\
    \wireright{r}{B}&
    \blank\wirecross{d}{.3}\wireright{r}{C}&
    \\
    \wireright{r}{A}&
    \blank\wirebraid{uu}{.2}\wireright{r}{B}&
    \\
    }}
  \]
\end{example}

\begin{example}
  The {\em Yang-Baxter equation} is the following equation, which is a
  consequence of the hexagon axiom and naturality:
  \[ (\sym_{B,C}\x\id_A)\cp(\id_B\x\sym_{A,C})\cp(\sym_{A,B}\x\id_C) =
  (\id_C\x\sym_{A,B})\cp(\sym_{A,C}\x\id_B)\cp(\id_A\x\sym_{B,C}).
  \]
  In the graphical language, it becomes:
  \[\vcenter{\wirechart{@C-4mm@R=6mm}{
    \wireright{r}{C}&
    \wireid\wireright{r}{C}&
    \blank\wirecross{d}{.3}\wireright{r}{A}&
    \wireid\wireright{r}{A}&
    \\
    \wireright{r}{B}&
    \blank\wirecross{d}{.3}\wireright{r}{A}&
    \blank\wirebraid{u}{.3}\wireright{r}{C}&
    \blank\wirecross{d}{.3}\wireright{r}{B}&
    \\
    \wireright{r}{A}&
    \blank\wirebraid{u}{.3}\wireright{r}{B}&
    \wireid\wireright{r}{B}&
    \blank\wirebraid{u}{.3}\wireright{r}{C}&
    \\
    }}
  \ssep=\ssep
  \vcenter{\wirechart{@C-4mm@R=6mm}{
    \wireright{r}{C}&
    \blank\wirecross{d}{.3}\wireright{r}{B}&
    \wireid\wireright{r}{B}&
    \blank\wirecross{d}{.3}\wireright{r}{A}&
    \\
    \wireright{r}{B}&
    \blank\wirebraid{u}{.3}\wireright{r}{C}&
    \blank\wirecross{d}{.3}\wireright{r}{A}&
    \blank\wirebraid{u}{.3}\wireright{r}{B}&
    \\
    \wireright{r}{A}&
    \wireid\wireright{r}{A}&
    \blank\wirebraid{u}{.3}\wireright{r}{C}&
    \wireid\wireright{r}{C}&
    \\
    }}
  \]
\end{example}

\begin{theorem}[Coherence for braided monoidal categories {\cite[Thm.~3.7]{JS91}}]
  A well-formed equation between morphisms in the language of braided
  monoidal categories follows from the axioms of braided monoidal
  categories if and only if it holds in the graphical language up to
  isotopy in 3 dimensions.
\end{theorem}

Here, by ``isotopy in 3 dimensions'', we mean that two diagrams, drawn
in a 3-dimensional box with incoming and outgoing wires attached to
the boundaries of the box, are isotopic if it is possible to transform
one to the other by moving around nodes in the box, without allowing
nodes or edges to cross each other or to be detached from the boundary
during the moving. Also, the linear order of the edges entering and
exiting each node must be respected. This is made more precise in
Joyal and Street {\cite{JS91}}.

\begin{caveat}
  The proof by Joyal and Street {\cite{JS91}} is subject to some minor
  technical assumptions: graphs are assumed to be {\em smooth}, and
  the isotopies are progressive, with continuously changing tangent
  vectors.
\end{caveat}

\subsection{Balanced monoidal categories}\label{subsec-balanced}

\begin{definition}[\cite{JS93}]
  A {\em twist} on a braided monoidal category is a natural family of
  isomorphisms $\theta_A:A\ii A$, satisfying $\theta_I=\id_I$ and such
  that the following diagram commutes for all $A,B$:
  \begin{equation}\label{eqn-balanced}
    \xymatrix{
      A\x B \ar[d]_{\theta_{A\x B}}\ar[r]^{\sym_{A,B}} & B\x A \ar[d]^{\theta_B\x\theta_A} \\
      A\x B & B\x A. \ar[l]^{\sym_{B,A}}
      }
  \end{equation}
  A {\em balanced monoidal category} is a braided monoidal category
  with twist.
\end{definition}

A balanced monoidal functor between balanced monoidal categories is a
braided monoidal functor that is also compatible with the twist, i.e.,
such that $F(\theta_A)=\theta_{FA}$ for all $A$. 

\paragraph{Graphical language.}

The graphical language of balanced monoidal categories is similar to
that of braided monoidal categories, except that morphisms are
represented by flat ribbons, rather than 1-dimensional wires.  A
ribbon can be thought of as a pair of parallel wires that are
infinitesimally close to each other, or as a wire that is equipped
with a {\em framing} {\cite{JS91}}. For example, the braiding looks
like this:
\[ \sym_{A,B} =
\mnew{\resizebox{2cm}{!}{\includegraphics{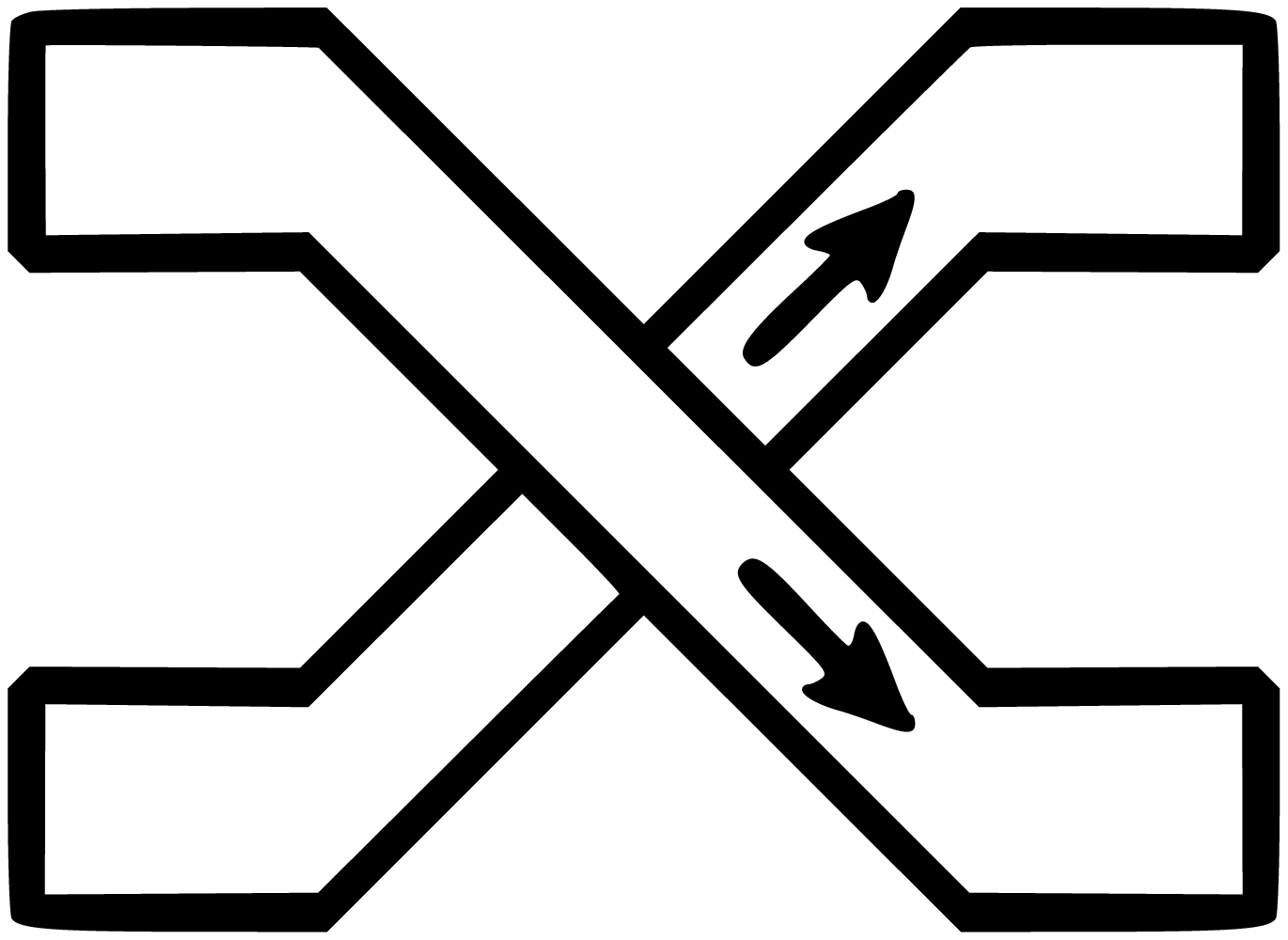}}}.
\]
The twist map $\theta_A$ is represented as a 360-degree twist in a
ribbon, or in several ribbons together, if $A$ is a composite object
term. This is easiest seen in the following illustration.
\[ \theta_A =
\resizebox{3cm}{!}{\includegraphics{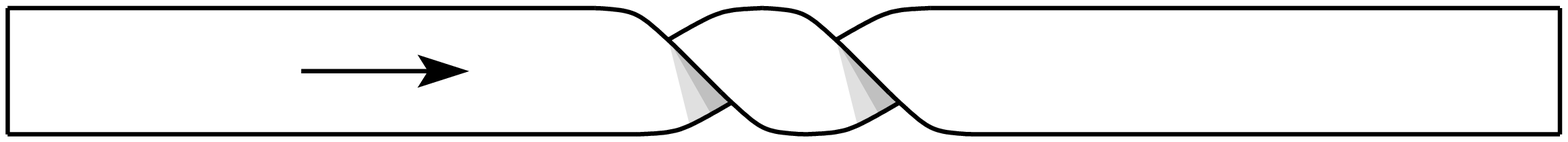}}~,
\sep
\theta_{A\x B} =
\mnew{\resizebox{4.5cm}{!}{\includegraphics{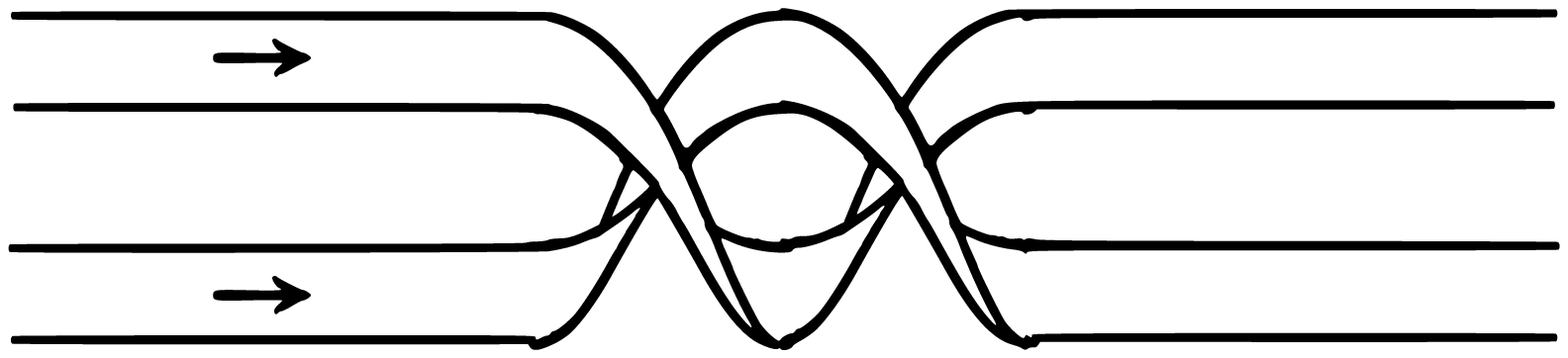}}}.
\]
The meaning of (\ref{eqn-balanced}) should then be obvious.

\begin{theorem}[Coherence for balanced monoidal categories {\cite[Thm.~4.5]{JS91}}]
  A well-formed equation between morphisms in the language of balanced
  monoidal categories follows from the axioms of balanced monoidal
  categories if and only if it holds in the graphical language up to
  framed isotopy in 3 dimensions.
\end{theorem}

\subsection{Symmetric monoidal categories}\label{subsec-symmetric-monoidal}

\begin{definition}
  A {\em symmetric monoidal category} is a braided monoidal category
  where the braiding is self-inverse, i.e.:
  \[ \sym_{A,B} = \symi_{B,A}
  \]
  In this case, the braiding is called a {\em symmetry}.
\end{definition}

\begin{remark}\label{rem-symmetric-from-balanced}
  Because of equation (\ref{eqn-balanced}), a symmetric monoidal
  category can be equivalently defined as a balanced monoidal category
  in which $\theta_A=\id_A$ for all $A$.
\end{remark}

\begin{remark}\label{rem-twisted-symmetric}
  The previous remark notwithstanding, there exist symmetric monoidal
  categories that possess a non-trivial twist (in addition to the
  trivial twist $\theta_A=\id_A$). Thus, in a balanced monoidal
  category, the symmetry condition $\sym_{A,B} = \symi_{B,A}$ does not
  in general imply $\theta_A=\id_A$. In other words, a balanced
  monoidal category that is symmetric as a braided monoidal category
  is not necessarily symmetric as a balanced monoidal category.  An
  example is the category of finite dimensional vector spaces and
  linear bijections, with $\theta_A(x)=nx$, where $n=\dim(A)$.
\end{remark}

\begin{examples}
  On the monoidal category $(\Set,\times)$ of sets with cartesian
  product, a symmetry is given by $\sym(x,y)=(y,x)$. On the category
  $(\Vect,\x)$ of vector spaces with tensor product, a symmetry is
  given by $\sym(x\x y)=y\x x$.
\end{examples}

\paragraph{Graphical language.}

The symmetry is graphically represented by a crossing:

\begin{center}
  \begin{tabular}{@{}llc@{}}
    Symmetry & $\sym_{A,B}$ &
    $\vcenter{\wirechart{@C=1.5cm@R=0.8cm}{
        *{}\wireright{r}{B}&\blank\wirecross{d}\wireright{r}{A}&\\
        *{}\wireright{r}{A}&\blank\wirecross{u}\wireright{r}{B}& }}$ \\
  \end{tabular}
\end{center}

\begin{theorem}[Coherence for symmetric monoidal categories {\cite[Thm.~2.3]{JS91}}]
  \label{thm-coherence-symmetric-monoidal}
  A well-formed equation between morphisms in the language of
  symmetric monoidal categories follows from the axioms of symmetric
  monoidal categories if and only if it holds, up to isomorphism of
  diagrams, in the graphical language.
\end{theorem}

Note that the graphical language for symmetric monoidal categories is
up to isomorphism of diagrams, without any reference to 2- or
3-dimensional structure. However, isomorphism of diagrams is
equivalent to ambient isotopy in 4 dimensions, so we can still regard
it as a geometric notion.

\section{Autonomous categories}\label{sec-autonomous}

Autonomous categories are monoidal categories in which the objects
have {\em duals}. In terms of graphical language, this means that some
wires are allowed to run from right to left.

\subsection{(Planar) autonomous categories}\label{subsec-planar-autonomous}

\begin{definition}[\cite{JS93}]
  In a (without loss of generality strict) monoidal category, an {\em
    exact pairing} between two objects $A$ and $B$ is given by a pair
  of morphisms $\eta:I\ii B\x A$ and $\eps:A\x B\ii I$, such that the
  following two adjunction triangles commute:
  \begin{equation}\label{eqn-pairing}
    \vcenter{\xymatrix{
      A \ar[r]^<>(.5){\id_A\x\eta} \ar[dr]_{\id_A} &
      A\x B\x A \ar[d]^{\eps\x\id_A} \\
      & A,
      }}
    \hspace{.7in}
    \vcenter{\xymatrix{
      B \ar[r]^<>(.5){\eta\x\id_{B}} \ar[dr]_{\id_{B}} &
      B\x A\x B \ar[d]^{\id_{B}\x\eps} \\
      & B.
      }}
  \end{equation}
  In such an exact pairing, $B$ is called the {\em right dual} of $A$
  and $A$ is called the {\em left dual} of $B$.
\end{definition}

\begin{remark}
  The maps $\eta$ and $\eps$ determine each other uniquely, and they
  are respectively called the {\em unit} and the {\em counit} of the
  adjunction. Moreover, the triple $(B,\eta,\eps)$, if it exists, is
  uniquely determined by $A$ up to isomorphism. The existence of duals
  is therefore a property of a monoidal category, rather than an
  additional structure on it. Moreover, every strong monoidal functor
  automatically preserves existing duals.
\end{remark}

\begin{definition}[\cite{JS86,JS88,JS93}]
  A monoidal category is {\em right autonomous} if every object $A$
  has a right dual, which we then denote $\rA$. It is {\em left
    autonomous} if every object $A$ has a left dual, which we then
  denote $\lA$. Finally, the category is {\em autonomous} if it is
  both right and left autonomous.
\end{definition}

\begin{remark}[Terminology]
  A [right, left, --] autonomous category is also called [right, left,
  --] rigid, see e.g.~{\cite[p.~78]{SR72}}. Also, the term
  ``autonomous'' is sometimes used in the weaker sense of ``monoidal
  closed''.  Although this latter usage is no longer common, it still
  lives on in the terminology ``*-autonomous category'' (Barr
  {\cite{Bar79}}, see also Section~\ref{sec-star-autonomous}).

  If we wish to emphasize that an autonomous category is not
  necessarily symmetric or braided, we sometimes call it a {\em planar
    autonomous category}.
\end{remark}

\paragraph{Graphical language.}

If $A$ is an object variable, the objects $\rA$ and $\lA$ are both
represented in the same way: by a wire labeled $A$ running from right
to left. The unit and counit are represented as half turns:
\begin{center}
  \begin{tabular}{@{}ll@{}cl@{}c@{}}
    Dual& $\rA, \lA$ &
    $\wirechart{@C=1cm}{*{}\wireleft{rr}{A}&&}$ \\\\
    Unit&
    $\eta_A:I\ii \rA\x A$ &
    $\vcenter{\wirechart{@C=1.5cm@R=0.5cm}{
        \blank\wireopen{d}\wireright{r}{A}&\\
        \blank\wireleft{r}{A}&
        }}
    $ &
    $\eta'_A:I\ii A\x \lA$ &
    $\vcenter{\wirechart{@C=1.5cm@R=0.5cm}{
        \blank\wireopen{d}\wireleft{r}{A}&\\
        \blank\wireright{r}{A}&
        }}
    $\\\\
    Counit&
    $\eps_A:A\x \rA\ii I$ &
    $\vcenter{\wirechart{@C=1.5cm@R=0.5cm}{
        \wireleft{r}{A}&\blank\wireclose{d}\\
        \wireright{r}{A}&\blank
        }}
    $ &
    $\eps'_A:\lA\x A\ii I$ &
    $\vcenter{\wirechart{@C=1.5cm@R=0.5cm}{
        \wireright{r}{A}&\blank\wireclose{d}\\
        \wireleft{r}{A}&\blank
        }}
    $\\\\
  \end{tabular}
\end{center}

More generally, if $A$ is a composite object represented by a number
of wires, then $\rA$ and $\lA$ are represented by the same set of
wires running backward (rotated by 180 degrees), and the units and
counits are represented as multiple wires turning.

\begin{example}
  The two diagrams in (\ref{eqn-pairing}), where $B=\rA$, translate
  into the graphical language as follows:
  \[
  \vcenter{\wirechart{@C=1.0cm@R=0.5cm}{
      \blank\wireopen{d}\wireright{r}{A}&\\
      \blank\wireleft{r}{A}&\blank\\
      \wireright{r}{A}&\blank\wireclose{u}
      }}
  = \ssep
  {\wirechart{@C=1.0cm@R=0.5cm}{
      \wireid\wireright{r}{A}&\wireid
      }},
  \hspace{.5cm}
  \vcenter{\wirechart{@C=1.0cm@R=0.5cm}{
      \wireleft{r}{A}&\blank\wireclose{d}\\
      \blank\wireright{r}{A}&\blank\\
      \blank\wireopen{u}\wireleft{r}{A}&
      }}
  = \ssep
  {\wirechart{@C=1.0cm@R=0.5cm}{
      \wireid\wireleft{r}{A}&\wireid
      }}.
  \]
\end{example}

\begin{example}
  For any morphism $f:A\ii B$, it is possible to define morphisms
  $\rd{f}:\rd{B}\ii\rd{A}$ and $\ld{f}:\ld{B}\ii\ld{A}$, called the
  {\em adjoint mates} of $f$, as follows:
  \[ \rd{f} \ssep =
  \vcenter{\wirechart{@C=0.5cm@R=0.5cm}{
      \vsblank\wireleft{r}{B}&\wireid\wire{r}{}&\blank\wireclose{d}\\
      \blank\wireright{r}{A}&\blank\ulbox{[]}{f}\wireright{r}{B}&\blank\\
      \blank\wireopen{u}\wire{r}{}&\wireid\wireleft{r}{A}&\vsblank
      }}
  \sep
  \ld{f} \ssep =
  \vcenter{\wirechart{@C=0.5cm@R=0.5cm}{
      \blank\wireopen{d}\wire{r}{}&\wireid\wireleft{r}{A}&\vsblank\\
      \blank\wireright{r}{A}&\blank\ulbox{[]}{f}\wireright{r}{B}&\blank\\
      \vsblank\wireleft{r}{B}&\wireid\wire{r}{}&\blank\wireclose{u}
      }}
  \]
  With these definitions, $\rd{(-)}$ and $\ld{(-)}$ become
  contravariant functors.
\end{example}

\begin{theorem}[Coherence for planar autonomous categories {\cite[Thm.~2.7]{JS88}}]
  \label{thm-coherence-planar-autonomous}
  A well-formed equation between morphisms in the language of
  autonomous categories follows from the axioms of autonomous
  categories if and only if it holds in the graphical language up to
  planar isotopy.
\end{theorem}

Here, the notion of planar isotopy is the same as before, except that
the wires are of course no longer restricted to being oriented
left-to-right during the deformation. However, the ability to turn
wires upside down does not extend to boxes: the notion of isotopy for
this theorem does not include the ability to rotate boxes. See
Joyal and Street {\cite{JS88}} for a more precise statement.

\begin{caveat}
  The proof by Joyal and Street {\cite{JS88}} assumes that the
  diagrams are piecewise linear.
\end{caveat}

Note that the same theorem applies to left autonomous, right
autonomous, or autonomous categories. Indeed, each individual term in
the language of autonomous categories involves only finitely many
duals, and thus may be translated into a term of (say) left
autonomous categories by replacing each object variable $A$ by
$A^{***\ldots*}$, for a sufficiently large, even number of $*$'s.  The
resulting term maps to the same diagram.

The same coherence theorem also holds for categories that are only
right (or left) autonomous. This is a consequence of the following
proposition.

\begin{proposition}\label{prop-right-autonomous}
  Each right (or left) autonomous category can be fully embedded in an
  autonomous category.
\end{proposition}

\begin{proof}
  Let $\Cc$ be a right autonomous category, and consider the strong
  monoidal functor $F:\Cc\ii\Cc$ given by $F(A)=A^{**}$. This functor
  is full and faithful, and every object in the image of $F$ has a
  left dual. Now let $\hat\Cc$ be the colimit (in the large category
  of right autonomous categories and strong monoidal functors) of the
  sequence
  \[ \Cc\catarrow{F}\Cc\catarrow{F}\Cc\catarrow{F}\ldots
  \]
  Then $\hat\Cc$ is autonomous, and $\Cc$ is fully and faithfully
  embedded in $\hat\Cc$. The proof for left autonomous categories is
  analogous. \eot
\end{proof}

\begin{corollary}[Coherence for right (left) autonomous categories]
  A well-formed equation between morphisms in the language of right
  (left) autonomous categories follows from the axioms of right (left)
  autonomous categories if and only if it holds in the graphical
  language up to planar isotopy.
\end{corollary}

\begin{proof}
  It suffices to show that an equation (in the language of right
  autonomous categories) holds in all right autonomous categories if
  and only if it holds in all autonomous categories. The ``only if''
  direction is trivial, since every autonomous category is right
  autonomous. For the opposite direction, suppose some equation holds
  in all autonomous categories, and let $\Cc$ be a right autonomous
  category. Then $\Cc$ can be faithfully embedded in an autonomous
  category $\hat\Cc$. By assumption, the equation holds in $\hat\Cc$,
  and therefore also in $\Cc$, since the embedding is faithful.\eot
\end{proof}

\subsection*{Technicalities}

\paragraph{Autonomous signatures.}

The diagrams of autonomous categories, and the concept of well-formed
equation in the coherence theorem, are defined relative to the notion
of an autonomous signature. These were called {\em autonomous tensor
  schemes} by Joyal and Street {\cite{JS88}}. We give a non-strict
version of the definition.

\begin{definition}{\cite[Def.~2.5]{JS88}}
  Given a set $\Sigma_0$ of {\em object variables}, let
  $\AutTerm(\Sigma_0)$ denote the free
  $(\x,I,\ld{(-)},\rd{(-)})$-algebra generated by $\Sigma_0$, i.e.,
  the set of {\em object terms} built from object variables and $I$
  via the operations $\x$, $\ld{(-)}$, and $\rd{(-)})$. For example,
  if $A,B\in\Sigma_0$, then the term $\rd{B}\x\rd{({}^{**}I\x A)}$
  is an element of $\AutTerm(\Sigma_0)$.
  
  An {\em autonomous signature} consists of a set $\Sigma_0$ of object
  variables, a set $\Sigma_1$ of {\em morphism variables}, and a pair
  of functions $\dom,\cod:\Sigma_1\ii\AutTerm(\Sigma_0)$.
\end{definition}

The concept of a {\em right autonomous signature} and {\em left
  autonomous signature} are defined analogously.  The remaining
graphical languages in this Section~\ref{sec-autonomous} are all given
relative to an autonomous signature.

\paragraph{Functors and natural transformations of autonomous categories.}

Any strong monoidal functor preserves exact pairings: if $\eta:I\ii
B\x A$ and $\eps:A\x B\ii I$ define an exact pairing, then so do
\[ \hat F\eta: 
  I \catarrow{\phi^0}
  FI \catarrow{F\eta} 
  F(B\x A) \catarrow{(\phi^2)\inv} 
  FB\x FA
\]
and
\[ \hat F\eps:
  FA\x FB \catarrow{\phi^2} 
  F(A\x B) \catarrow{F\eps}
  FI \catarrow{(\phi^0)\inv}
  I.
\]
In particular, if $\Cc$ and $\Dd$ are autonomous categories and
$F:\Cc\ii\Dd$ is a monoidal functor, by uniqueness of duals, there
will be a unique induced natural isomorphism $F(\rd{A})\iso \rd{(FA)}$
such that
\[ \mnew{\xymatrix{
    I\ar[dr]_<>(.5){\eta_{FA}}\ar[r]^<>(.5){\hat F\eta_A} & 
    F(\rd{A})\x FA\ar[d]^{{\iso}\x\id} \\
    & \rd{(FA)}\x FA 
    }}
\sep\mbox{and}\sep
\mnew{\xymatrix{
    FA\x F(\rd{A}) \ar[d]^{\id\x{\iso}} \ar[r]^<>(.5){\hat F\eps_A} &
    I, \\
    FA\x \rd{(FA)} \ar[ur]_<>(.5){\eps_{FA}} 
    }}
\]
and similarly for $F(\ld{A})\iso\ld{(FA)}$.

For natural transformations, we have the following lemma:
\begin{lemma}[Saavedra Rivano {\cite[Prop.~5.2.3]{SR72}}, see also {\cite[Prop.~7.1]{JS93}}]
  \label{lem-saavedra-rivano}
  Suppose $\natt:F\ii G$ is a monoidal natural transformation between
  strong monoidal functors $F,G:\Cc\ii\Dd$. If $A$ has a right dual
  $A^*$ in $\Cc$, then $\natt_{A^*}$ and $(\natt_A)^*$ are mutually
  inverse in $\Dd$ (up to the above canonical isomorphism), or more
  precisely:
  \[ \xymatrix{
    F(A^*)\ar[r]^{\natt_{A^*}} \ar[d]_{\iso} & 
    G(A^*) \ar[d]^{\iso} \\
    (FA)^* &
    (GA)^*\ar[l]^{(\natt_A)^*} & 
    }
  \]
  In particular, if $\Cc$ is autonomous, then any such monoidal
  natural transformation is invertible.
\end{lemma}

\paragraph{Coherence and free autonomous categories.}

The graphical language, as we have defined it above for autonomous
categories, is sufficient for the purposes of
Theorem~\ref{thm-coherence-planar-autonomous}. However, it does not
characterize the free autonomous category over an autonomous signature
as stated. For example, consider a signature with a single morphism
variable $f:A\ii A$. The problem is that there are clearly some
diagrams, such as
\begin{equation}\label{eqn-not-autonomous}
  \vcenter{\wirechart{@R=.5cm}{
      \blank\wireopen{d}\wireleft{rr}{} && \blank\wireclose{d} \\
      \blank\wireright{r}{A} & \blank\ulbox{[]}{f}\wireright{r}{A} & \blank,
      }
    }
\end{equation}
which are not translations of any well-formed term of autonomous
categories. Indeed, for this diagram to correspond to a well-formed
term, we would have to have e.g.~$f:A^{**}\ii A$ or $f:A\ii {}^{**}A$.

Joyal and Street {\cite{JS88}} characterize the free autonomous
category by equipping each edge with a winding number. Effectively,
the horizontal segments of edges are labeled with pairs $(A,n)$, where
$A$ is an object variables and $n$ is an integer winding number.
Left-to-right segments have even winding numbers, right-to-left
segments have odd winding numbers, and winding numbers increase by one
on counterclockwise turns, and decrease by one on clockwise turns. The
winding numbers on the input and output of each box, and on the global
inputs and outputs, are restricted to be consistent with the domain
and codomain information, where e.g.~$A^{**}$ corresponds to $(A,2)$,
and ${}^{***}B$ to $(B,-3)$. See {\cite{JS88}} for precise details.
Here is an example of a well-formed diagram of type $I\ii B^{**}\x A$,
where $g:I\ii A\x B$: 
\[ \m{\resizebox{1.0in}{!}{\includegraphics{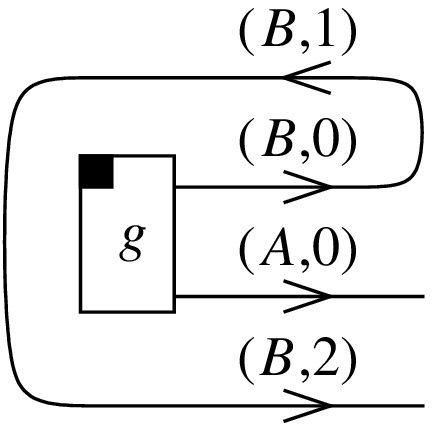}}}
\]

\begin{theorem}\label{thm-free-autonomous}
  The graphical language (with winding numbers) of autonomous
  categories over an autonomous signature $\Sigma$, up to planar
  isotopy of diagrams, forms a free autonomous category over $\Sigma$.
\end{theorem}

We remark that if a diagram of planar autonomous categories can be
labeled with winding numbers, then this labeling is necessarily
unique. In particular, for the purposes of
Theorem~\ref{thm-coherence-planar-autonomous}, there is no harm in
dropping the winding numbers, because by hypothesis, the theorem only
considers diagrams that are the translation of well-formed terms,
whose winding numbers can therefore uniquely reconstructed.

\subsection{(Planar) pivotal categories}

A pivotal category is an autonomous category with a suitable
isomorphism $A\iso A^{**}$.

\begin{definition}[\cite{FY89,FY92,JSXX}]
  A {\em pivotal category} is a right autonomous category equipped
  with a monoidal natural isomorphism $i_A:A\ii A^{**}$.
\end{definition}

Note that any pivotal category is immediately left autonomous,
therefore autonomous.  The requirement that $i_A$ is a {\em monoidal}
natural transformation here means that $i_I$ is the canonical
isomorphism $I\iso I^{**}$, and that the following diagram commutes,
where the horizontal arrow is the canonical isomorphism derived from
the autonomous structure:
\begin{equation}\label{eqn-pivotal-monoidal}
  \mnew{\xymatrix@C=0cm{
    &A\x B\ar[dl]_{i_A\x i_B}\ar[dr]^{i_{A\x B}}\\
    A^{**}\x B^{**}\ar[rr]^{\iso}&&(A\x B)^{**}.
    }}
\end{equation}

The following property, which is sometimes taken as part of the
definition of pivotal categories {\cite[Def.~3.1.1]{JSXX}}, is a
direct consequence of Saavedra Rivano's Lemma
(Lemma~\ref{lem-saavedra-rivano}).

\begin{lemma}
  In any pivotal category, the following diagram commutes:
  \[ \xymatrix{
    A^* \ar[r]^{i_{A^*}} \ar[dr]_{\id_{A^*}}&
    A^{***} \ar[d]^{i_A^*} \\
    & A^*.
    }
  \]
\end{lemma}

\begin{remark}
  One can equivalently define a pivotal category as an autonomous
  category equipped with a monoidal natural isomorphism (of
  contravariant monoidal functors) $\phi:\rd{A}\catarrow{\iso}\ld{A}$.
  This was done by Freyd and Yetter {\cite{FY92}}. Condition (S) of
  {\cite[Def.~4.1]{FY92}} is also a consequence of Saavedra Rivano's
  Lemma, and is therefore redundant.
\end{remark}

\begin{remark}[Terminology]
  Freyd and Yetter {\cite{FY92}} also introduced the term {\em
    sovereign category} for a pivotal category.
\end{remark}

A pivotal functor between pivotal categories is a monoidal functor
that also satisfies
\[ \xymatrix{
  FA \ar[r]^{F(i_A)}\ar[rd]_{i_{FA}} &
  F(A^{**})\ar[d]^{\iso} \\
  & (FA)^{**}.  
}
\]

\paragraph{Graphical language.}

The graphical language for pivotal categories is the same as that for
autonomous categories, where the isomorphism $i_A:A\ii A^{**}$ is
represented like an identity map. Of course, there are now additional
diagrams that are the translation of well-formed terms. For example,
when $f:A\ii A$, then (\ref{eqn-not-autonomous}) is a well-formed
diagram of pivotal categories, but not of autonomous categories.
Indeed, in the case of pivotal categories, the problem of winding
numbers (discussed before Theorem~\ref{thm-free-autonomous})
disappears, as winding numbers are taken modulo 2, and hence add
nothing beyond orientation.

\begin{theorem}[Coherence for pivotal categories]
  \label{thm-coherence-pivotal}
  A well-formed equation between morphisms in the language of pivotal
  categories follows from the axioms of pivotal categories if and only
  if it holds in the graphical language up to planar isotopy,
  including rotation of boxes.
\end{theorem}

\begin{caveat}
  Only special cases of this theorem have been proved in the
  literature. Freyd and Yetter {\cite[Thm.~4.4]{FY92}} considered the
  case of the free pivotal category generated by a category. In our
  terminology, this means that they only considered diagrams for
  pivotal categories over {\em simple signatures}, rather than
  over {\em autonomous signatures}. In other words, they only
  considered boxes of the form
  \[ \wmetamorph{f}{A}{B},
  \]
  with exactly one input and one output. Joyal and Street's draft report
  {\cite{JSXX}} claims the general result but contains no proof. 
\end{caveat}

The notion of planar isotopy for pivotal categories includes the
ability to rotate boxes in the plane of the diagram. For example, the
following two diagrams are isotopic in this sense:
\begin{equation}\label{eqn-rotation}
  \vcenter{\wirechart{}{
      \blank\wireopen{ddd}\wireleft{rr}{} &&
      \blank\wireclose{d}
      \\
      & \blank\ulbox{[].[d]}{f}\wireright{r}{} & \blank\\
      & \blank \wireright{r}{} & *{} \\
      \blank \wireright{rr}{} && *{}
      }
    }
  =
  \vcenter{\wirechart{}{
      \blank\wireopen{ddd}\wireright{rr}{} &&
      *{}
      \\
      & \blank\ulbox{[].[d]}{f}\wireright{r}{} & *{}\\
      & \blank \wireright{r}{} & \blank\wireclose{d} \\
      \blank \wireleft{rr}{} && \blank
      }
    }
\end{equation}
This also explains why we have marked a corner of each box. With the
ability to rotate boxes, we need to keep track of their ``natural''
orientation, so that the diagrams from (\ref{eqn-rotation}) can also
be represented like this:
\[   \vcenter{\wirechart{}{
      \blank\wireopen{d}\wireright{rr}{} && *{}
      \\
      \blank\wireleft{r}{}& \blank\drbox{[].[d]}{f}\\
      \blank\wireleft{r}{}& \blank\\
      \blank\wireopen{u}\wireright{rr}{} && *{}
      }
    }
\]
More generally, the adjoint mate of $f:A\ii B$ can be represented by a
rotated box:
\begin{equation}\label{eqn-adjoint-mate}
  \rd{f} \ssep =
  \vcenter{\wirechart{@C=0.5cm@R=0.5cm}{
      \vsblank\wireleft{r}{B}&\wireid\wire{r}{}&\blank\wireclose{d}\\
      \blank\wireright{r}{A}&\blank\ulbox{[]}{f}\wireright{r}{B}&\blank\\
      \blank\wireopen{u}\wire{r}{}&\wireid\wireleft{r}{A}&\vsblank
      }}
  \ssep=\ssep
  \wirechart{@C=1.4cm@R=0.5cm}{
    \wireleft{r}{B}&\blank\drbox{[]}{f}\wireleft{r}{A}&
    }
\end{equation}
Also note that is $f$ is a composite diagram, then the whole diagram
may be rotated to obtain $f^*$.

\subsection{Spherical pivotal categories}\label{subsec-spherical-pivotal}

\begin{definition}[Barrett and Westbury \cite{BW99}]
  A pivotal category is {\em spherical} if for all objects $A$ and
  morphisms $f:A\ii A$,
  \begin{equation}\label{eqn-spherical}
    \vcenter{\wirechart{@R=.5cm}{
        \blank\wireopen{d}\wireleft{rr}{} && \blank\wireclose{d} \\
        \blank\wireright{r}{A} & \blank\ulbox{[]}{f}\wireright{r}{A} & \blank
        }
      }
    =  \vcenter{\wirechart{@R=.5cm}{
        \blank\wireopen{d}\wireright{r}{A} &\blank\ulbox{[]}{f}\wireright{r}{A}& \blank\wireclose{d} \\
        \blank\wireleft{rr}{} & & \blank \\
        }
      }
  \end{equation}
\end{definition}

The intuition behind the ``spherical'' axioms is that diagrams should
be embedded in a 2-sphere, rather than the plane. It is then obvious
that the left-hand side of (\ref{eqn-spherical}) can be continuously
transformed into the right-hand side, namely by moving the loop across
the back of the 2-sphere.

\paragraph{Failure of coherence.}

The spherical axiom is not sound for the graphical language of
diagrams embedded in the 2-sphere. The problem is that the notion of
``diagram embedded in the 2-sphere'' is not compatible with
composition or tensor.  The following is a consequence of the
spherical axiom, but does not hold up to isotopy in the 2-sphere.
\[
\vcenter{\wirechart{@C=.5cm@R=.8cm}{
    \blank\wireopen{dd}\wire{rr}{}&&
    \blank\wireclose{dd}
    \\
    &\blank\ulbox{[]}{g}\\
    \blank\wireright{r}{A}&
    \blank\ulbox{[]}{f}\wireright{r}{A}&
    \blank
    \\}}
=
\vcenter{\wirechart{@C=.5cm@R=.8cm}{
    &\blank\ulbox{[]}{g}\\
    \blank\wireright{r}{A}&
    \blank\ulbox{[]}{f}\wireright{r}{A}&
    \blank
    \\
    \blank\wireopen{u}\wire{rr}{}&&
    \blank\wireclose{u}
    \\}}
=
\vcenter{\wirechart{@C=.5cm@R=.8cm}{
    &\blank\ulbox{[]}{g}\\
    \blank\wireopen{d}\wire{rr}{}&&
    \blank\wireclose{d}
    \\
    \blank\wireright{r}{A}&
    \blank\ulbox{[]}{f}\wireright{r}{A}&
    \blank
    \\}}
\]
Note that this counterexample is similar to the spacial axiom
(\ref{eqn-spacial}), but does not quite imply it. If one adds the
spacial axiom, as we are about to do, then any notion of isotopy is
lost and equivalence of diagrams collapses to isomorphism.

\subsection{Spacial pivotal categories}

\begin{definition}
  A pivotal category is {\em spacial} if it satisfies the spacial
  axiom (\ref{eqn-spacial}) and the spherical axiom
  (\ref{eqn-spherical}).
\end{definition}

\paragraph{Graphical language and coherence.}

The graphical language for spacial pivotal categories is the same as
that for planar pivotal categories, except that equivalence of
diagrams is now taken up to isomorphism. Clearly, the axioms are sound
for the graphical language. We conjecture that they are also complete.

\begin{conjecture}[Coherence for spacial pivotal categories]
  \label{conj-coherence-spacial-pivotal}
  A well-formed equation between morphisms in the language of spacial
  pivotal categories follows from the axioms of spacial pivotal
  categories if and only if it holds in the graphical language up to
  isomorphism.
\end{conjecture}

\subsection{Braided autonomous categories}\label{subsec-braided-autonomous}

An braided autonomous category is an autonomous category that is also
braided (as a monoidal category). The notion of braided autonomous
categories is not extremely natural, as the graphical language is only
sound for a restricted form of isotopy called {\em regular isotopy}.
Nevertheless, it is useful to collect some facts about braided
autonomous categories.

\begin{lemma}[{\cite[Prop.~7.2]{JS93}}]\label{lem-braided-right-autonomous}
  A braided monoidal category is autonomous if and only if it is right
  autonomous.
\end{lemma}

\begin{proof}
  If $\eta:I\ii B\x A$ and $\eps:A\x B\ii I$ form an exact pairing,
  then so do $\symi_{A,B}\cp\eta:I\ii A\x B$ and
  $\eps\cp\sym_{B,A}:B\x A\ii I$. Therefore any right dual of $A$ is
  also a left dual of $A$. \eot
\end{proof}

In any braided autonomous category $\Cc$, we can define a natural
isomorphism $\lop_A:A^{**}\ii A$.  This follows from the proof of
Lemma~\ref{lem-braided-right-autonomous}, using the fact that both $A$
and $A^{**}$ are right duals of $A^*$. More concretely, $\lop_A$
and its inverse are defined by:
\[ \begin{array}{lll}
  \lop_A &=& A^{**}\catarrow{\eta_A\x\id}A^*\x A\x A^{**}
  \catarrow{\id\x\sym_{A,A^{**}}}A^*\x A^{**}\x A\catarrow{\eps_{A^*}\x\id}
  A, \\
  \lopi_A &=& A\catarrow{\id\x\eta_{A^*}}A\x A^{**}\x A^*
  \catarrow{\symi_{A^{**},A}\x\id}A^{**}\x A\x A^*\catarrow{\id\x\eps_A}
  A^{**}.
\end{array}
\]
Here we have written, without loss of generality, as if $\Cc$ were
strict monoidal.  Graphically, $\lop_A$ and its inverse look like this:
\[ \lop_A =
\vcenter{\wirechart{@C-4ex}{
    *{}\wireright{rr}{A^{**}}&&
    \blank\wirecross{d}\wireright{rr}{A}&&
    *{}
    \\&
    \blank\wire{r}{}&
    \blank\wirebraid{u}{.3}\wire{r}{}&
    \blank
    \\&
    \blank\wireopen{u}\wire{rr}{A^*}&&
    \blank\wireclose{u}
    }}
\sep
\lopi_A =
\vcenter{\wirechart{@C-4ex}{
    &\blank\wireopen{d}\wire{rr}{A^*}&&
    \blank\wireclose{d}
    \\
    &\blank\wire{r}{}&
    \blank\wirebraid{d}{.3}\wire{r}{}&
    \blank
    \\
    *{}\wireright{rr}{A}&&
    \blank\wirecross{u}\wireright{rr}{A^{**}}&&
    *{}
    }}
\]
We must note that although $\lop_A$ is a natural isomorphism, it is
not canonical. In general, there exist infinitely many natural
isomorphisms $A\iso A^{**}$. Also, $\lop$ is not a {\em monoidal}
natural transformation, and therefore does not define a pivotal
structure on $\Cc$. A general braided autonomous category is not
pivotal.

\paragraph{Graphical language and coherence.}

The graphical language braided autonomous categories is obtained
simply by adding braids to the graphical language of autonomous
categories. However, the correct notion of equivalence of diagrams is
neither planar isotopy (like for autonomous categories), nor
3-dimensional isotopy (like for braided monoidal categories), but
an in-between notion called {\em regular isotopy} {\cite{Kau86}}.

\begin{table}
  \[
  \begin{array}{lc}
    \m{(R1)}& \m{\resizebox{1.7in}{!}{\includegraphics{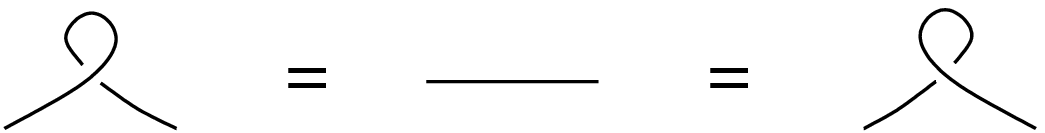}}}\\[3ex]
    \m{(R2)}& \m{\resizebox{1.0in}{!}{\includegraphics{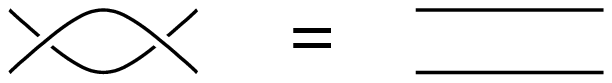}}}\\[3ex]
    \m{(R3)}& \m{\resizebox{1.0in}{!}{\includegraphics{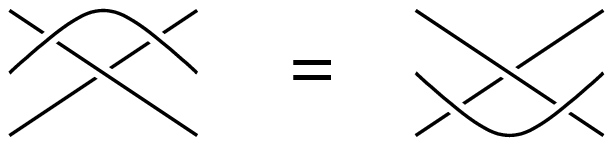}}}\\
  \end{array}
  \sep
  \begin{array}{lc}
    \m{($\Lambda$1)}& \m{\resizebox{1.7in}{!}{\includegraphics{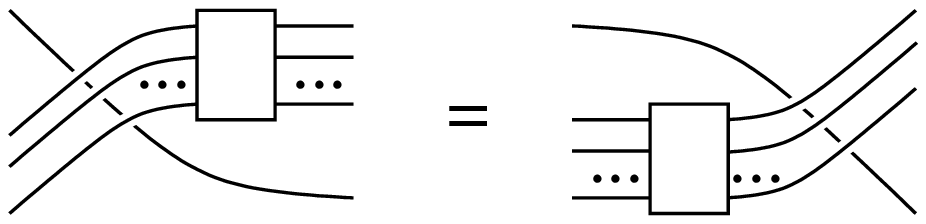}}}\\[4.5ex]
    \m{($\Lambda$2)}& \m{\resizebox{1.7in}{!}{\includegraphics{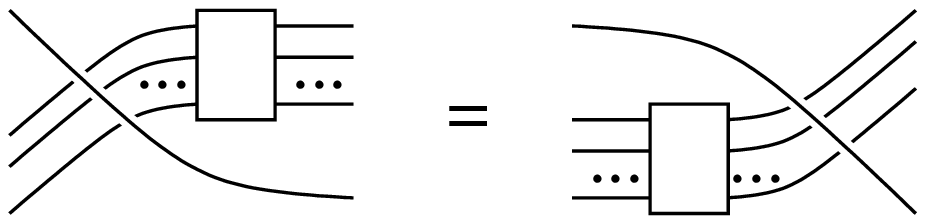}}}\\
  \end{array}
  \]
  \caption{Reidemeister moves and $\Lambda$-moves}
  \label{tab-reidemeister}
\end{table}

It is well-known that 3-dimensional isotopy of links and tangles is
equivalent to planar isotopy of their (non-degenerate) projections
onto a 2-dimensional plane, plus the three {\em Reidemeister moves}
{\cite{Rei32}} shown as (R1)--(R3) in Figure~\ref{tab-reidemeister}.
To extend this to diagrams with nodes, one also has to add the moves
($\Lambda$1) and ($\Lambda$2).

{\em Regular isotopy} is defined to be the equivalence obtained by
dropping Reidemeister move (R1). Note that regular isotopy is an
equivalence on 2-dimensional representation of 3-dimensional diagrams
(and not of 3-dimensional diagrams themselves).

\begin{theorem}[Coherence for braided autonomous categories]
  A well-formed equation between morphisms in the language of braided
  autonomous categories follows from the axioms of braided autonomous
  categories if and only if it holds in the graphical language up to
  regular isotopy.
\end{theorem}

\begin{caveat}
  Only special cases of this theorem have been proved in the
  literature. Freyd and Yetter {\cite[Thm.~3.8]{FY92}} proved this
  only for diagrams over a simple signature. 
\end{caveat}

\subsection{Braided pivotal categories}\label{subsec-braided-pivotal}

\begin{lemma}[Deligne, see {\cite[Prop.~2.11]{Yet92}}]\label{lem-braided-pivotal}
  Let $\Cc$ be a braided autonomous category. Then giving a twist
  $\theta_A:A\ii A$ on $\Cc$ (making $\Cc$ into a balanced category)
  is equivalent to giving a pivotal structure $i_A:A\ii A^{**}$
  (making $\Cc$ into a pivotal category).
\end{lemma}

The lemma is remarkable because the concept of a braided autonomous
category does not include any assumption relating the braided
structure to the autonomous structure.  Moreover, the axioms for a
twist depend only on the braided structure, whereas the axioms for a
pivotal structure depend only on the autonomous structure. Yet, they
are equivalent if $\Cc$ is braided autonomous.

\begin{proofof}{Lemma~\ref{lem-braided-pivotal}}
  Recall the natural isomorphism $\lop_A:A^{**}\ii A$ that was defined
  in Section~\ref{subsec-braided-autonomous} for any braided
  autonomous category. Given a twist $\theta_A:A\ii A$, we define a
  pivotal structure by
  \begin{equation}\label{eqn-i-from-theta}
    i_A = A\catarrow{\theta_A}A\catarrow{\lopi_A}A^{**}.
  \end{equation}
  Conversely, given a pivotal structure $i_A:A\ii A^{**}$, we define
  a twist by
  \begin{equation}\label{eqn-theta-from-i}
    \theta_A = A\catarrow{i_A}A^{**}\catarrow{\lop_A}A.
  \end{equation}
  The two constructions are clearly each other's inverse. To verify
  their properties, it is obvious that $i_A$ is a natural isomorphism
  if and only if $\theta_A$ is a natural isomorphism. Moreover,
  $\theta_I=\id$ iff $i_I=\lopi_I$, and $\lopi_I$ is the canonical
  isomorphism $I\iso I^{**}$. What remains to be shown is that
  $\theta$ satisfies equation (\ref{eqn-balanced}) if and only if $i$
  satisfies equation (\ref{eqn-pivotal-monoidal}). However, this is a
  direct consequence of the following fact about $\lop$, which is
  easily verified:
  \[ \xymatrix@R-5mm{
    A^{**}\x B^{**} \ar[d]_{\iso}\ar[r]^{\sym_{A,B}} &
    B^{**}\x A^{**} \ar[dd]^{\lop_B\x\lop_A} \\
    (A\x B)^{**}\ar[d]_{\lop_{A\x B}}
    \\
    A\x B & B\x A. \ar[l]^{\sym_{B,A}}
    }
  \]
  \eottwo
\end{proofof}

\begin{corollary}
  A braided pivotal category is the same thing as a balanced
  autonomous category. \eot
\end{corollary}

\begin{remark}\label{rem-braided-pivotal-other-theta}
  While Lemma~\ref{lem-braided-pivotal} establishes a one-to-one
  correspondence between twists and pivotal structures, the
  correspondence is not canonical. Indeed, instead of
  (\ref{eqn-i-from-theta}) and (\ref{eqn-theta-from-i}), we could have
  equally well used
  \begin{equation}\label{eqn-i-from-theta-alt}
    i_A = A\catarrow{\theta\inv_A}A\catarrow{\lopalt_A}A^{**}
  \end{equation}
  and
  \begin{equation}\label{eqn-theta-from-i-alt}
    \theta_A = A\catarrow{\lopalt_A}A^{**}\catarrow{i\inv_A}A,
  \end{equation}
  where
  \[ \lopalt_A = \m{\resizebox{.4in}{!}{\includegraphics{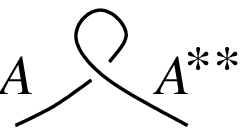}}}.
  \]
  In fact, there are a countable number of such similar one-to-one
  correspondences, all induced by the existence of a monoidal natural
  transformation $\lopalt_A{}\inv\cp i_A\cp \lop_A\cp i_A:A\ii A$.
  They all coincide if and only if the category is tortile, as
  discussed in the next section.
\end{remark}

\paragraph{Graphical language and coherence.}

The graphical language for braided pivotal categories is the same as
the graphical language for pivotal categories, with the addition of
braids. Equivalence of diagrams is up to regular isotopy, just as for
braided autonomous categories (see
Section~\ref{subsec-braided-autonomous}).

\begin{theorem}[Coherence for braided pivotal categories]
  \label{thm-coherence-braided-pivotal}
  A well-formed equation between morphisms in the language of braided
  pivotal categories follows from the axioms of braided pivotal
  categories if and only if it holds in the graphical language up to
  regular isotopy.
\end{theorem}

\begin{caveat}\label{cav-coherence-braided-pivotal}
  Only special cases of this theorem have been proved in the
  literature. Freyd and Yetter {\cite[Thm.~4.4]{FY92}} proved this
  only for diagrams over a simple signature. 
\end{caveat}

\begin{remark}
  The equation
  \[ \resizebox{1.8in}{!}{\includegraphics{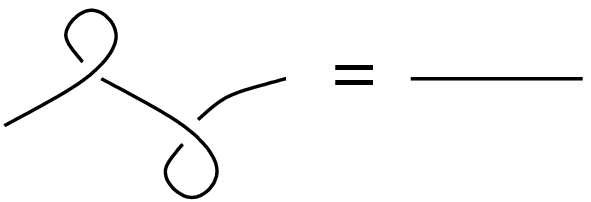}}
  \]
  holds up to regular isotopy, as it can be proved using only the
  Reidemeister moves (R2) and (R3). It is therefore valid in braided
  pivotal categories (or even braided autonomous categories).  On the
  other hand, the equation
  \[ \resizebox{1.8in}{!}{\includegraphics{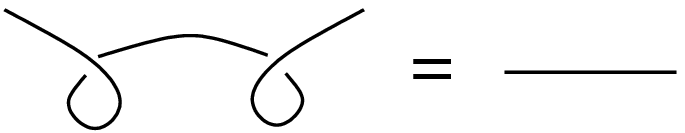}}
  \]
  holds up to isotopy, but not up to regular isotopy (because regular
  isotopy preserves total curvature, as pointed out by Freyd and Yetter
  {\cite[p.~169]{FY89}}). It is therefore not valid in braided pivotal
  categories. The use of regular isotopy does not seem natural, and
  this is precisely the reason why Joyal and Street introduced tortile
  categories, which we discuss in the next section.
\end{remark}

\begin{remark}\label{braided-not-spherical}
  A braided pivotal category is not in general spherical (and
  therefore also not spacial). Indeed, instead of the spherical axiom
  (\ref{eqn-spherical}), only the following holds up to regular
  isotopy:
  \[ \resizebox{2.4in}{!}{\includegraphics{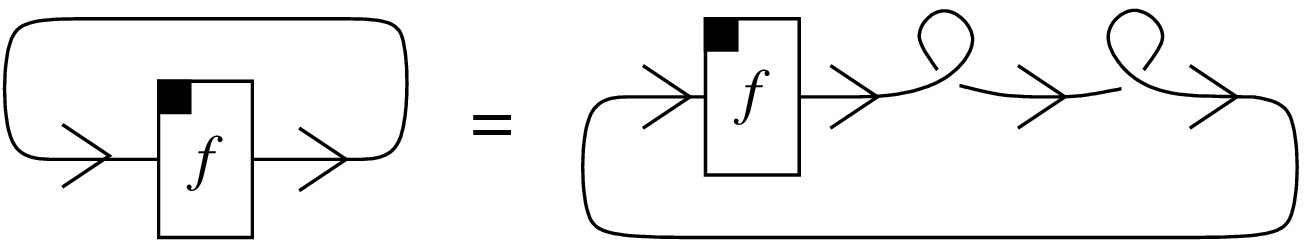}}
  \]
  Along with Remark~\ref{rem-braided-pivotal-other-theta}, this is
  further evidence that braided pivotal categories (and braided
  autonomous categories) are not ``natural'' notions.
\end{remark}

\subsection{Tortile categories}

\begin{lemma}\label{lem-tortile}
  Consider a braided pivotal category, which is equivalently balanced
  autonomous via (\ref{eqn-i-from-theta}) and
  (\ref{eqn-theta-from-i}).  For any object $A$ the following are
  equivalent:
  \begin{enumerate}\alphalabels
  \item $
      (\eps_{A^*}\x\id_A) \cp
      (\id_{A^*}\x\symi_{A^{**},A}) \cp
      (\eta_{A}\x\id_{A^{**}})\cp
      i_A \cp
      (\eps_{A^*}\x\id_A) \cp
      (\id_{A^*}\x\sym_{A,A^{**}}) \cp
      (\eta_{A}\x\id_{A^{**}})\cp
      i_A
      = \id_A,
    $
    or graphically:
    \[ \resizebox{1.8in}{!}{\includegraphics{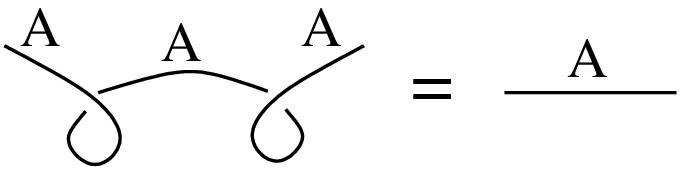}}
    \]
  \item $\theta_{A^*}=(\theta_A)^*$.
  \end{enumerate}
\end{lemma}

\begin{proof}
  The proof is a straightforward calculation, but it is best explained
  by the fact that the following hold in the graphical language:
  \[ \theta_A = \m{\resizebox{.4in}{!}{\includegraphics{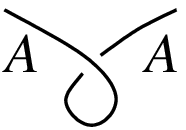}}} \sep
  (\theta_A)^* = \m{\resizebox{.4in}{!}{\includegraphics{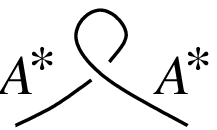}}} \sep
  \theta_{A^*} = \m{\resizebox{.4in}{!}{\includegraphics{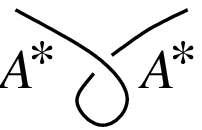}}} \sep
  (\theta_{A^*})\inv = \m{\resizebox{.4in}{!}{\includegraphics{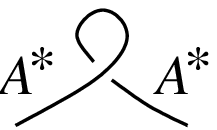}}}.
  \]
  Therefore, the equation (b) is equivalent to
  \[ \resizebox{1.3in}{!}{\includegraphics{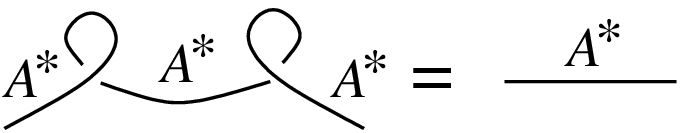}},
  \]
  which is the adjoint mate of (a).\eot
\end{proof}

\begin{remark}
  The condition in Lemma~\ref{lem-tortile}(a) holds if and only if the
  two definitions of $\theta_A$ from (\ref{eqn-theta-from-i}) and
  (\ref{eqn-theta-from-i-alt}) coincide.
\end{remark}

\begin{definition}[{\cite{JS93}}]
  A {\em tortile category} is a braided pivotal category satisfying
  the condition of Lemma~\ref{lem-tortile}(a). Equivalently, a tortile
  category is a balanced autonomous category satisfying the condition
  of Lemma~\ref{lem-tortile}(b).
\end{definition}

\begin{remark}[Terminology]
  A tortile category is also sometimes called a {\em ribbon category},
  see e.g.~{\cite{Tur94}}.
\end{remark}

\paragraph{Graphical language and coherence.}

The graphical language for tortile categories is like the graphical
language for braided pivotal categories, except that morphisms are
represented by ribbons, rather than wires. These ribbons are just like
the ones for balanced categories from Section~\ref{subsec-balanced}.
Units and counits are represented in the obvious way, for example
\[ \eta_A = \m{\resizebox{.4in}{!}{\includegraphics{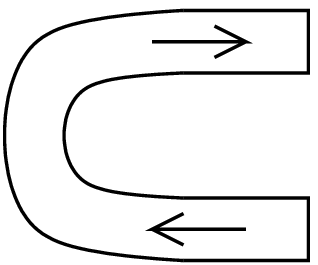}}},
\sep \eps_A = \m{\resizebox{.4in}{!}{\includegraphics{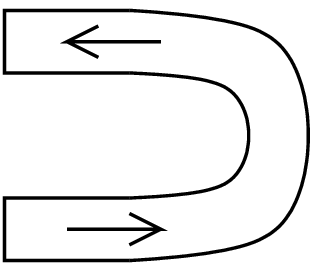}}}.
\]
The twist map $\theta_A:A\ii A$ can be represented in several
equivalent ways:
\[ \theta_A = \m{\resizebox{2.5cm}{!}{\includegraphics{twistA}}}
  = \m{\resizebox{2.3cm}{!}{\includegraphics{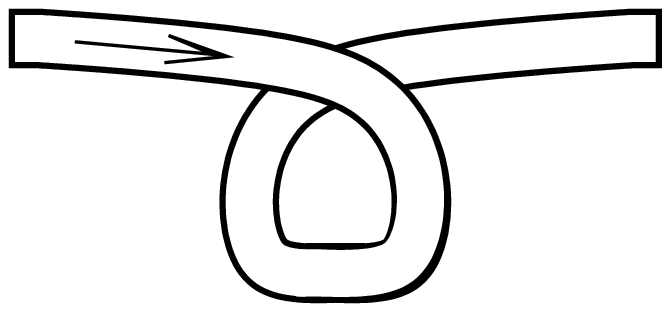}}}
  = \m{\resizebox{2.3cm}{!}{\includegraphics{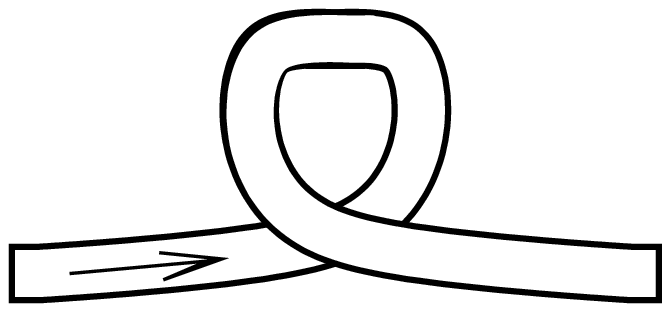}}}.
\]
Note that these diagrams are equivalent up to framed 3-dimensional
isotopy, and define the same morphism in a tortile category. (On the
other hand, in a mere braided pivotal category, the latter two
diagrams are not equal). Also note that the map $\lop_A$ from
Section~\ref{subsec-braided-autonomous} is also represented in the
graphical language as
\[ \lop_A = \m{\resizebox{2.3cm}{!}{\includegraphics{twist-loop-down}}},
\]
but this is of type $\lop_A:A^{**}\ii A$, whereas $\theta_A:A\ii A$.
They differ, of course, only by an invisible pivotal map $i_A:A\ii
A^{**}$.

\begin{theorem}[Coherence for tortile categories]
  \label{thm-coherence-tortile}
  A well-formed equation between morphisms in the language of tortile
  categories follows from the axioms of tortile categories if and only
  if it holds in the graphical language up to framed 3-dimensional
  isotopy.
\end{theorem}

\begin{caveat}
  \label{cav-coherence-tortile}
  Only special cases of this theorem have been proved in the
  literature. Shum {\cite[Thm.~6.1]{Shum94}} proved it for the case of
  the free tortile category generated by a category, i.e., for
  diagrams over a simple signature only.
\end{caveat}

\subsection{Compact closed categories}\label{subsec-compact-closed}

A compact closed category is a tortile category that is symmetric (as
a balanced monoidal category) in the sense of
Section~\ref{subsec-symmetric-monoidal}.  Equivalently, because of
Remark~\ref{rem-symmetric-from-balanced}, a compact closed category is
a tortile category in which $\theta_A=\id_A$ for all $A$.

The definition can be simplified. Notice that a right autonomous
symmetric monoidal category is automatically autonomous (by
Lemma~\ref{lem-braided-right-autonomous}), balanced (with
$\theta_A=\id_A$) and therefore pivotal (by
Lemma~\ref{lem-braided-pivotal}). Moreover, it is tortile (because
$\theta_{A^*}=(\theta_A)^*=\id_{A^*}$). We can therefore define:

\begin{definition}
  A {\em compact closed category} is a right autonomous symmetric
  monoidal category.
\end{definition}

\begin{remark}
  By analogy with Remark~\ref{rem-twisted-symmetric}, it is possible
  for a compact closed category to possess a non-trivial twist (with
  the associated non-trivial pivotal structure), in addition to the
  trivial twist $\theta_A=\id_A$, making it into a tortile category.
  In other words, for a given tortile category, the symmetry condition
  $\sym_{A,B} = \symi_{B,A}$ does not in general imply
  $\theta_A=\id_A$. However, it does imply $\theta_A^2=\id_A$, as the
  following argument shows:
  \[ \theta_A^2 = \m{\resizebox{1in}{!}{\includegraphics{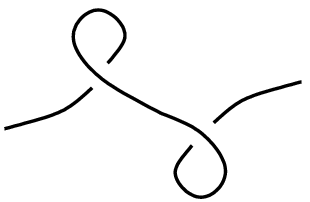}}}
  = \m{\resizebox{.9in}{!}{\includegraphics{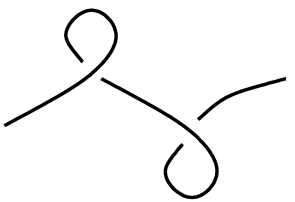}}} = \id_A.
  \]
  To construct an example where $\theta\neq\id$, consider the category
  $\Cc$ of finite-dimensional real vector spaces and linear functions.
  Define an equivalence relation on objects by $A\sim B$ iff $\dim(A\x
  B)$ is a square. Then define a subcategory $\Cc_{\sim}$ by
  \[ \hom_{\Cc_{\sim}}(A,B) = \left\{\begin{array}{lp{1in}}
      \hom_{\Cc}(A,B) & if $A\sim B$,\\
      \emptyset & else.
    \end{array}\right.
  \]
  Then $\Cc_{\sim}$ is compact closed. Let $\N^+=\s{1,2,3,\ldots}$ be
  the positive integers, and consider some multiplicative homomorphism
  $\phi:\N^+\ii\s{-1,1}$. Any such homomorphism is determined by a
  sequence $a_1,a_2,\ldots\in\s{-1,1}$ via
  \[ \phi(p_1^{n_1}p_2^{n_2}\cdots p_k^{n_k}) = a_1^{n_1}a_2^{n_2}\cdots
  a_k^{n_k},
  \]
  where $p_i$ is the $i$th prime number. Finally, define the twist map
  $\theta_A$ as multiplication by the scalar $\phi(\dim(A))$, or as
  $\id_A$ if $A$ is 0-dimensional. With this twist, $\Cc_{\sim}$ is
  tortile. In fact, this shows that there exists a continuum of
  possible twists on $\Cc_{\sim}$.
\end{remark}

\begin{examples}
  The monoidal category $(\Rel,\times)$ is compact closed with
  $A^*=A$.  The category $(\FdVect,\x)$ of finite dimensional vectors
  spaces is compact closed with $A^*$ the dual space of $A$, and
  similarly for the category of finite dimensional Hilbert spaces
  $(\FdHilb,\x)$. The corresponding categories of possibly infinite
  dimensional spaces are not autonomous. $(\Cob,+)$ is compact closed
  with $A^*$ equal to $A$ with reversed orientation.
\end{examples}

\paragraph{Graphical language and coherence.}

The graphical language for compact closed categories is like that of
tortile categories, except that we remove the framing and twist maps,
and use symmetries instead of braidings.

\begin{theorem}[Coherence for compact closed categories]
  \label{thm-coherence-compact-closed}
  A well-formed equation between morphisms in the language of compact
  closed categories follows from the axioms of compact closed
  categories if and only if it holds, up to isomorphism of diagrams,
  in the graphical language.
\end{theorem}

\begin{caveat}
  \label{cav-coherence-compact-closed}
  The special case of diagrams over a simple signature was proven
  by Kelly and Laplaza {\cite[Thm.~8.2]{KL80}}. The general case does
  not appear in the literature.
\end{caveat}

\section{Traced categories}\label{sec-traced}

The graphical languages considered in Section~\ref{sec-progressive}
were {\em progressive}, which means that all wires were oriented
left-to-right. By contrast, the graphical languages of autonomous
categories in Section~\ref{sec-autonomous} allow wires to be oriented
left-to-right or right-to-left.  We now turn out attention to an
intermediate notion, namely {\em traced} categories.

Like autonomous graphical languages, traced graphical languages permit
loops, but with a restriction: all wires must be directed
left-to-right at their endpoints. In other words, traced diagrams are
like autonomous diagrams, but are taken relative to a {\em monoidal
  signature} (see Section~\ref{subsec-planar-monoidal}), rather than
an {\em autonomous signature} (see
Section~\ref{subsec-planar-autonomous}).
Table~\ref{tab-traced-vs-autonomous} shows a typical example of a
traced diagram, and a typical example of an autonomous diagram that is
not a traced diagram.

\begin{table}
  \[ (a)\ssep\m{\wirechart{@C=1cm@R=.8cm}{
    \blank\wireopen{d}\wire{rr}{}&&
    \blank\wireclose{d}
    \\
    \blank\wireright{r}{}&
    \blank\wireright{r}{}&
    \blank
    \\
    *{}\wireright{r}{}&
    \blank\ulbox{[].[u]}{f}\wireright{r}{}&
    *{}
    \\}}\sep
  (b)\ssep\raisebox{-5mm}{\m{\wirechart{@C=1cm@R=.8cm}{
      \wireright{r}{}&
      \blank\ulbox{[].[d]}{f}\wireright{r}{}&
      \blank\wireclose{d}\\
      \wireleft{r}{}&
      \blank\wireleft{r}{} &
      \blank
      }}}
  \]
  \caption{(a) A traced diagram. (b) An autonomous diagram that is not traced.}
  \label{tab-traced-vs-autonomous}
\end{table}

Logically, we should have considered traced categories before pivotal
categories, because traced categories have less structure than pivotal
categories (i.e., every pivotal category is traced, and not the other
way around). However, many of the coherence theorems of this section
are consequences of the corresponding theorems for pivotal categories,
and therefore it made sense to present the pivotal notions first.

Symmetric traced categories and their graphical language (in the
strict monoidal case, and with one additional axiom) were first
introduced in the 1980's by {\Stefanescu} and {\Cazanescu} under the
name ``biflow'' {\cite{6,12,8}}.  Joyal, Street, and Verity later
rediscovered this notion independently, generalized it to balanced
monoidal categories, and proved the fundamental embedding theorem
relating balanced traced categories to tortile categories
{\cite{JSV96}}.

\begin{remark}
  Joyal, Street, and Verity use the term {\em traced monoidal
    category}. However, I prefer {\em traced category}, usually
  prefixed by an adjective such as planar, spacial, balanced,
  symmetric. The word ``monoidal'' is redundant, because one cannot
  have a traced structure without a monoidal structure. Also, by
  putting the adjective before the word ``traced'', rather than after
  it, we make it clear that the traced structure, and not just the
  underlying monoidal structure, if being modified.
\end{remark}

\subsection{Right traced categories}\label{subsec-right-traced}

\begin{definition}
  A {\em right trace} on a monoidal category is a family of operations
  \[      \TrR^X:\hom(A\x X,B\x X)\ii\hom(A,B),
  \]
  satisfying the following four axioms. For notational convenience, we
  assume without loss of generality that the monoidal structure is
  strict.
  \begin{enumerate}\alphalabels
  \item Tightening (naturality in $A,B$): $\TrR^X((g\x\id_X)\cp f\cp
    (h\x\id_X))=g\cp (\TrR^Xf)\cp h$;
  \item Sliding (dinaturality in $X$): $\TrR^Y(f\cp(\id_A\x
    g))=\TrR^X((\id_B\x g)\cp f)$, where $f:A\x X\ii B\x Y$ and $g:Y\ii
    X$;
  \item Vanishing: $\TrR^I f=f$ and $\TrR^{X\x Y} f=\TrR^X(\TrR^Y(f))$;
  \item Strength. $\TrR^X(g\x f)=g\x\TrR^X f$.
  \end{enumerate}
  A {\em (planar) right traced category} is a monoidal category equipped
  with a right trace.
\end{definition}

These axioms are similar to those of Joyal, Street, and Verity
{\cite{JSV96}}, except that we have omitted the yanking axioms which
does not apply in the planar case, and we have replaced the non-planar
``superposing'' axiom by the planar ``strength'' axiom. I do not know
whether this set of planar axioms appears in the literature.

\paragraph{Graphical language and coherence.}

The right trace of a diagram $f:A\x X\ii B\x X$ is graphically
represented by drawing a loop from the output $X$ to the input $X$, as
follows:

\begin{equation}\label{eqn-right-trace}
\TrR^X f = \vcenter{\wirechart{@C=1cm@R=.8cm}{
    \blank\wireopen{d}\wire{rr}{}&&
    \blank\wireclose{d}
    \\
    \blank\wireright{r}{X}&
    \blank\wireright{r}{X}&
    \blank
    \\
    *{}\wireright{r}{A}&
    \blank\ulbox{[].[u]}{f}\wireright{r}{B}&
    *{}
    \\}}
\end{equation}

Note that in the graphical language of right traced categories, parts
of wires can be oriented right-to-left, but each wire must be oriented
left-to-right near the endpoints.  The four axioms of right traced
categories are illustrated in the graphical language in
Table~\ref{tab-right-traced}.
\begin{table}
  \hfill\resizebox{0.95\textwidth}{!}{\includegraphics{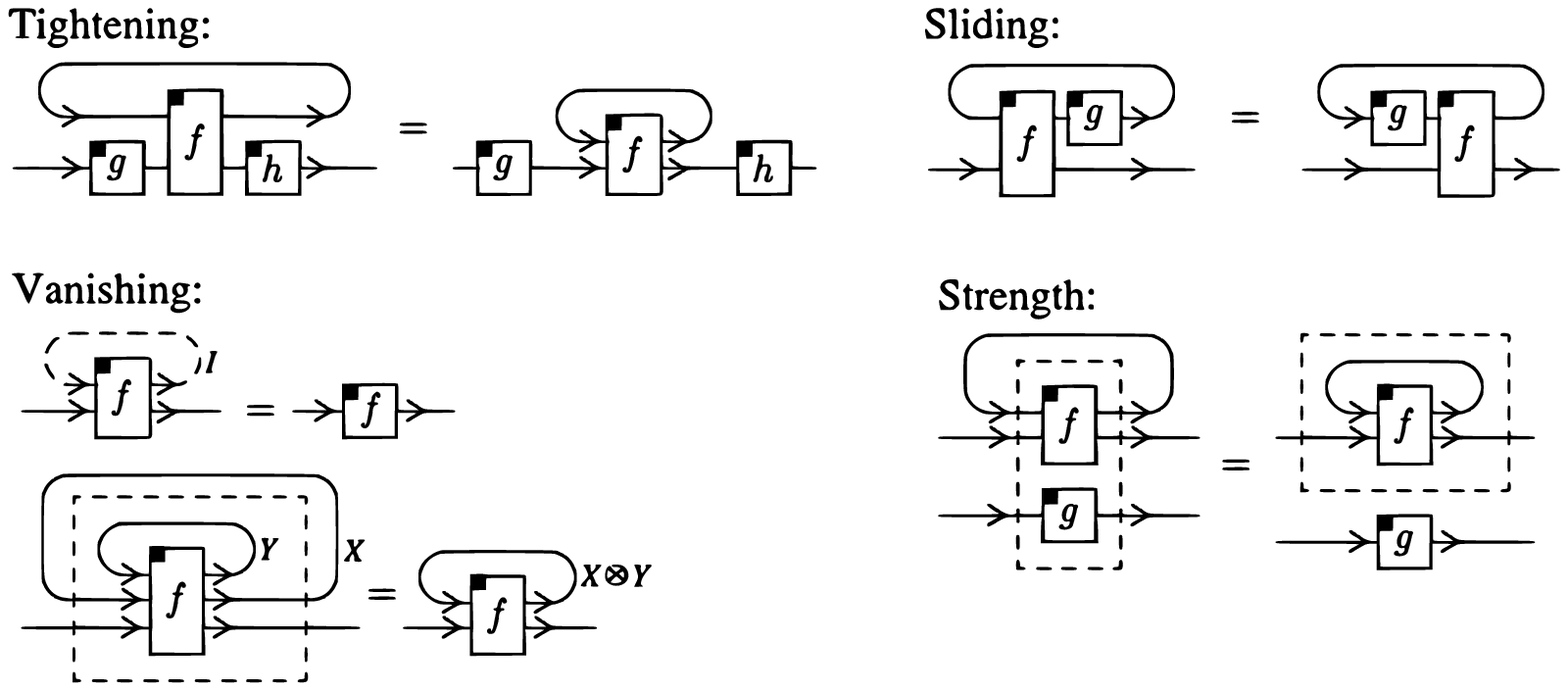}}\hfill
  \caption{The axioms of right traced categories}
  \label{tab-right-traced}
\end{table}
The axioms of right traced categories are obviously sound for the
graphical language, up to planar isotopy. We conjecture that they are
also complete.

\begin{conjecture}[Coherence for right traced categories]
  A well-formed equation between morphism terms in the language of
  right traced categories follows from the axioms of right traced
  categories if and only if it holds in the graphical language up
  planar isotopy.
\end{conjecture}

This is a weak conjecture, in the sense that there is not much
empirical evidence to support it, nor is there an obvious strategy for
a proof. If this conjecture turns out to be false, the axioms for
right traced categories should be amended until it becomes true.

The concept of a {\em left trace} is defined similarly as a family of
operations
\[      \TrL^X:\hom(X\x A,X\x B)\ii\hom(A,B),
\]
satisfying symmetric axioms. A left trace is graphically depicted as
follows:
\begin{equation}\label{eqn-left-trace}
\TrL^X g = \vcenter{\wirechart{@C=1cm@R=.8cm}{
    *{}\wireright{r}{A}&
    \blank\wireright{r}{B}&
    *{}
    \\
    \blank\wireright{r}{X}&
    \blank\ulbox{[].[u]}{g}\wireright{r}{X}&
    \blank
    \\
    \blank\wireopen{u}\wire{rr}{}&&
    \blank\wireclose{u}
    \\}}
\end{equation}

We say that a monoidal functor $F$ {\em preserves right traces} if
$F(\TrR^X f) = \TrR^{FX}((\phi^2)\inv\cp Ff\cp\phi^2)$, and similarly
for left traces.

\subsection{Planar traced categories}

\begin{definition}
  A {\em planar traced category} is a monoidal category equipped with
  a right trace and a left trace, such that the two traces satisfy
  three additional axioms:
  \begin{enumerate}\alphalabels
  \item Interchange: $\TrR^X(\TrL^Y f)=\TrL^Y(\TrR^X f)$, for all
    $f:Y\x A\x X\ii Y\x B\x X$;
  \item Left pivoting: $\TrR^B(\id_B\x f)=\TrL^A(f\x\id_A)$, for all
    $f:I\ii A\x B$;
  \item Right pivoting: $\TrR^B(\id_B\x f)=\TrL^A(f\x\id_A)$, for all
    $f:A\x B\ii I$.
  \end{enumerate}
\end{definition}

\paragraph{Graphical language and coherence.}

The graphical language of planar traced categories consists of
diagrams using the left and right trace together, modulo planar
isotopy. The axioms of interchange, left pivoting, and right pivoting
are shown graphically in Table~\ref{tab-planar-traced}. Compare also
equation~(\ref{eqn-rotation}) on page~\ref{eqn-rotation}.
\begin{table}
  \[ \begin{array}{ccc}
    \m{\resizebox{4cm}{!}{\includegraphics{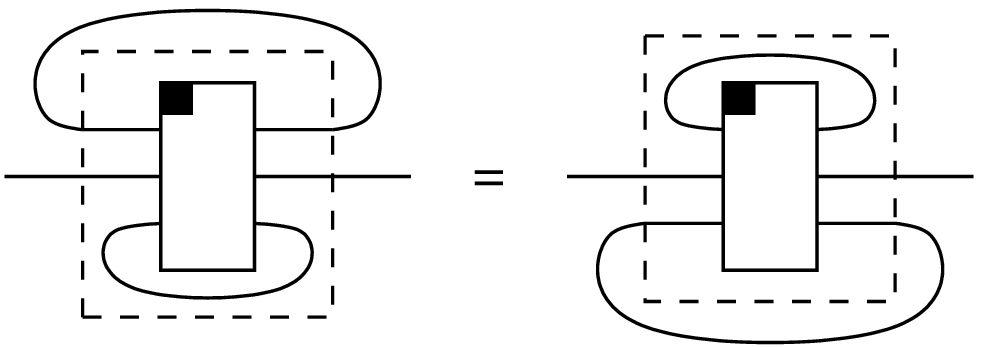}}} &
    \m{\resizebox{3cm}{!}{\includegraphics{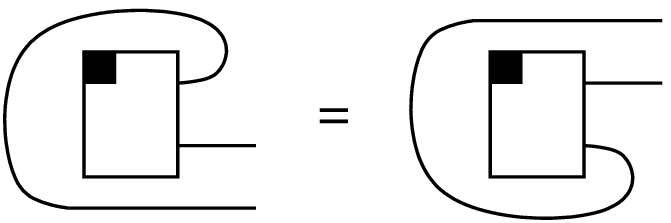}}} &
    \m{\resizebox{3cm}{!}{\includegraphics{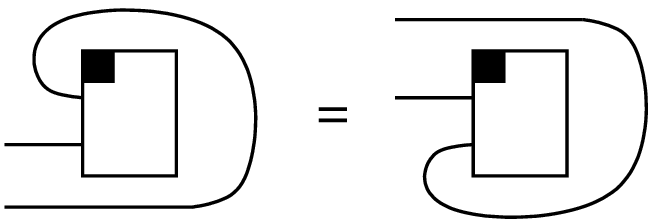}}} \\
    \mbox{(a) interchange} &
    \mbox{(b) left pivoting} &
    \mbox{(c) right pivoting}
  \end{array}
  \]
  \caption{Axioms relating left and right trace}
  \label{tab-planar-traced}
\end{table}
The axioms are clearly sound; we conjecture that they are also
complete:

\begin{conjecture}[Coherence for planar traced categories]
  \label{conj-coherence-planar-traced}
  A well-formed equation between morphism terms in the language of
  planar traced categories follows from the axioms of planar traced
  categories if and only if it holds in the graphical language up
  planar isotopy.
\end{conjecture}

As for right traced categories, this conjecture is weak. If it turns
out to be false, then one should amend the axioms of planar traced
categories accordingly.

\begin{remark}\label{rem-planar-not-free}
  Even if the conjecture is true, the graphical language does not in
  itself give an easy description of the free planar traced category.
  This is because there are diagrams, such as the following, that
  ``look'' planar traced, but are not actually the diagram of any planar
  traced term (not even up to planar isotopy).
  \[ \resizebox{1.6in}{!}{\includegraphics{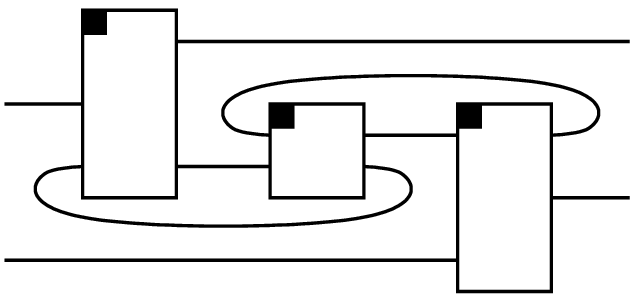}}
  \]
  It is not obvious how to characterize the ``planar traced'' diagrams
  intrinsically, or how to extend the notion of planar traced
  categories to encompass all such diagrams.
\end{remark}

\begin{remark}
  \label{rem-pivotal-is-traced}
  An autonomous category is not necessarily traced. However, every
  pivotal category is planar traced with the obvious definitions of
  left and right trace:
  \[ \begin{array}{lll}
    \TrR^X f &=& (\id_B\x\eps_X)\cp((f\cp (\id_A\x i\inv_X))\x\id_{X^*})\cp(\id_A\x\eta_{X^*}),\\
    \TrL^X f &=& (\eps_{X^*}\x\id_B)\cp(\id_{X^*}\x((i_X\x\id_B)\cp f))\cp(\eta_{X}\x\id_A).\\
  \end{array}
  \]
  In the graphical language, this looks just like the diagrams
  (\ref{eqn-right-trace}) and (\ref{eqn-left-trace}). As a
  consequence, each diagram of planar traced categories can be
  regarded as a diagram of planar pivotal categories, but not the
  other way around.
\end{remark}

\subsection{Spherical traced categories}

The concept of a spherical traced category is analogous to that of
spherical pivotal categories from
Section~\ref{subsec-spherical-pivotal}.

\begin{definition}
  A planar traced category satisfies the {\em spherical axiom} if for
  all $f:A\ii A$,
  \begin{equation}\label{eqn-spherical-traced}
    \TrL^A f = \TrR^A f,
  \end{equation}
  or equivalently, in the graphical language:
  \[
  \vcenter{\wirechart{@R=.5cm}{
      \blank\wireopen{d}\wireright{r}{A} &\blank\ulbox{[]}{f}\wireright{r}{A}& \blank\wireclose{d} \\
      \blank\wire{rr}{} & & \blank \\
      }
    }
  =    \vcenter{\wirechart{@R=.5cm}{
      \blank\wireopen{d}\wire{rr}{} && \blank\wireclose{d} \\
      \blank\wireright{r}{A} & \blank\ulbox{[]}{f}\wireright{r}{A} & \blank
      }
    }
  \]
  A {\em spherical traced category} is a planar traced category
  satisfying the spherical axiom.
\end{definition}

Every spherical pivotal category is spherical traced.

\paragraph{Failure of coherence.}

Just like for spherical pivotal categories, the graphical language of
spherical traced categories is not coherent for any geometrically
useful notion of equivalence of diagrams.

\subsection{Spacial traced categories}\label{subsec-spacial-traced}

\begin{definition}
  A {\em spacial traced category} is a planar traced category if it
  satisfies the spacial axiom (\ref{eqn-spacial}) and the spherical
  axiom (\ref{eqn-spherical-traced}).
\end{definition}

\paragraph{Graphical language and coherence.}

The graphical language for spacial traced categories is the same as
that for planar traced categories, except that equivalence of diagrams
is now taken up to isomorphism.

\begin{conjecture}[Coherence for spacial traced categories]
  \label{conj-coherence-spacial-traced}
  A well-formed equation between morphism terms in the language of
  spacial traced categories follows from the axioms of spacial traced
  categories if and only if it holds, up to isomorphism of diagrams,
  in the graphical language.
\end{conjecture}

\begin{remark}
  Every spacial pivotal category is clearly spacial traced. I do not
  know whether conversely every spacial traced category can be
  faithfully embedded in a spacial pivotal category. If this is true,
  then Conjecture~\ref{conj-coherence-spacial-traced} follows from
  Conjecture~\ref{conj-coherence-spacial-pivotal}.
\end{remark}

\subsection{Braided traced categories}

Braided traced categories, like braided pivotal categories, are a
somewhat unnatural notion, because coherence is only satisfied up to
regular isotopy. (If one considers full isotopy, one obtains the more
natural notion of balanced traced categories, which we will consider
in the next section). Nevertheless, we include this section on braided
traced categories, not least because it is the first traced notion for
which we can actually prove a coherence theorem (modulo
Caveat~\ref{cav-coherence-braided-pivotal}).

\begin{definition}
  A {\em braided traced category} is a planar traced category with a
  braiding (as a monoidal category), such that
  \begin{equation}\label{eqn-braided-traced}
    (\TrL^A\sym_{A,A})\cp(\TrR^A\symi_{A,A}) = \id_A,
  \end{equation}
  or graphically:
  \[ \resizebox{1.8in}{!}{\includegraphics{braided-pivotal}}.
  \]
\end{definition}

\begin{lemma}
  \begin{enumerate}\alphalabels
  \item The axiom (\ref{eqn-braided-traced}) does not follow from the
    remaining axioms.
  \item In the presence of the remaining axioms,
    (\ref{eqn-braided-traced}) is equivalent to
    \begin{equation}\label{eqn-braided-traced-2}
      (\TrL^A\symi_{A,A})\cp(\TrR^A\sym_{A,A}) = \id_A,
    \end{equation}
    or graphically:
    \[ \resizebox{1.8in}{!}{\includegraphics{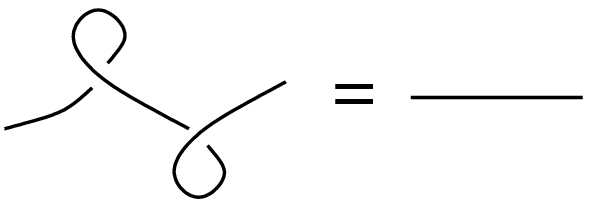}}.
    \]
  \item In the presence of the remaining axioms of braided traced
    categories, the left and right pivoting axioms are redundant.
  \end{enumerate}
\end{lemma}

\begin{proof}
  (a) To see this, consider morphism terms in the language of braided
  traced categories with one object generator and no morphism
  generators.  Define the {\em degree} of a term to the be tensor
  product of all traced-out objects, i.e., $\deg(\id)=I$, $\deg(f\cp
  g)=\deg(f)\x \deg(g)$, $\deg(\TrR^Xf) = X\x \deg(f)$, etc. This is
  well-defined up to isomorphism.  All the axioms of planar traced
  categories and braided categories respect degree; the only axioms
  where the left-hand side and right-hand side could potentially have
  different degree are sliding in Table~\ref{tab-right-traced} and
  pivoting in Table~\ref{tab-planar-traced}. However, in the absence
  of morphism generators, it is easy to show that all morphism terms
  are of the form $f:A\ii B$ where $A\iso B$. Therefore, neither
  sliding nor pivoting change the degree (the latter because it is
  vacuous).  Therefore degree is an invariant. On the other hand,
  (\ref{eqn-braided-traced}) is not degree-preserving; therefore it
  cannot follow from the other axioms.

  (b) The following graphical proof sketch can be turned into an
  algebraic proof:
  \[ \resizebox{3in}{!}{\includegraphics{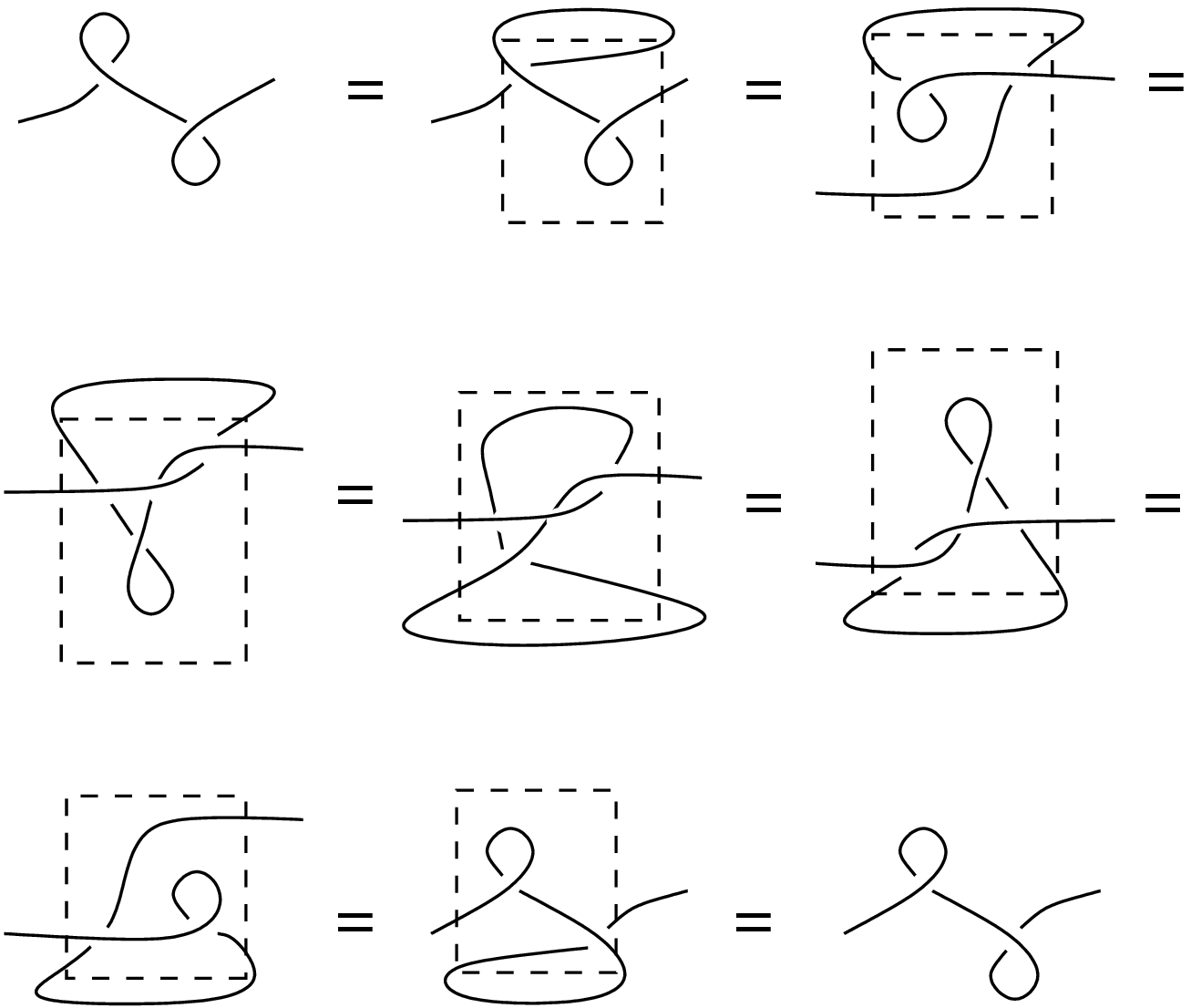}}
  \]

  (c) Here is a proof sketch for the left pivoting axiom. Notably, the
  second to last step uses dinaturality (sliding).
  \[ \resizebox{4in}{!}{\includegraphics{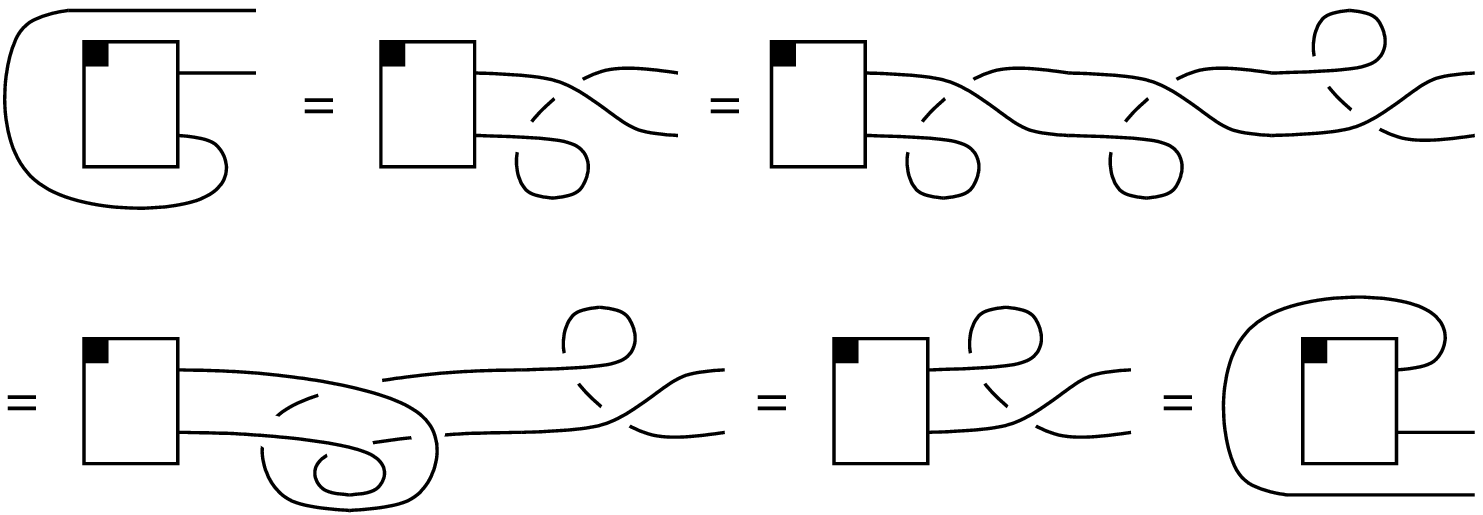}}
  \]
\end{proof}

\begin{remark}
  Each braided traced category possesses a balanced structure (as a
  braided monoidal category) given by $\theta_A=\TrL^A\symi_{A,A}$,
  with inverse $\theta\inv_A=\TrR^A\sym_{A,A}$ (cf.
  (\ref{eqn-braided-traced})). However, this twist is not canonical;
  for example, another twist can be defined by
  $\theta'_A=\TrR^A\sym_{A,A}$ with inverse
  $\theta'_A{}\inv=\TrL^A\symi_{A,A}$ (cf.
  (\ref{eqn-braided-traced-2})). In fact, there are countably many
  other possible twists. This is entirely analogous to
  Remark~\ref{rem-braided-pivotal-other-theta}. The various twists
  coincide if and only if the yanking equation (\ref{eqn-yanking})
  holds, yielding a balanced traced category as discussed in
  Section~\ref{subsec-balanced-traced} below.
\end{remark}

We note that every braided pivotal category is braided traced, with
the traced structure as given in Remark~\ref{rem-pivotal-is-traced}.
Moreover, there is an embedding theorem giving a partial converse:

\begin{theorem}[Representation of braided traced categories]\label{thm-embedding-braided-traced}
  Every braided traced category $\Cc$ can by fully and faithfully
  embedded into a braided pivotal category $\Int(\Cc)$, via a braided
  traced functor.
\end{theorem}

\begin{proof}
  The proof exactly mimics the Int-construction of Joyal, Street, and
  Verity {\cite{JSV96}}, except that we must replace the twist by
  \m{\resizebox{!}{1em}{\includegraphics{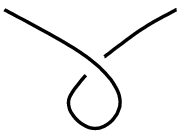}}}, and be careful only
  to use the braided traced axioms. We omit the details, which are
  both lengthy and tedious.\eot
\end{proof}

\begin{remark}
  A braided traced category is not necessarily spherical (and
  therefore not spacial). This is analogous to
  Remark~\ref{braided-not-spherical}.
\end{remark}

\paragraph{Graphical language and coherence.}

The graphical language for braided traced categories is obtained by
adding braids to the graphical language of planar traced categories.
Equivalence of diagrams is up to {\em regular isotopy} (see
Section~\ref{subsec-braided-autonomous}).

\begin{theorem}[Coherence for braided traced categories]
  \label{thm-coherence-braided-traced}
  A well-formed equation between morphisms in the language of braided
  traced categories follows from the axioms of braided traced
  categories if and only if it holds in the graphical language up to
  regular isotopy.
\end{theorem}

\begin{proof}
  Soundness is easy to check by inspection of the axioms.
  Completeness is a consequence of
  Theorems~\ref{thm-coherence-braided-pivotal} and
  {\ref{thm-embedding-braided-traced}}. Namely, consider an
  equation in the language of braided traced categories that holds in
  the graphical language up to regular isotopy.  The diagrams
  corresponding to the left-hand side and right-hand side of the
  equation can also be regarded as diagrams of braided pivotal
  categories, and since they are regularly isotopic, the equation
  holds in all braided pivotal categories by
  Theorem~\ref{thm-coherence-braided-pivotal}. Since any braided
  traced category $\Cc$ can be faithfully embedded in a braided
  pivotal category $\Int(\Cc)$ by
  Theorem~\ref{thm-embedding-braided-traced}, an equation that
  holds in $\Int(\Cc)$ must also hold in $\Cc$. It follows that the
  equation in question holds in all braided traced categories $\Cc$,
  and therefore, it is a consequence of the axioms.\eot
\end{proof}

\begin{caveat}
  Because of the dependence on
  Theorem~\ref{thm-coherence-braided-pivotal},
  Caveat~\ref{cav-coherence-braided-pivotal} also applies here.
\end{caveat}

\subsection{Balanced traced categories}\label{subsec-balanced-traced}

\begin{definition}[\cite{JSV96}]
  A {\em balanced traced category} is a balanced monoidal category
  equipped with a right trace $\Tr$, and satisfying the following {\em
    yanking axioms}:
  \begin{equation}\label{eqn-yanking}
    \Tr^X(\sym_{X,X}) = \theta_X
    \sep\mbox{and}\sep
    \Tr^X(\symi_{X,X}) = \theta\inv_X
  \end{equation}
\end{definition}

\paragraph{Graphical language and coherence.}

The graphical language of balanced traced categories combines the
ribbons and twists of balanced categories with the loops of traced
categories. The trace is represented as expected:
\[ \Tr^X f =
\raisebox{-10mm}{\resizebox{3cm}{!}{\includegraphics{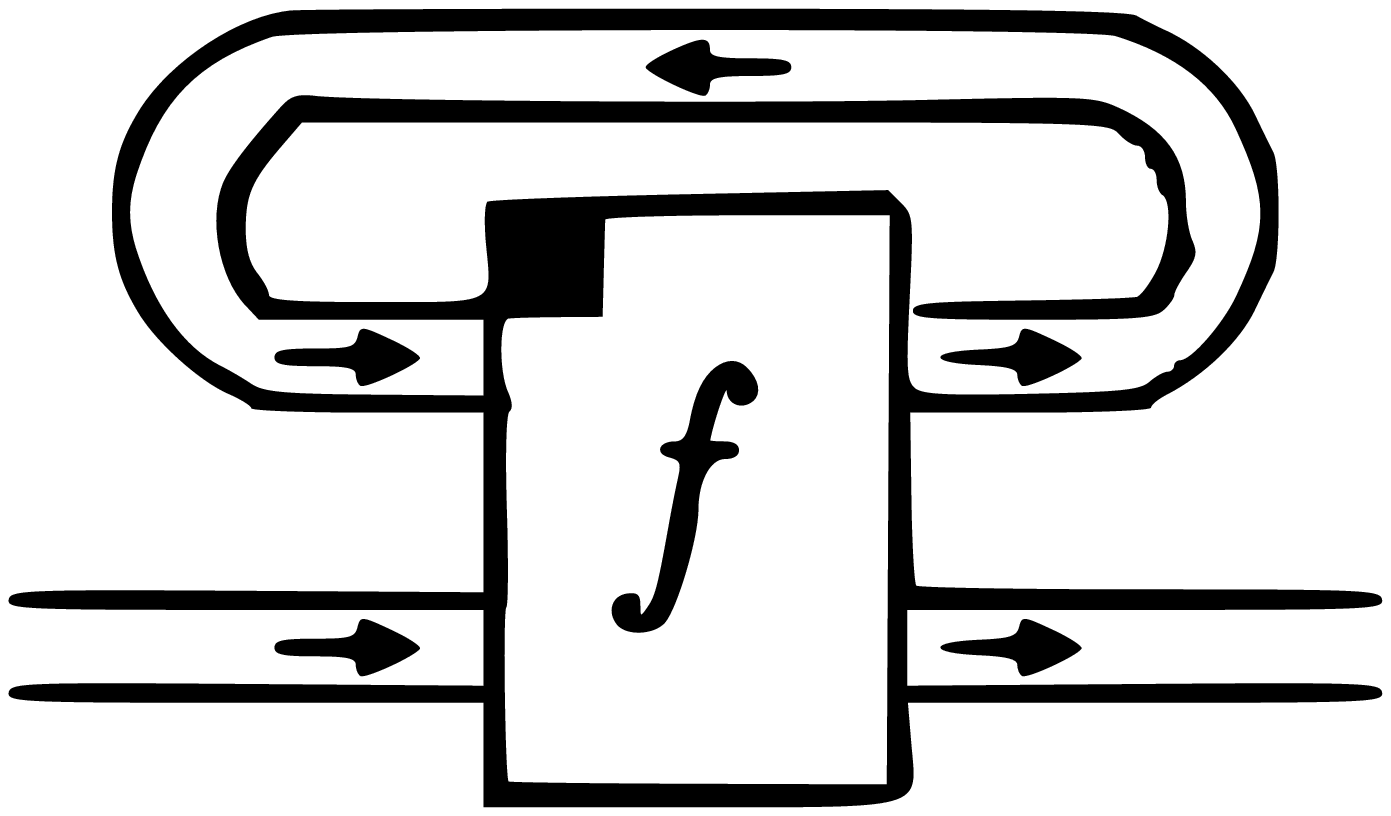}}}.
\]
Note that there is no need to postulate a left trace, because a left
trace is definable from the right trace and braidings as follows:
\[ \TrL^X f =
\raisebox{-10mm}{\resizebox{3cm}{!}{\includegraphics{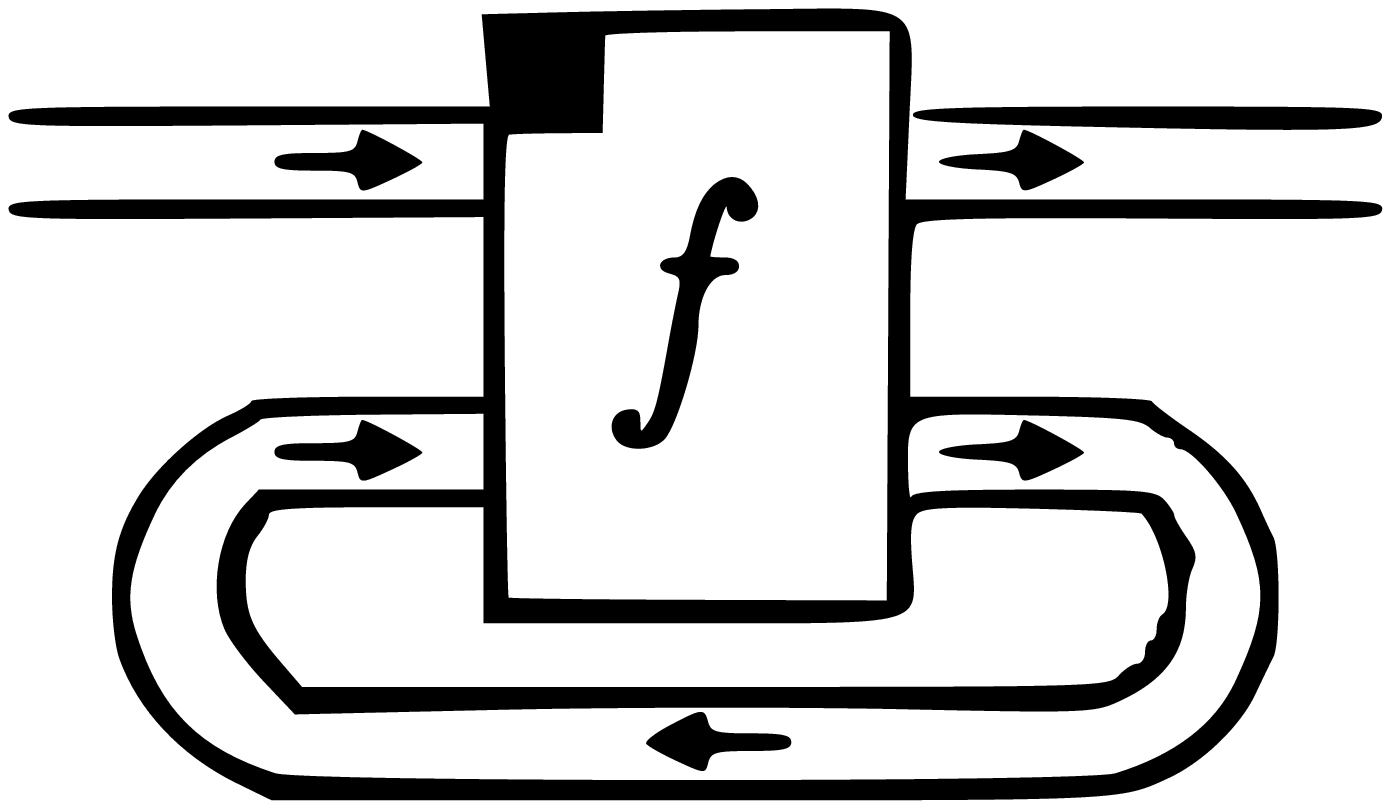}}}
:=
\raisebox{-8mm}{\resizebox{4cm}{!}{\includegraphics{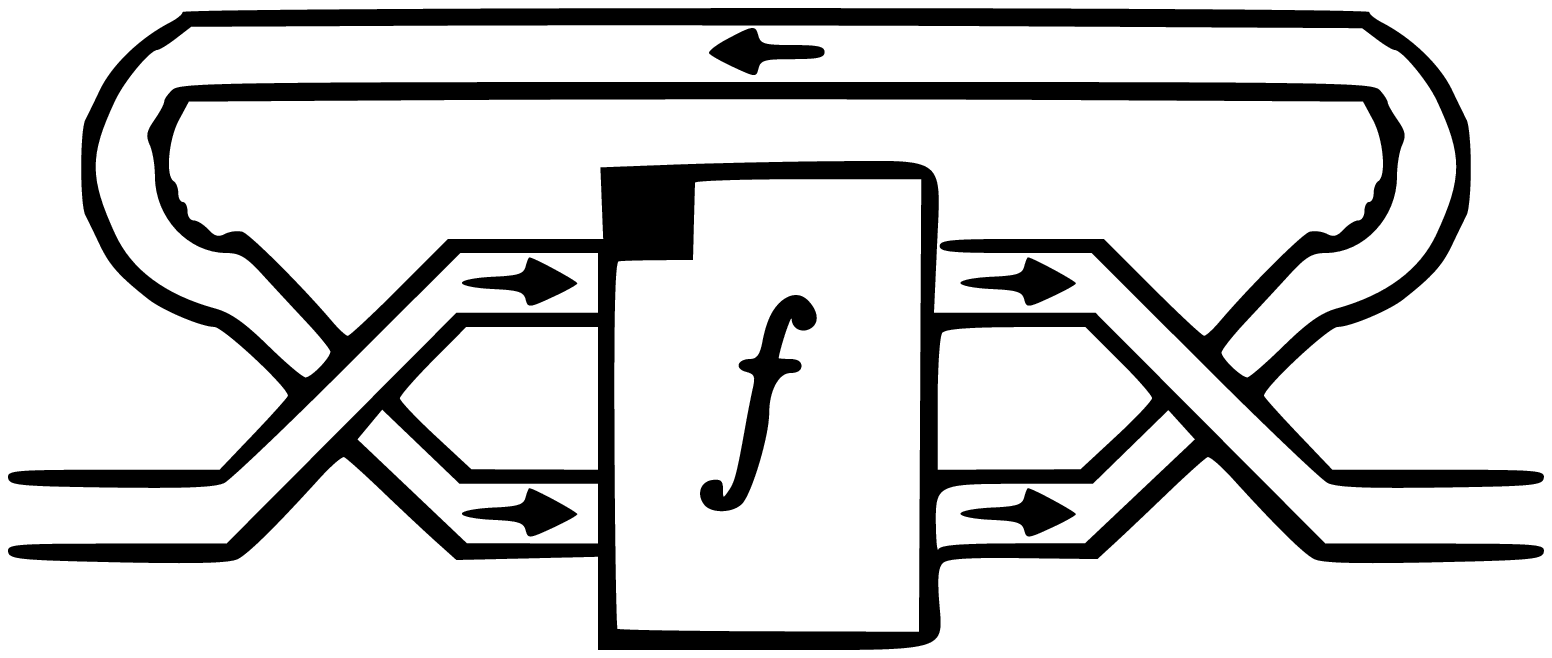}}}
\]

\begin{remark}
  The defined left trace automatically satisfies interchange and the
  pivoting axioms (Table~\ref{tab-planar-traced}), as well as the
  spherical axiom (\ref{eqn-spherical-traced}) and the braided traced
  axiom (\ref{eqn-braided-traced}). The spacial axiom
  (\ref{eqn-spacial}) is satisfied by any braided monoidal category.
  Therefore, any balanced traced category is spacial traced and
  braided traced.
\end{remark}

The graphical validity of the yanking axiom is easily verified using a
shoe string:
\[
\raisebox{-1mm}{\resizebox{1.5cm}{!}{\includegraphics{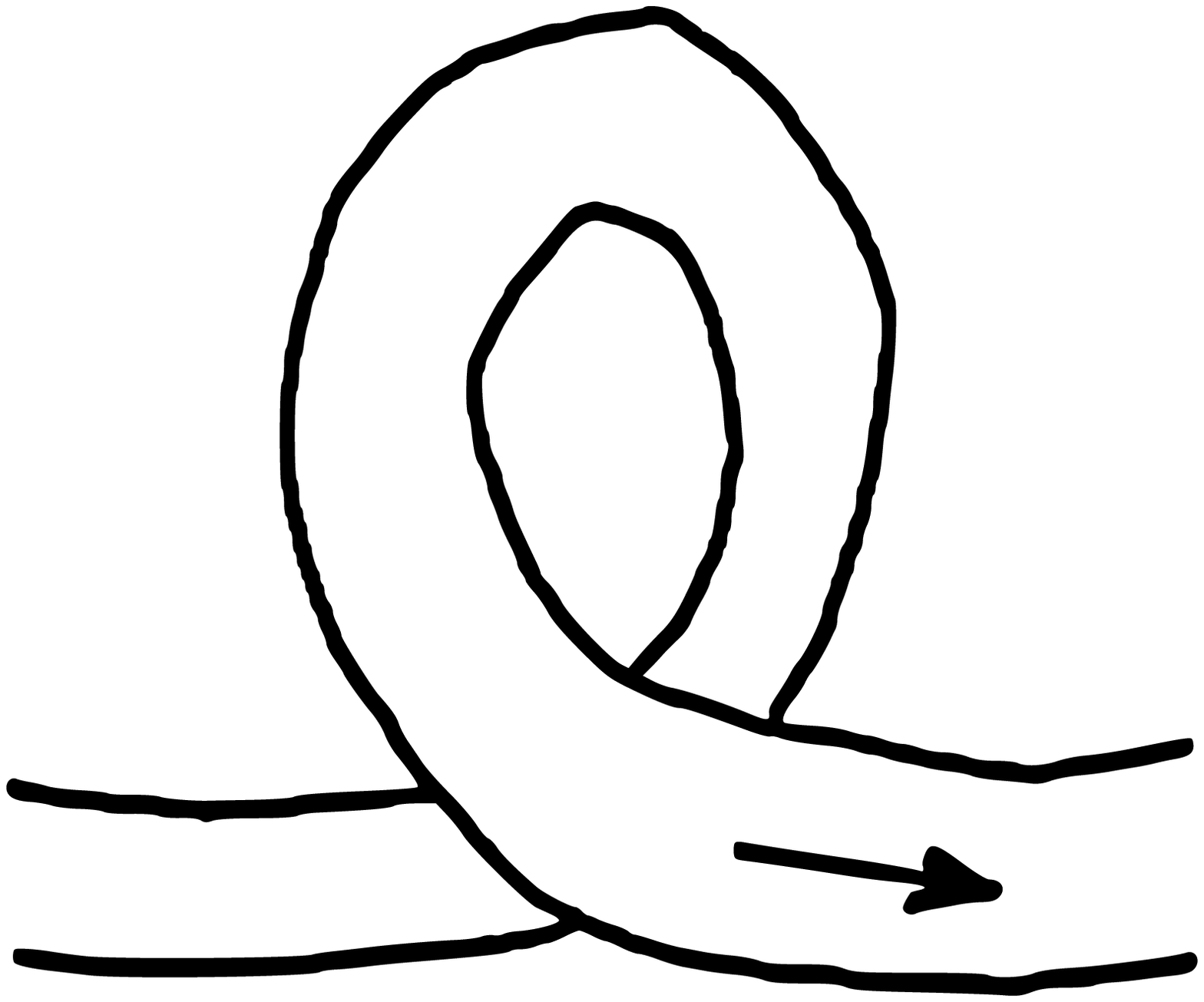}}}
~=
\resizebox{2.7cm}{!}{\includegraphics{twistA}}~,
\sep
\raisebox{-1mm}{\resizebox{1.5cm}{!}{\includegraphics{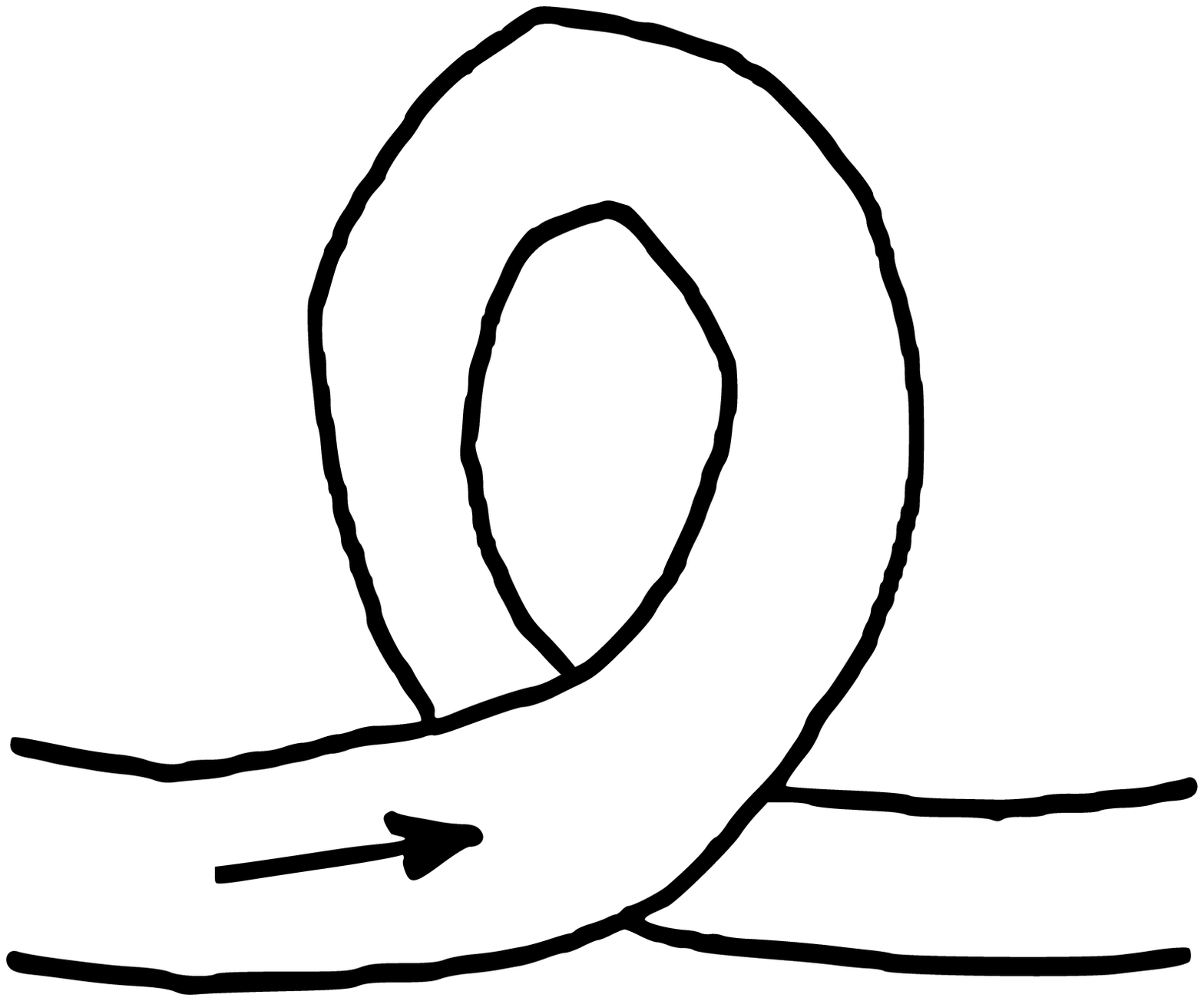}}}
~=
\resizebox{2.7cm}{!}{\includegraphics{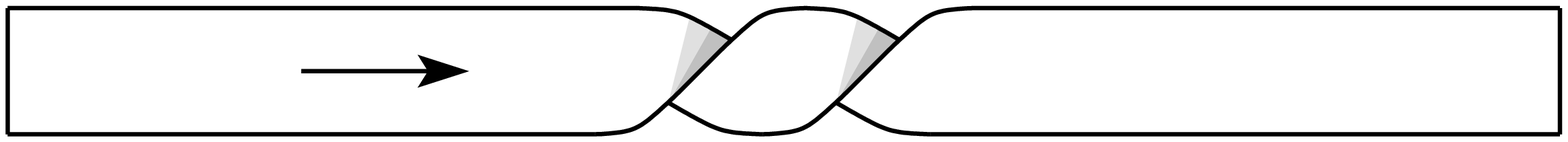}}~.
\]

Every tortile category is balanced traced, with the traced structure as
given in Remark~\ref{rem-pivotal-is-traced}. Moreover, there is an
embedding theorem:

\begin{theorem}[Representation of balanced traced categories {\cite[Prop.~5.1]{JSV96}}]
  \label{thm-embedding-balanced-traced}
  Every balanced traced category can be fully and faithfully embedded
  into a tortile category, via a balanced traced functor.
\end{theorem}

\begin{theorem}[Coherence for balanced traced categories]
  \label{thm-coherence-balanced-traced}
  A well-formed equation between morphisms in the language of balanced
  traced categories follows from the axioms of balanced
  traced categories if and only if it holds in the
  graphical language up to framed isotopy in 3 dimensions.
\end{theorem}

\begin{proof}
  This follows from Theorems~\ref{thm-coherence-tortile} and
  {\ref{thm-embedding-balanced-traced}}, by the exact same
  argument that was used in the proof of
  Theorem~\ref{thm-coherence-braided-traced}.\eot
\end{proof}

\begin{caveat}
  Because of the dependence on Theorem~\ref{thm-coherence-tortile},
  Caveat~\ref{cav-coherence-tortile} also applies here.
\end{caveat}

\begin{remark}
  In any braided monoidal category with a right trace, the twist and
  its inverse are definable by equation (\ref{eqn-yanking}). These
  maps are automatically natural and satisfy $\theta_I=\id_I$ and
  (\ref{eqn-balanced}). However, they are not automatically inverse to
  each other. Therefore, a balanced traced category could be
  equivalently defined as a braided monoidal category with a right
  trace, satisfying
  \[
    \Tr^X(\symi_{X,X}) = \Tr^X(\sym_{X,X})\inv.
  \]
\end{remark}

\subsection{Symmetric traced categories}

\begin{definition}[\cite{8,12,JSV96}]
  A {\em symmetric traced category} is a symmetric monoidal
  category with a right trace $\Tr$, satisfying the {\em symmetric
    yanking axiom}:
  \[ \Tr^X(\sym_{X,X})=\id_X.
  \]
\end{definition}

\begin{remark}
  Because of Remark~\ref{rem-symmetric-from-balanced}, a symmetric
  traced category can be equivalently defined as a balanced
  traced category in which $\theta_A=\id_A$ for all $A$.
\end{remark}

Obviously every compact closed category is symmetric traced with the
structure from Remark~\ref{rem-pivotal-is-traced}. Here, too, we have
an embedding theorem:

\begin{theorem}[Representation of symmetric traced categories \cite{JSV96}]
  \label{thm-embedding-symmetric-traced}
  Every symmetric traced category can be fully and faithfully embedded
  into a compact closed category, via a symmetric traced functor.
\end{theorem}

\begin{example}[\cite{JSV96}]\label{exa-rel-as-traced}
  Consider the category $\Rel$ of sets and relations, with biproducts
  given by disjoint union $A+B$. Given a relation $R:A+X\ii B+X$,
  define its trace $\Tr^X(R):A\ii B$ by $(a,b)\in\Tr^X(R)$ iff there
  exists $n\geq 0$ and $x_1,\ldots,x_n\in X$ such that
  $a\,R\,x_1\,R\,x_2\,R\,\ldots \,R\,x_n\,R\,b$. This defines a
  symmetric traced category which is not compact closed.
\end{example}

\paragraph{Graphical language and coherence.}

The graphical language is like that of planar traced categories,
combined with the symmetry. A typical diagram looks like this:
\[ \resizebox{2.7cm}{!}{\includegraphics{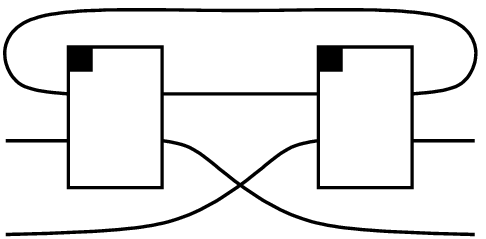}}.
\]
The notion of equivalence of diagrams is isomorphism.

\begin{theorem}[Coherence for symmetric traced categories]
  \label{thm-coherence-symmetric-traced}
  A well-formed equation between morphisms in the language of
  symmetric traced categories follows from the axioms of symmetric
  traced categories if and only if it holds in the graphical language
  up to isomorphism of diagrams.
\end{theorem}

\begin{proof}
  A consequence of Theorems~\ref{thm-coherence-compact-closed} and
  {\ref{thm-embedding-symmetric-traced}}, as in
  Theorems~\ref{thm-coherence-braided-traced} and
  {\ref{thm-coherence-balanced-traced}}.
\end{proof}

\begin{caveat}
  Because of the dependence on
  Theorem~\ref{thm-coherence-compact-closed},
  Caveat~\ref{cav-coherence-compact-closed} also applies here.
\end{caveat}

\begin{remark}
  Strict symmetric traced categories, with the additional axiom
  \begin{equation}\label{eqn-stefanescu}
    \Tr^X(\id_{A\x X})=\id_A,
  \end{equation}
  first appear in the work of {\Stefanescu} under the name ``biflow''. A
  precursor of the definition appears in {\cite{6}}, and the axioms
  were given their modern form by {\Cazanescu} and {\Stefanescu}
  {\cite{12,8}}. The paper {\cite{6}} also contains a detailed
  proof sketch of coherence, namely, that the graphical language,
  modulo isomorphism and the equation (\ref{eqn-stefanescu}), forms
  the free biflow over a monoidal signature. This proof sketch remains
  valid with respect to the modern definition, provided that one
  assumes coherence for symmetric monoidal categories. 
\end{remark}

\section{Products, coproducts, and biproducts}\label{sec-products}

In this section, we consider graphical languages for monoidal
categories where the monoidal structure is given by a categorical
product, coproduct, or biproduct. The main difference with the
graphical languages of ``purely'' monoidal categories from
Sections~\ref{sec-progressive}--\ref{sec-traced} is that equivalence
of diagrams must now be defined up to diagrammatic equations.

\subsection{Products}\label{subsec-products}

\begin{table}
\[ \begin{array}{lll}
  \mbox{Naturality axioms:}\\
  \Delta_B\cp f = (f\x f)\cp\Delta_A : & A\ii B\x B\\
  \Diamond_B\cp f = \Diamond_A : & A\ii I\\[2ex]
  \mbox{Commutative comonoid axioms:}\\
  (\id_A\x \Delta_A)\cp\Delta_A = (\Delta_A\x\id_A)\cp\Delta_A :& A\ii A\x A\x A\\
  (\id_A\x \Diamond_A)\cp\Delta_A = \rho\inv_A :& A\ii A\x I\\
  (\Diamond_A\x\id_A)\cp\Delta_A = \lambda\inv_A :& A\ii I\x A\\
  \sym_{A,A}\cp\Delta_A = \Delta_A :& A\ii A\x A\\[2ex]
  \mbox{Coherence axioms:}\\
  \Delta_I=\lambda\inv_I:&I\ii I\x I\\
  (\id_A\x \sym_{B,A}\x\id_B)\cp\Delta_{A\x B}=\Delta_A\x\Delta_B :& A\x A\ii B\x B\x A\x B \\
  \Diamond_I=\id_I:&I\ii I\\
  \Diamond_{A\x B} = \lambda_I\cp (\Diamond_A\x\Diamond_B):&A\x B\ii I
\end{array}
\]
\caption{The axioms for products}
\label{tab-product-axioms}
\end{table}

\begin{definition}
  In a category, a {\em product} of objects $A$ and $B$ is given by an
  object $A\times B$, together with morphisms $\pi_1:A\times B\ii A$
  and $\pi_2:A\times B\ii B$, such that for all objects $C$ and pairs
  of morphisms $f:C\ii A$ and $g:C\ii B$, there exists a unique
  morphism $h:C\ii A\x B$ such that the following diagram commutes:
  \[\xymatrix{
    &C\ar[dl]_{f}\ar[dr]^{g}\ar@{-->}[d]^{h}\\
    A&A\x B\ar[l]_{\pi_1}\ar[r]^{\pi_2}&B.
    }
  \]
  The unique morphism $h$ is often written as $h=\apair{f,g}$.  An
  object $I$ is {\em terminal} if for all objects $C$, there exists a
  unique morphism $h:C\ii I$. A {\em finite product category} (or {\em
    cartesian category}) is a category with a chosen terminal object,
  and a chosen product for each pair of objects.
\end{definition}

Equivalently, a finite product category can be described as a
symmetric monoidal category, together with natural families of {\em
  copy} and {\em erase} maps
\[ \Delta_A:A\ii A\x A,\sep \Diamond_A:A\ii I
\]
subject to a number of axioms, shown in Table~\ref{tab-product-axioms}.

\paragraph{Graphical language.}

We extend the graphical language of symmetric monoidal categories by
adding the following representations of the copy and erase maps.

\begin{center}
  \begin{tabular}{@{}llc@{}}
    Copy & $\Delta_A:A\ii A\x A$
    & $\vcenter{\wirechart{@C=1cm@R=.4cm}{
        &&*{}\\
        *{}\wireright{r}{A}&
        *{\bullet}\wireright{ur}{A}\wireright{dr}{A}\\
        &&*{}\\
        }}$
    \\\\[-.5ex]
    Erase & $\Diamond_A:A\ii I$
    & $\vcenter{\wirechart{@C=1cm@R=.4cm}{
        *{}\wireright{r}{A}&
        *{\bullet}&
        *{}
        }}$
    \\
  \end{tabular}
\end{center}

As usual, if $A$ is a composite object term, a wire labeled $A$ should
be replaced by multiple parallel wires.
Table~\ref{tab-product-examples} contains graphical representations of
some of the axioms for finite product categories.

\begin{table}
  \[ \begin{array}{c@{\hspace{1cm}}c}
    \m{\resizebox{1.6in}{!}{\includegraphics{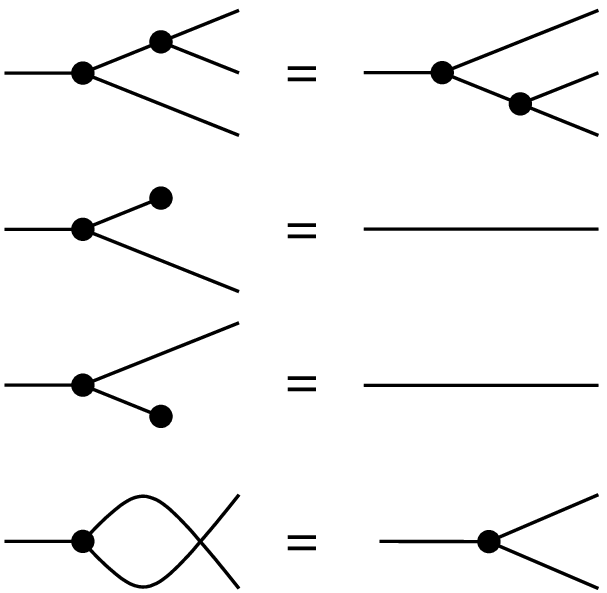}}} &
    \m{\resizebox{1.8in}{!}{\includegraphics{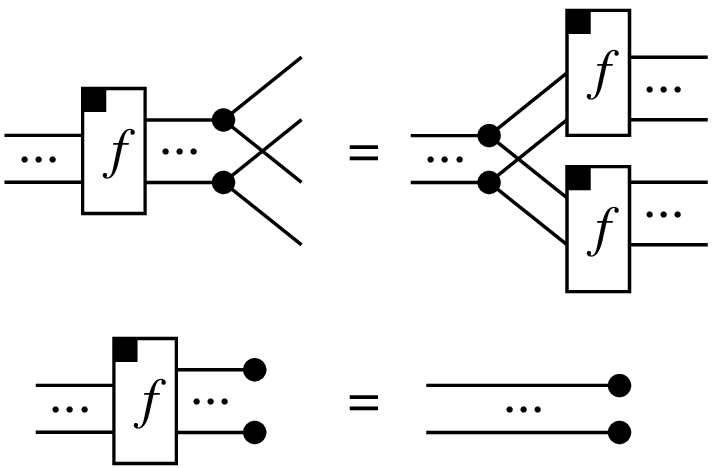}}}
    \\\\[-.5ex]
    \mbox{Commutative comonoid axioms} & \mbox{Naturality}
  \end{array}
  \]
  \caption{Graphical representation of some product axioms}
  \label{tab-product-examples}
\end{table}

Note that the projections $\pi_1:A\times B\ii A$ and $\pi_2:A\times
B\ii B$, and the pairing $h:C\ii A\x B$ of $f:C\ii A$ and $g:C\ii B$,
are represented graphically as follows:
\[ \m{\resizebox{2.8in}{!}{\includegraphics{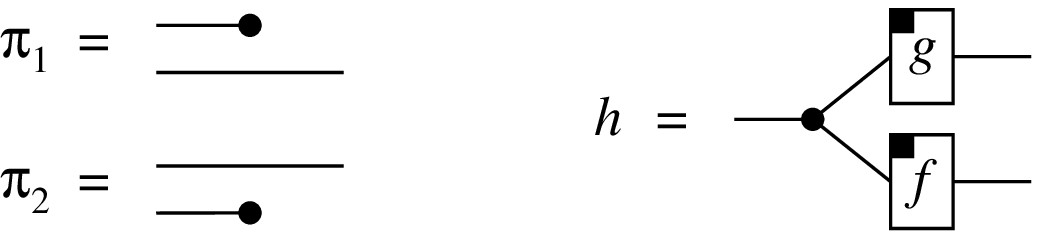}}}
\]

\paragraph{Coherence.}

As the equivalences in Table~\ref{tab-product-examples} demonstrate,
coherence in the graphical language of finite product categories does
not hold up to isomorphism or isotopy of diagrams. Rather, is holds up
to {\em manipulations} of diagrams. So unlike the graphical languages
considered in Sections~\ref{sec-categories}--\ref{sec-traced}, we now
have to consider axioms on diagrams.

\begin{theorem}[Coherence for finite product categories]
  \label{thm-coherence-product}
  A well-formed equation between morphism terms in the language of
  finite product categories follows from the axioms of finite product
  categories if and only if it holds in the graphical language, up to
  isomorphism of diagrams and the diagrammatic manipulations shown
  in Table~\ref{tab-product-examples}.
\end{theorem}

This theorem is a simple consequence of coherence for symmetric
monoidal categories (Theorem~\ref{thm-coherence-symmetric-monoidal}),
together with the fact that all the axioms of finite product
categories, except those shown in Table~\ref{tab-product-examples},
hold up to isomorphism of diagrams.

\subsection{Coproducts}

The definition of coproducts and initial objects is dual to that of
products and terminal objects, i.e., it is obtained by reversing all
the arrows in Section~\ref{subsec-products}. Explicitly, an object $0$
is {\em initial} if for all objects $C$, there exists a unique
morphism $h:0\ii C$. A {\em coproduct} of objects $A,B$ is given by an
object $A+B$, together with morphisms $\iota_1:A\ii A+B$ and
$\iota_2:B\ii A+B$, such that for all objects $C$ and pairs of
morphisms $f:A\ii C$ and $g:B\ii C$, there exists a unique morphism
$h:A+B\ii C$ such that $h\cp\iota_1=f$ and $h\cp\iota_2=g$. One often
writes $h=[f,g]$. A category with finite coproducts is also called a
{\em co-cartesian category}.

Dualizing the presentation of Section~\ref{subsec-products}, one can
equivalently define a finite coproduct category as a symmetric
monoidal category with natural families of {\em merge} and {\em
  initial} maps
\[ \nabla_A:A\x A\ii A,\sep \Box_A:I\ii A,
\]
satisfying the duals of the axioms in Table~\ref{tab-product-axioms}.

\paragraph{Graphical language.}
The graphical language of finite coproduct categories is obtained by
dualizing that of finite product categories, with the duals of the
axioms from Table~\ref{tab-product-examples}.

\begin{center}
  \begin{tabular}{@{}llc@{}}
    Merge & $\nabla_A:A\x A\ii A$
    & $\vcenter{\wirechart{@C=1cm@R=.4cm}{
        *{}\wireright{dr}{A}&&\\
        &*{\bullet}\wireright{r}{A}&
        *{}\\
        *{}\wireright{ur}{A}&&\\
        }}$
    \\\\[-.5ex]
    Initial & $\Box_A:I\ii A$
    & $\vcenter{\wirechart{@C=1cm@R=.4cm}{
        *{}&
        *{\bullet}\wireright{r}{A}&
        *{}
        }}$
    \\
  \end{tabular}
\end{center}

\subsection{Biproducts}\label{subsec-biproduct}

\begin{definition}
  An object is called a {\defit zero object} if it is initial and
  terminal. If $\zero$ is a zero object, then there is a distinguished
  map $A\ii \zero\ii B$ between any two objects, denoted $0_{A,B}$.  A
  {\em biproduct} of objects $A_1$ and $A_2$ is given by an object
  $A_1\oplus A_2$, together with morphisms $\iota_i:A_i\ii A_1\oplus
  A_2$ and $\pi_i:A_1\oplus A_2\ii A_i$, for $i=1,2$, such that
  $A\oplus B$ is a product with $\pi_1,\pi_2$, a coproduct with
  $\iota_1,\iota_2$ and such that $\pi_i\cp \iota_j=\delta_{ij}$.
  Here $\delta_{ii}=\id_{A_i}$ and $\delta_{ij}=0_{A_j,A_i}$ when
  $i\neq j$.  We say that $\Cc$ is a {\em biproduct category}
  if it has a chosen zero object $\zero$ and a chosen biproduct for
  any pair of objects.
\end{definition}

\begin{remark}
  The axiom $\pi_i\cp \iota_j=\delta_{ij}$ is equivalent to the
  assertion that the symmetric monoidal structure defined by $\oplus$
  ``as a product'' coincides with the symmetric monoidal structure
  defined by $\oplus$ ``as a coproduct''. Therefore, a 
  biproduct category is symmetric monoidal in a canonical way.
\end{remark}

Equivalently, a biproduct category can be defined as a
symmetric monoidal category, together with natural families of
morphisms
\[ \Delta_A:A\ii A\x A,\sep \Diamond_A:A\ii I,\sep
\nabla_A:A\x A\ii A,\sep \Box_A:I\ii A,
\]
satisfying the axioms in Table~\ref{tab-product-axioms}, as well as
their duals.

\paragraph{Graphical language.}

The graphical language of biproducts is obtained by combining the
graphical languages for products and coproducts. In this case, one has
the equalities in Table~\ref{tab-biproduct-naturality}, which are
consequences of the naturality axioms in
Table~\ref{tab-product-examples}. Note that the axiom
$\pi_i\cp\iota_j=\delta_{ij}$ holds automatically in the graphical
language.
\begin{table}
  \[
  \m{\resizebox{3.7in}{!}{\includegraphics{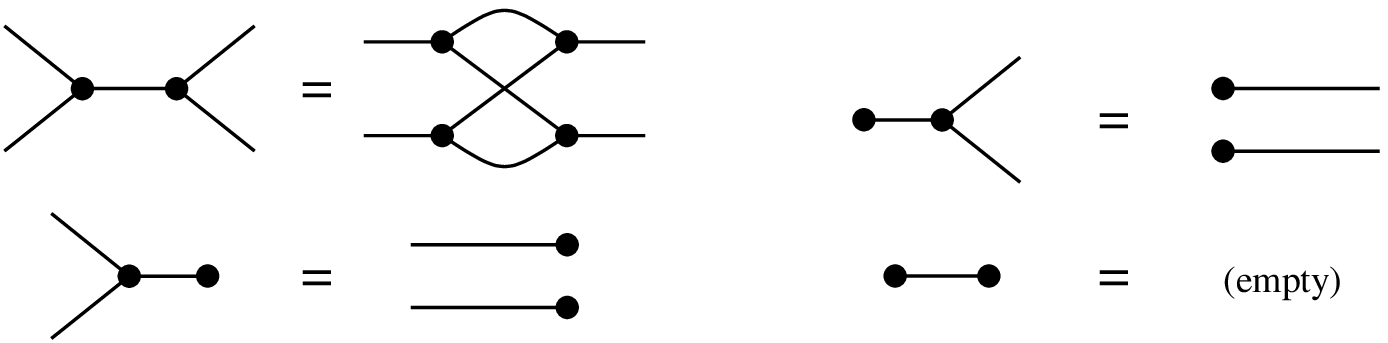}}}
  \]
  \caption{Some biproduct laws}
  \label{tab-biproduct-naturality}
\end{table}

\begin{theorem}[Coherence for biproduct categories]
  A well-formed equation between morphism terms in the language of
  biproduct categories follows from the axioms of biproduct
  categories if and only if it holds in the graphical language, up to
  isomorphism of diagrams, the diagrammatic manipulations shown
  in Table~\ref{tab-product-examples}, and their duals.
\end{theorem}

This theorem is a simple consequence of coherence for symmetric
monoidal categories, together with the fact that the axioms in
Table~\ref{tab-product-examples} (and their duals) are exactly the
graphical representations of the axioms in
Table~\ref{tab-product-axioms} (and their duals) that do not already
hold up to graphical isomorphism.

\subsection{Traced product, coproduct, and biproduct categories}

It potentially makes sense to revisit each of the notions of
Sections~\ref{sec-progressive}--{\ref{sec-traced}} and consider the
case where the monoidal structure is given by a product, coproduct, or
biproduct. Since products, coproducts, and biproducts are
automatically symmetric, we do not need to consider the weaker notions
(such as balanced, braided, etc).

Moreover, we do not need to consider any autonomous cases, because an
autonomous category where the tensor is given by a product (or
coproduct) is trivial. Indeed, for any objects $A,B$, the morphisms
$f:A\ii B$ are in one-to-one correspondence with morphism $A\x B^*\ii
I$. Since $I$ is terminal, there is exactly one such morphism, and
therefore there is a unique morphism between any two objects. Such a
category is equivalent to the one-object one-morphism category.

Therefore, the only new notion from
Sections~\ref{sec-progressive}--\ref{sec-traced} that admits
non-trivial examples in the context of products, coproducts, or
biproducts is that of a symmetric traced category.

\begin{definition}
  A {\em traced product [coproduct, biproduct] category} is a
  symmetric traced category where the tensor is given by a categorical
  product [coproduct, biproduct].
\end{definition}

\begin{example}[\cite{JSV96}]\label{exa-rel-as-traced-biprod}
  The symmetric traced category $(\Rel,+)$ from
  Example~\ref{exa-rel-as-traced} is a traced biproduct category.
\end{example}

\begin{example}
  Consider the category $\Set_{\bot}$ whose objects are sets, and
  whose morphisms are partial functions, regarded as a subcategory of
  $\Rel$ from Example~\ref{exa-rel-as-traced-biprod}. In this category, the
  empty set $0$ is a zero object, and the disjoint union operation
  $A+B$ defines a coproduct (but not a product). Trace is given as in
  Example~\ref{exa-rel-as-traced-biprod}. With these definitions,
  $\Set_{\bot}$ is a traced coproduct category.
\end{example}

\paragraph{Graphical language.}

As expected, the graphical language of traced product [coproduct,
biproduct] categories is given by adding a trace (as in
Section~\ref{sec-traced}) to the graphical language of finite product
[finite coproduct, biproduct] categories.

\begin{theorem}[Coherence for traced product {[coproduct, biproduct]} categories]
  A well-formed equation between morphism terms in the language of
  traced product [coproduct, biproduct] categories follows from the
  respective axioms if and only if it holds in the graphical language,
  up to isomorphism of diagrams, and the diagrammatic manipulations
  shown in Table~\ref{tab-product-examples} and/or their duals (as
  appropriate).
\end{theorem}

\begin{remark}\label{rem-while-loop}
  In computer science, traces arise naturally in the context of {\em
    data flow} (as fixed points), and in the context of {\em control
    flow} (as iteration). The two situations correspond to traced
  product categories and traced coproduct categories, respectively.
  The duality between data flow and control flow was first described
  by Bainbridge {\cite{Bai76}}.  The following are typical examples of
  a data flow diagram (on the left) and a control flow diagram (on the
  right). The data flow diagram represents the fixed point expression
  $y = (3+x)(x+y)$, parametric on an input $x$.  The control flow
  diagram represents a generic ``while loop''. Note that data flow
  diagrams have a notion of ``copying'' data, whereas control flow
  diagrams have a dual notion of ``merging'' control paths.
  \[
  \m{\resizebox{!}{.6in}{\includegraphics{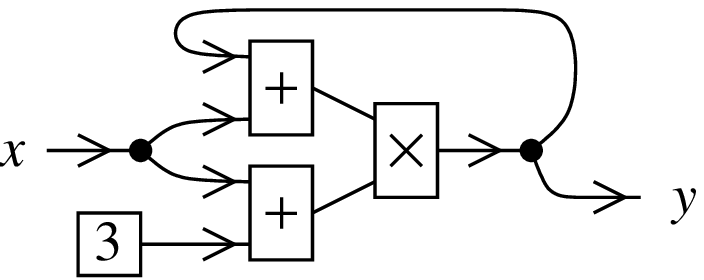}}}
  \sep
  \m{\resizebox{!}{.5in}{\includegraphics{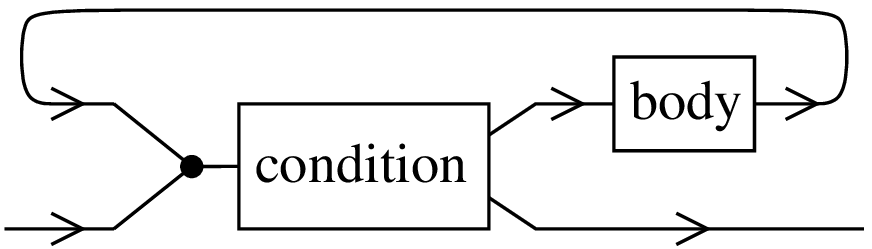}}}
  \]
\end{remark}

\begin{proposition}[{\Cazanescu} and {\Stefanescu} {\cite{12,8}}]
  \label{prop-iteration}
  In a category with finite coproducts, giving a trace is equivalent
  to giving an {\em iteration operator}. Here, an iteration operator
  is a family of operations
  \[ \iter^X : \hom(X,A+X) \ii \hom(X,A),
  \]
  natural in $A$ and dinatural in $X$, satisfying
  \begin{enumerate}
  \item Iteration: 
    $\iter(f) = [\id_A,\iter(f)]\cp f$, for all $f:X\ii A+X$;
  \item Diagonal property: 
    $\iter(\iter(f)) = \iter((\id_A+[\id_X,\id_X])\cp f)$, 
    for all $f:X\ii A+X+X$.
  \end{enumerate}
  Dually, on a finite product category, giving a trace is equivalent
  to a {\em fixed point operator} $\fix^X : \hom(A\times X,X) \ii
  \hom(A,X)$.
\end{proposition}

This makes precise the intuitive idea that in the presence of
coproducts, the while loop in Remark~\ref{rem-while-loop} is
sufficient for constructing arbitrary traces.

\begin{remark}
  In the presence of the other axioms, the diagonal property is
  equivalent to the so-called Beki\v{c} Lemma:
  \[ \iter[f,g] = [\id_A,\iter([\id_{A+X},\iter(g)]\cp f)]\cp[\inj{2},\iter(g)],
  \] 
  for all $f:X\ii A+X+Y$ and $g:Y\ii A+X+Y$ {\cite[Prop.~B.1]{2a}}.
\end{remark}

\begin{remark}
  Iteration operators in the sense of Proposition~\ref{prop-iteration}
  were first defined, using different but equivalent axioms, by
  {\Cazanescu} and Ungureanu {\cite{22,25}}, under the name
  ``algebraic theory with iterate''. 
\end{remark}

\begin{proposition}[{\cite{8}}]
  In a category with finite biproducts, giving a trace is equivalent
  to giving a {\em repetition operation}, i.e., a family of operators
  \[ *:\hom(A,A)\ii\hom(A,A)
  \]
  satisfying
    \begin{enumerate}
  \item $f^* = \id + ff^*$,
  \item $(f+g)^* = (f^*g)^*f^*$.
  \item $(fg)^*f = f(gf)^*$ (dinaturality).
  \end{enumerate}
  Here, $f+g$ denotes the morphism $\nabla_A\cp(f\oplus
  g)\cp\Delta_A:A\ii A$, for $f,g:A\ii A$.
\end{proposition}

\subsection{Uniformity and regular trees}

\begin{definition}
  Suppose we are given a traced category with a distinguished subclass
  of morphisms called the {\em strict} morphisms. Then the trace is
  called {\em uniform} if for all $f:A\x X\ii B\x X$, $g:A\x Y\ii B\x
  Y$, and strict $h:X\ii Y$, the following implication holds:
  \[ (\id_B\x h)\cp f = g \cp (\id_A\x h) \ssep\imp\ssep
  \Tr^X(f) = \Tr^Y(g).
  \]
  Equivalently, in pictures:
  \[
  \m{\resizebox{!}{.4in}{\includegraphics{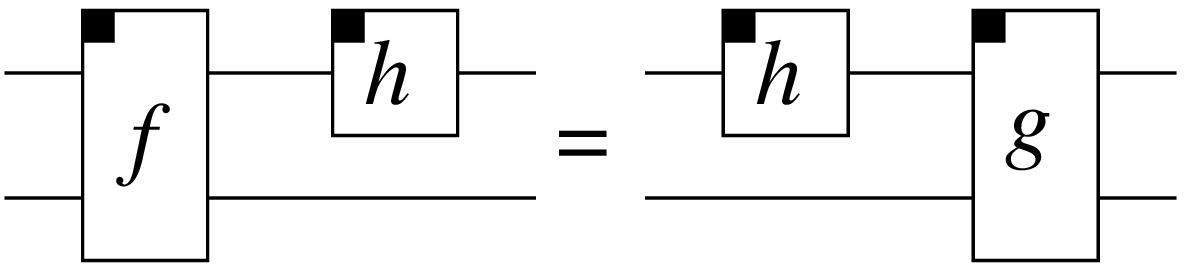}}}
  \sep\imp\sep \m{\resizebox{!}{.4in}{\includegraphics{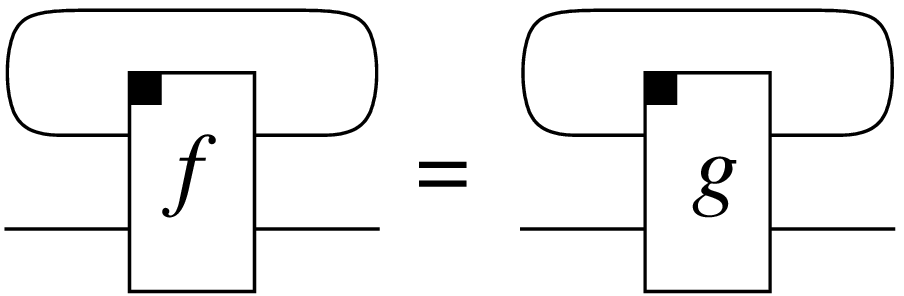}}}\sep.
  \]
  whenever $h$ is strict.  Note that uniformity is not an equational
  property.
\end{definition}

\begin{proposition}[{\cite{8}}]
  A traced coproduct category is uniformly traced if and only if for
  all $f:X\ii A+X$, $g:Y\ii A+Y$, and strict $h:X\ii Y$,
  \[ (\id_A+h)\cp f = g\cp h 
  \ssep\imp\ssep
  \iter^X(f)=\iter^Y(g)\cp h.
  \]
  Moreover, a traced biproduct category is uniformly traced if and
  only if for all $f:X\ii X$, $g:Y\ii Y$, and strict $h:X\ii Y$,
  \[ h\cp f=g\cp h 
  \ssep\imp\ssep
  h\cp f^*=g^*\cp h.
  \]
\end{proposition}

In the particular case where the class of strict morphisms is taken to
be the smallest co-cartesian subcategory containing all objects,
{\Stefanescu} {\cite{2a,5}} proved that the free uniformly
traced coproduct category over a monoidal signature is given by the
graphical language of traced coproduct categories, modulo a suitable
notion of simulation equivalence on diagrams. This simulation
equivalence is easiest to describe in the case where all morphism
variables are of input arity 1. In this case, two diagrams are
simulation equivalent if and only if they have the same infinite tree
unwinding. There is also an analogous result for biproducts. We refer
the reader to {\cite{2a,2b,4}} for full details.

The following is an example of an equation that holds up to infinite
tree unwinding, but fails in general traced coproduct categories:
\begin{equation}\label{eqn-tree}
\m{\resizebox{!}{.5in}{\includegraphics{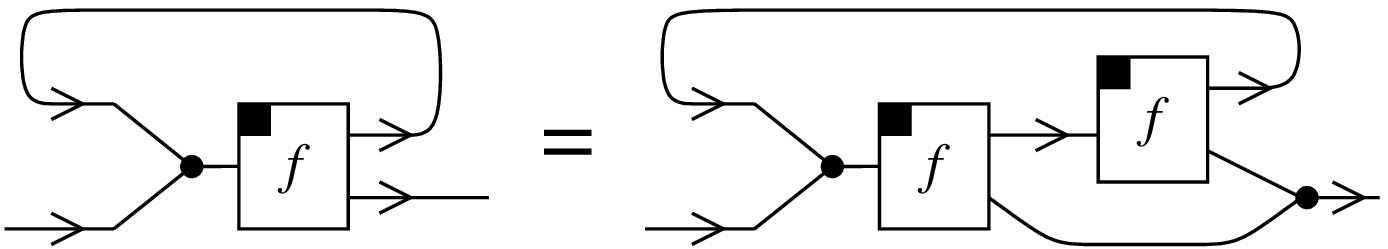}}}
\end{equation}
{\'E}sik's ``iteration theories'' {\cite{15}} are a direct equational
axiomatization of such infinite tree unwindings. They include an
iteration operator as in Proposition~\ref{prop-iteration}, but with an
infinite family of additional properties, such as (\ref{eqn-tree}).

\subsection{Cartesian center}

Sometimes it is useful to consider notions that are weaker than
product categories, yet still have copy and erase maps $\Delta_A:A\ii
A\x A$ and $\Diamond_A:A\ii I$. For example, it is common to drop the
naturality axioms, while retaining the commutative comonoid and
coherence axioms (see Tables~\ref{tab-product-axioms} and
{\ref{tab-product-examples}}).  An equivalent way to describe such a
category is as a symmetric monoidal category with (faithful) {\em
  cartesian center} {\cite{Has97}}, i.e., a symmetric monoidal
category with a symmetric monoidal subcategory that contains all the
objects and is cartesian. Similar ideas have occurred, with varying
degrees of explicitness, in the literature on flowcharts, see e.g.
{\cite{22,10,11}}.

Similarly, if one omits naturality from the axioms for coproducts, one
obtains categories with a co-cartesian center. A weakened version of
biproducts is obtained by combining the axioms of cartesian center and
co-cartesian center. In this case, one requires the operations
$\Delta$, $\Diamond$, $\nabla$, $\Box$ to be natural with respect to
one another, yielding the properties from
Table~\ref{tab-biproduct-naturality}. More generally, one may require
any subset of the operations $\Delta$, $\Diamond$, $\nabla$, $\Box$ to
exist, and a further subset to be natural transformations. As the
reader may imagine, this leads to a nearly endless number of
categorical notions and corresponding graphical languages; see
e.g.~{\cite{11,4}}.

\section{Dagger categories}

The concept of a dagger category (also called {\em involutive
  category} or {\em *-category} in the literature) is motivated by the
category of Hilbert spaces, where each morphism $f:A\ii B$ has an {\em
  adjoint} $f\da:B\ii A$.

\begin{definition}
  A {\em dagger category} is a category $\Cc$ together with an
  involutive, identity-on-objects, contravariant functor
  $\dagger:\Cc\ii\Cc$.
\end{definition}

Concretely, this means that to every morphism $f:A\ii B$, one
associates a morphism $f\da:B\ii A$, called the {\defit adjoint} of
$f$, such that for all $f:A\ii B$ and $g:B\ii C$:
\[ \begin{array}{l@{~}l}
    \id_A\da = \id_A &: A\ii A, \\
    (g\cp f)\da = f\da\cp g\da &: C\ii A, \\
    f\dada = f&:A\ii B, \\
  \end{array}
\]

\begin{example}
  The category $\Hilb$ of Hilbert spaces and bounded linear maps is a
  dagger category, where $f\da:B\ii A$ is given by the usual
  adjointness property of linear algebra, i.e., $\iprod{f\da
    x}{y}=\iprod{x}{fy}$ for all $x\in B$ and $y\in A$.
\end{example}

\begin{definition}{\bf (Unitary map, self-adjoint map)}
  In a dagger category, a morphism $f:A\ii B$ is called {\defit unitary}
  if it is an isomorphism and $f^{-1}=f\da$. A morphism $f:A\ii A$ is
  called {\defit self-adjoint} or {\defit hermitian} if $f=f\da$.
\end{definition}

A {\em dagger functor} between dagger categories is a functor that
satisfies $F(f\da)=(Ff)\da$ for all $f$.

\paragraph{Graphical language.}

The graphical language of dagger categories extends that of
categories. The adjoint of a morphism variable $f:A\ii B$ is
represented diagrammatically as follows:
\begin{center}
  \begin{tabular}{@{}llc@{}}
    & $f:A\ii B$ &
    $\wirechart{@C=1.7cm@R=0.5cm}{\wireright{r}{A}&\blank\ulbox{[]}{f}\wireright{r}{B}&
      }$\\\\
    & $f\da:B\ii A$ &
    $\wirechart{@C=1.7cm@R=0.5cm}{\wireright{r}{B}&\blank\urbox{[]}{f}\wireright{r}{A}&
      }$\\\\
  \end{tabular}
\end{center}

More generally, the adjoint of any diagram is its mirror image. Note
that the mirror image of a box is visually distinguishable because we
have marked the upper left corner of each box representing a morphism
variable.  Also note that, while we have taken the mirror image of
each box, we have reversed the location, but not the direction, of the
wires. Contrast this with (\ref{eqn-adjoint-mate}).

\begin{theorem}[Coherence for dagger categories]
  \label{thm-coherence-dagger-categories}
  A well-formed equation between two morphism terms in the language of
  dagger categories follows from the axioms of dagger categories if
  and only if it holds in the graphical language up to isomorphism of
  diagrams.
\end{theorem}

\begin{proof}
  This is a consequence of coherence for categories, from
  Theorem~\ref{thm-coherence-categories}. As usual, soundness is easy
  to check. For completeness, notice that any morphism term $t$ of
  dagger categories can be transformed, via the axioms $(g\cp
  f)\da=f\da\cp g\da$, $\id\da = \id$, and $f\dada = f$, into an
  equivalent term $t'$ with the property that $\dagger$ is only
  applied to morphism variables in $t'$. Such a term can be regarded
  as a term in the language of categories, over the extended set of
  morphism variables $\s{f,f\da,\ldots}$. Now if $t$ and $s$ are two
  terms that have isomorphic diagrams, then by soundness, $t'$ and
  $s'$ have isomorphic diagrams. By
  Theorem~\ref{thm-coherence-categories}, $t'$ and $s'$ are provably
  equal from the axioms of categories. Therefore $t$ and $s$ are
  provably equal from the axioms of dagger categories.\eot
\end{proof}

We now consider ``dagger notions'' for the various monoidal categories
from Sections~\ref{sec-progressive}--\ref{sec-traced}.

\subsection{Dagger monoidal categories}

\begin{definition}
  A {\em dagger monoidal category} is a monoidal category that is a
  dagger category, such that the dagger structure is compatible with
  the monoidal structure in the following sense:
  \begin{enumerate}\alphalabels
  \item $(f\x g)\da = f\da\x g\da$, for all $f,g$;
  \item the canonical isomorphisms of the monoidal structure,
    $\alpha_{A,B,C}: (A\x B)\x C\ii A\x(B\x C)$, $\lambda_A:I\x A\ii
    A$, and $\rho_A: A\x I\ii A$, are unitary.
  \end{enumerate}
\end{definition}

\paragraph{Graphical language.}

The graphical language of dagger monoidal categories is like the
graphical language of monoidal categories, with the adjoint of a
diagram given by its mirror image. For example,
\[ \left[\mnew{\wirechart{@R+1ex@C+1em}{
      *{}\wireright{r}{A}&\blank\ulbox{[].[d]}{g}\wire{rr}{}&\blank\wireright{r}{E}&*{}\\
      *{}\wireright{r}{B}&\blank\wireright{r}{D}&\blank\urbox{[].[d]}{g}\wireright{r}{F}&*{}\\
      *{}\wireright{r}{C}\wire{rr}{}&\blank&\blank\wireright{r}{G}&*{}\\
      }}\right]\largeda
\ssep=\ssep
\mnew{\wirechart{@R+1ex@C+1em}{
    *{}\wireright{r}{E}\wire{rr}{}&\blank&\blank\wireright{r}{A}&*{}\\
    *{}\wireright{r}{F}&\blank\wireright{r}{D}&\blank\urbox{[].[u]}{f}\wireright{r}{B}&*{}\\
    *{}\wireright{r}{G}&\blank\ulbox{[].[u]}{g}\wire{rr}{}&\blank\wireright{r}{C}&*{}\\
    }}
\]

\begin{theorem}[Coherence for planar dagger monoidal categories]
  \label{thm-coherence-dagger-monoidal}
  A well-formed equation between morphism terms in the language of
  dagger monoidal categories follows from the axioms of dagger
  monoidal categories if and only if it holds, up to planar isotopy,
  in the graphical language.
\end{theorem}

\begin{proof}
  This is a consequence of coherence for planar monoidal categories,
  from Theorem~\ref{thm-coherence-planar}. The proof is analogous to
  that of Theorem~\ref{thm-coherence-dagger-categories}. Note that the
  axioms of dagger monoidal categories are precisely what is needed to
  ensure that all occurrences of $\dagger$ can be removed from a
  morphism term, except where applied directly to a morphism
  variable.\eot
\end{proof}

\subsection{Other progressive dagger monoidal notions}

We can now ``daggerize'' the other progressive monoidal notions from
Section~\ref{sec-progressive}:

\begin{definition}
  \begin{itemize}
  \item A dagger monoidal category is {\em spacial} if it is spacial
    as a monoidal category.
  \item A {\em dagger braided monoidal category} is a dagger monoidal
    category with a unitary braiding $\sym_{A,B}:A\x B\ii B\x A$.
  \item A {\em dagger balanced monoidal category} is a dagger braided
    monoidal category with a unitary twist $\theta_A:A\ii A$.
  \item A {\em dagger symmetric monoidal category} {\cite{Sel05}} is a
    dagger braided monoidal category such that the unitary braiding is
    a symmetry.
  \end{itemize}
\end{definition}

\paragraph{Graphical languages.}

In each case, the graphical language extends the corresponding
language from Section~\ref{sec-progressive}, with the dagger of a
diagram taken to be its mirror image. Each notion has a coherence
theorem, proved by the same method as
Theorems~\ref{thm-coherence-dagger-categories}
and~\ref{thm-coherence-dagger-monoidal}. The requirements that the
braiding and twist are unitary ensures that the dagger can be removed
from the corresponding terms. The respective caveats from
Section~\ref{sec-progressive} also apply to the dagger cases.

\begin{example}
  The category $\Hilb$ of Hilbert spaces is dagger symmetric monoidal,
  with the usual tensor product and symmetry.
\end{example}

\subsection{Dagger pivotal categories}

In defining dagger variants of the notions of
Section~\ref{sec-autonomous}, we find that the notion of a dagger
autonomous category and a dagger pivotal category coincide. This is
because the presence of a dagger structure on an autonomous category
already induces a canonical isomorphism $A\iso A^{**}$, which
automatically satisfies the pivotal axioms under mild assumptions.

To be more precise, consider a dagger monoidal category that is also
right autonomous (as a monoidal category). Because $\eta_A:I\ii A^*\x
A$ has an adjoint $\eta\da_A:A^*\x A\ii I$, we can define a family of
isomorphisms
\[ i_A = A \catarrow{\iso} I\x A\catarrow{\eta_{A^*}\x \id_A} A^{**}\x A^*\x A \catarrow{\id_{A^{**}}\x \eta\da_A}A^{**}\x I\catarrow{\iso} A^{**}.
\]
We can represent this schematically as follows (but bearing in mind
that we do not yet have a formal graphical language to work with):
\begin{equation}\label{eqn-definition-i}
  \wirechart{}{\wire{r}{A}&\circbox{i_A}\wire{r}{A^{**}}&}
  \ssep=\ssep
  \raisebox{1ex}{\mnew{\wirechart{@C=1.0cm@R=0.5cm}{
        \wire{r}{A}&\blank\wirecloselabel{d}{\eta\da_A}\\
        \blank\wire{r}{A^*}&\blank\\
        \blank\wireopenlabel{u}{\eta_{A^*}}\wire{r}{A^{**}}&
        }}}
\end{equation}

\begin{lemma}\label{lem-dagger-pivotal-1}
  The following are equivalent in a right autonomous, dagger monoidal
  category:
  \begin{itemize}\alphalabels
  \item the family of isomorphisms $i_A:A\ii A^{**}$, as defined
    above, determines a pivotal structure;
  \item for all $A,B$, the canonical isomorphisms $(A\x B)^*\iso B^*\x
    A^*$ and $I^*\iso I$ (determined by the right autonomous
    structure) are unitary, and for all $f:A\ii B$, the equation
    $f^{*\dagger}=f^{\dagger*}$ holds.
  \end{itemize}
\end{lemma}

\begin{proof}
  By a direct calculation from the definitions, one can check three
  separate and independent facts:
  \begin{itemize}
  \item For any given $f:A\ii B$, the diagram
    \[ \xymatrix{
      A \ar[r]^{i_A} \ar[d]_{f} &
      A^{**} \ar[d]^{f^{**}} \\
      B \ar[r]^{i_B} &
      B^{**}
      }
    \]
    commutes if and only if $f^{*\dagger}=f^{\dagger*}$. In
    particular, the family $i_A$ is a natural transformation if and
    only if this condition holds for all $f$.
  \item The diagram from (\ref{eqn-pivotal-monoidal}),
    \[ \xymatrix@C=0cm{
      &
      A\x B \ar[dl]_{i_A\x i_B} \ar[dr]^{i_{A\x B}} \\
      A^{**}\x B^{**} \ar[rr]^{\iso} &&
      (A\x B)^{**}
      }
    \]
    commutes if and only if the canonical isomorphism $(A\x B)^*\iso B^*\x
    A^*$ is unitary.
  \item The morphism $i_I:I\ii I^{**}$ is equal to the canonical
    isomorphism (from the right autonomous structure) if and only if
    the canonical isomorphism $I\ii I^*$ is unitary.
  \end{itemize}
  Since the three conditions are the defining conditions for a pivotal
  structure, the lemma follows. \eot
\end{proof}

\begin{lemma}\label{lem-dagger-pivotal-2}
  Under the equivalent conditions of Lemma~\ref{lem-dagger-pivotal-1},
  the following hold:
  \begin{enumerate}\alphalabels
  \item[(a)] $i_A$ is unitary.
  \item[(b)] $i_A = A \catarrow{\iso} A\x
    I\catarrow{\id_A\x\eps\da_{A^*}} A\x A^*\x A^{**}
    \catarrow{\eps_A\x\id_{A^{**}}}I\x A^{**}\catarrow{\iso}
    A^{**}$:
    \[  \wirechart{}{\wire{r}{A}&\circbox{i_A}\wire{r}{A^{**}}&}
    \ssep=\ssep
    \raisebox{1ex}{\mnew{\wirechart{@C=1.0cm@R=0.5cm}{
          \blank\wireopenlabel{d}{\eps\da_{A^*}}\wire{r}{A^{**}}&\\
          \blank\wire{r}{A^*}&\blank\\
          \wire{r}{A}&\blank\wirecloselabel{u}{\eps_A}\\
          }}}
    \]
  \item[(c)] $\eta\da_A = \eps_{A^*}\cp(\id_{A^*}\x i_A)$:
    \[\raisebox{1ex}{\mnew{\wirechart{}{
          \wire{r}{A}&\blank\wirecloselabel{d}{\eta\da_A} \\
          \wire{r}{A^*}&\blank
          }}}
    \ssep=\ssep
    \raisebox{1ex}{\mnew{\wirechart{}{
          \wire{r}{A}&\circbox{i_A}\wire{r}{A^{**}}&\blank\wirecloselabel{d}{\eps_{A^*}} \\
          \wire{rr}{A^*}&&\blank
          }}}
    \]
  \item[(d)] $\eps\da_A = (i\inv_A\x\id_{A^*})\cp\eta_{A^*}$:
    \[\raisebox{1ex}{\mnew{\wirechart{}{
          \blank\wireopenlabel{d}{\eps\da_A}\wire{r}{A^*}&\\
          \blank\wire{r}{A}&
          }}}
    \ssep=\ssep
    \raisebox{1ex}{\mnew{\wirechart{}{
          \blank\wireopenlabel{d}{\eta_{A^*}}\wire{rr}{A^*}&&\\
          \blank\wire{r}{A^{**}}&\circbox{i\inv_A}\wire{r}{A}&
          }}}
    \]
  \end{enumerate}
\end{lemma}

\begin{proof}
  To prove (a), first consider
  \[ (i_A)\da 
  \ssep=\ssep
  \raisebox{1ex}{\mnew{\wirechart{@C=1.0cm@R=0.5cm}{
        \blank\wireopenlabel{d}{\eta_A}\wire{r}{A}&\\
        \blank\wire{r}{A^*}&\blank\\
        \wire{r}{A^{**}}&\blank\wirecloselabel{u}{\eta\da_{A^*}}\\
        }}}
  \]
  By definition of adjoint mates, we have
  \[ (i_A)^{\dagger*} 
    \ssep=\ssep
    \raisebox{1ex}{\mnew{\wirechart{@C=1.0cm@R=0.5cm}{
          \wire{r}{A^*}&\blank\wirecloselabel{d}{\eta\da_{A^*}}\\
          \blank\wire{r}{A^{**}}&\blank\\
          \blank\wireopenlabel{u}{\eta_{A^{**}}}\wire{r}{A^{***}}&
          }}}
  \]
  But this is just the definition of $i_{A^*}$, therefore
  $(i_A)^{\dagger*}=i_{A^*}$. By definition, $i_A$ is unitary iff
  $(i_A)\da=i_A^{-1}$, iff $(i_A)^{\dagger*}=(i_A^{-1})^*$, iff
  $i_{A^*} = (i_A^{-1})^* = (i_A^*)^{-1}$. Since $i$ is a monoidal
  natural transformation, this holds by Saavedra Rivano's Lemma
  (Lemma~\ref{lem-saavedra-rivano}). 
  
  To prove (b), note that the right-hand side is the inverse of
  $(i_A)\da$. Therefore, (b) is equivalent to (a). 
  
  Finally, equations (c) and (d) are restatements of the definition of
  $i_A$ from (\ref{eqn-definition-i}). \eot
\end{proof}

\begin{remark}
  The equivalence between (a) and (b) in
  Lemma~\ref{lem-dagger-pivotal-2} holds only if $i_A$ is defined as
  in (\ref{eqn-definition-i}). It does not hold for an arbitrary
  pivotal structure on a right autonomous dagger monoidal category.
\end{remark}

Armed with these results, we finally state the two equivalent
definitions of a dagger pivotal category:

\begin{definition}
  A {\em dagger pivotal category} is defined in one of the following
  equivalent ways:
  \begin{enumerate}
  \item as a dagger monoidal, right autonomous category such that the
    natural isomorphisms $(A\x B)^*\iso B^*\x A^*$ and $I^*\iso I$
    (from the right autonomous structure) are unitary, and such that
    $f^{*\dagger}=f^{\dagger*}$ holds for all morphisms $f$; or
  \item as a pivotal, dagger monoidal category satisfying the
    condition in Lemma~\ref{lem-dagger-pivotal-2}(c) (or equivalently,
    (d)).
  \end{enumerate}
\end{definition}

The first form of this definition is much easier to check in practice.
The second form is more suitable for the proof of the coherence
theorem below.

\begin{remark}
  In a dagger pivotal category, the operation $(-)^*$ arises from an
  adjunction (in the categorical sense) of {\em objects}, with
  associated unit, counit, and adjoint mates. On the other hand, the
  operation $(-)\da$ arises from an adjunction (in the linear algebra
  sense) of {\em morphisms}. The two concepts should not be confused
  with each other.
\end{remark}

\paragraph{Graphical language.}

The graphical language of dagger pivotal categories is like that of
pivotal categories, where the adjoint of a diagram is given, as usual,
by its mirror image. For example:
\[ \left[\mnew{\wirechart{@R-1ex}{
      \wireright{r}{A}&\blank\ulbox{[].[dd]}{g}&\blank \\
      &\blank\wireright{r}{C}& \\
      \blank\wireopen{dd}\wireright{r}{B}&\blank \\\\
      \blank\wireleft{rr}{B}&&
      }}\right]\largeda
\ssep=\ssep
\mnew{\wirechart{@R-1ex}{
      &\blank\urbox{[].[dd]}{g}\wireright{r}{A}& \\
      \wireright{r}{C}&\blank& \\
      \blank&\blank\wireright{r}{B}&\blank\wireclose{dd} \\\\
      \wireleft{rr}{B}&&\blank
      }}
\]
Note that in the graphical language, adjoint mates $f^*:B^*\ii A^*$
are represented by rotation and adjoints $f\da:B\ii A$ by mirror
image.  Therefore, each morphism variable $f:A\ii B$ induces four
kinds of boxes: 
\[ \begin{array}{rr}
  f=\vcenter{\wirechart{}{
      \wireright{r}{A}&\blank\ulbox{[]}{f}\wireright{r}{B}&
      }}
  &
  f\da=\vcenter{\wirechart{}{
      \wireright{r}{B}&\blank\urbox{[]}{f}\wireright{r}{A}&
      }}
  \\\\
  f^{*\dagger}=\vcenter{\wirechart{}{
      \wireleft{r}{A}&\blank\dlbox{[]}{f}\wireleft{r}{B}&
      }}
  &
  f^*=\vcenter{\wirechart{}{
      \wireleft{r}{B}&\blank\drbox{[]}{f}\wireleft{r}{A}&
      }}
\end{array}
\]

Also note that, unlike the informal notation used above, the graphical
language does not explicitly display the isomorphism $i_A:A\ii
A^{**}$, and it does not explicitly distinguish $\eta_A:I\ii A^*\x A$
from $\eps\da_{A^*}:I\ii A^*\x A^{**}$. This is justified by the
following coherence theorem.

\begin{theorem}[Coherence for dagger pivotal categories]
  A well-formed equation between morphisms in the language of dagger
  pivotal categories follows from the axioms of dagger pivotal
  categories if and only if it holds in the graphical language up to
  planar isotopy, including rotation of boxes (by multiples of 180
  degrees).
\end{theorem}

\begin{proof}
  This follows from coherence of pivotal categories
  (Theorem~\ref{thm-coherence-pivotal}), by the same argument used in
  the proof of Theorem~\ref{thm-coherence-dagger-monoidal}. The
  equations from Lemma~\ref{lem-dagger-pivotal-2}(c) and (d), and the
  fact that $i_A$ is unitary, can be used to replace $\eta\da_A$,
  $\eps\da_A$, and $i\da_A$ by equivalent terms not containing
  $\dagger$.\eot
\end{proof}

\subsection{Other dagger pivotal notions}

It is possible to define dagger variants of the remaining pivotal
notions from Section~\ref{sec-autonomous}:

\begin{definition}
  A dagger pivotal category is {\em spherical} (respectively {\em
    spacial}) if it is spherical (respectively spacial) as a pivotal
  category.
\end{definition}

\begin{definition}
  A {\em dagger braided pivotal category} is a dagger pivotal category
  with a unitary braiding $\sym_{A,B}:A\x B\ii B\x A$.
\end{definition}

\begin{remark}
  Like any braided pivotal category, a dagger braided pivotal category
  is balanced by Lemma~\ref{lem-braided-pivotal}. However, in general
  the resulting twist $\theta_A:A\ii A$ is not unitary. In fact,
  $\theta_A$ is unitary in this situation if and only if
  $\theta_{A^*}=(\theta_A)^*$, i.e., if and only if the category is
  tortile.
\end{remark}

\begin{definition}
  A {\em dagger tortile category} is defined in one of the following
  equivalent ways:
  \begin{enumerate}
  \item as a dagger braided pivotal category in which the canonical
    twist $\theta_A$, defined as in Lemma~\ref{lem-braided-pivotal},
    is unitary;
  \item as a tortile, dagger monoidal category such that the braiding
    is unitary, and such that $\eps_A$ and $\eta_A$ satisfy the
    (equivalent) conditions of Lemma~\ref{lem-dagger-pivotal-2}(c) and
    (d); or
  \item as a dagger balanced monoidal category that is right
    autonomous and satisfies
    \begin{equation}\label{eqn-theta-dagger}
      \theta_A 
      \ssep=\ssep
      \vcenter{\wirechart{@C-4ex}{
          *{}\wire{rr}{A}&&
          \blank\wirecross{d}\wire{rr}{A}&&
          *{}
          \\
          &\blank\wire{r}{}&
          \blank\wirebraid{u}{.3}\wire{r}{}&
          \blank
          \\
          &\blank\wireopenlabel{u}{\eta_A}\wire{rr}{A^*}&&
          \blank\wirecloselabel{u}{\eta\da_A}
          }}
    \end{equation}
  \end{enumerate}
\end{definition}
  
The first form of this definition emphasizes the relationship to
dagger pivotal categories. The second form is easiest to check if a
category is already known to be tortile. Finally, the third form takes
$\eps_A$, $\eta_A$, $\sym_{A,B}$ and $\theta_A$ as primitive
operations and does not mention the pivotal structure $i_A$ at all.
The pivotal structure, in this case, is definable from
(\ref{eqn-i-from-theta}) or (\ref{eqn-definition-i}), with the condition
(\ref{eqn-theta-dagger}) ensuring that the two definitions coincide.

\begin{definition}[\cite{AC04,Sel05}]
  A {\em dagger compact closed category} is a dagger tortile category
  such that $\theta_A=\id_A$ for all $A$. Equivalently, it is a dagger
  symmetric monoidal category that is right autonomous and satisfies
  \begin{equation}\label{eqn-dagger-compact-closed}
    \raisebox{1ex}{\mnew{\wirechart{}{
          \blank\wireopenlabel{d}{\eta_A}\wire{r}{A}&\\
          \blank\wire{r}{A^*}&
          }}}
    \ssep=\ssep
    \raisebox{1ex}{\mnew{\wirechart{}{
          \blank\wireopenlabel{d}{\eps\da_{A}}\wire{r}{A^*}&\blank\wirecross{d}\wire{r}{A}&\\
          \blank\wire{r}{A}&\blank\wirecross{u}\wire{r}{A^*}&
          }}}
  \end{equation}
\end{definition}

The equivalence of the two definition is immediate from the third form
of the definition of dagger tortile categories. Note that
(\ref{eqn-theta-dagger}) is equivalent to
(\ref{eqn-dagger-compact-closed}) in the symmetric case.  Further,
these conditions are equivalent to the condition in
Lemma~\ref{lem-dagger-pivotal-2}(d).

\begin{example}
  The category $\FdHilb$ of finite dimensional Hilbert spaces is
  dagger compact closed, with $A^*$ the usual dual space of linear
  functions from $A$ to $I$, and with $f\da$ the usual linear algebra
  adjoint.
\end{example}

\paragraph{Graphical languages.}

Each of the notions defined in this section (except the spherical
notion) has a graphical language, extending the corresponding
graphical language from Section~\ref{sec-autonomous}, with the dagger
of a diagram taken to be its mirror image. Each notion has a coherence
theorem, proved by the same method as
Theorems~\ref{thm-coherence-dagger-categories}
and~\ref{thm-coherence-dagger-monoidal}. As expected, equivalence of
diagrams is up to isomorphism (for spacial dagger pivotal categories);
up to regular isotopy (for dagger braided pivotal categories); up to
framed 3-dimensional isotopy (for dagger tortile categories); and up
to isomorphism (for dagger compact closed categories).

\subsection{Dagger traced categories}\label{subsec-dagger-traced}

There is no difficulty in defining dagger variants of each of the
traced notions of Section~\ref{sec-traced}. A (left or right) trace on
a dagger monoidal category is called a {\em dagger trace} if it
satisfies
\begin{equation}\label{eqn-trace-dagger}
  (\Tr f)\da = \Tr(f\da)
\end{equation}
For example: a {\em dagger right traced category} is a right traced
dagger monoidal category satisfying (\ref{eqn-trace-dagger}). A
balanced traced category is {\em dagger balanced traced} if it is
dagger balanced and satisfies (\ref{eqn-trace-dagger}). And similarly
for the other notions.  The representation theorems of
Section~\ref{sec-traced} extend to these dagger variants:

\begin{theorem}[Representation of dagger braided/balanced/symmetric traced categories]
  Every dagger braided [balanced, symmetric] traced category can be
  fully and faithfully embedded in a dagger braided pivotal [dagger
  tortile, dagger compact closed] category, via a dagger braided
  [balanced, symmetric] traced functor.\eot
\end{theorem}

The proof, in each case, is by Joyal, Street, and Verity's
Int-construction {\cite{JSV96}}, which respects the dagger structure.

\paragraph{Graphical languages.}

The graphical language of each class of traced categories extends to
the corresponding dagger traced categories, in a way suggested by
equation (\ref{eqn-trace-dagger}). As usual, the dagger of a diagram
is its mirror image, thus for example
\[ \left[\mnew{\wirechart{@C=1cm@R=.8cm}{
  \blank\wireopen{d}\wire{rr}{}&&
  \blank\wireclose{d}
  \\
  \blank\wireright{r}{X}&
  \blank\wireright{r}{X}&
  \blank
  \\
  *{}\wireright{r}{A}&
  \blank\ulbox{[].[u]}{f}\wireright{r}{B}&
  *{}
  \\}}\right]\largeda
\ssep=\ssep
\mnew{\wirechart{@C=1cm@R=.8cm}{
  \blank\wireopen{d}\wire{rr}{}&&
  \blank\wireclose{d}
  \\
  \blank\wireright{r}{X}&
  \blank\wireright{r}{X}&
  \blank
  \\
  *{}\wireright{r}{B}&
  \blank\urbox{[].[u]}{f}\wireright{r}{A}&
  *{}
  \\}}
\]
The coherence theorems of Section~\ref{sec-traced} extend to this setting. 

\subsection{Dagger biproducts}

In a dagger category, if $A\oplus B$ is a categorical product (with
projections $\pi_1:A\oplus B\ii A$ and $\pi_2:A\oplus B\ii B$), then
it is automatically a coproduct (with injections $\pi\da_1:A\ii
A\oplus B$ and $\pi\da_2:B\ii A\oplus B$). It therefore makes sense to
define a notion of {\em dagger biproduct}.

\begin{definition}
  A {\em dagger biproduct category} is a biproduct category carrying a
  dagger structure, such that $\pi\da_i=\iota_i:A_i\ii A_1\oplus A_2$
  for $i=1,2$.
\end{definition}

As in Section~\ref{subsec-biproduct}, we can equivalently define a
dagger biproduct category as a dagger symmetric monoidal category,
together with natural families of morphisms
\[ \Delta_A:A\ii A\x A,\sep \Diamond_A:A\ii I,\sep
\nabla_A:A\x A\ii A,\sep \Box_A:I\ii A,
\]
such that $\Delta\da_A=\nabla_A$ and $\Diamond\da_A=\Box_A$,
satisfying the axioms in Table~\ref{tab-product-axioms}.

\paragraph{Graphical language.}

The graphical language of dagger biproduct categories is like that of
biproduct categories, where the dagger of a diagram is taken to be its
mirror image. For example, 
\[ \left[\mnew{\resizebox{3.5cm}{!}{\includegraphics{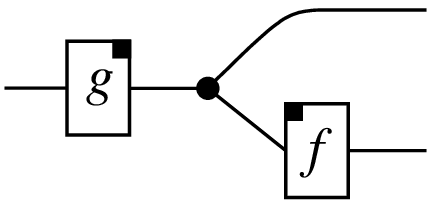}}}\right]\largeda
\ssep=\ssep
\mnew{\resizebox{3.5cm}{!}{\includegraphics{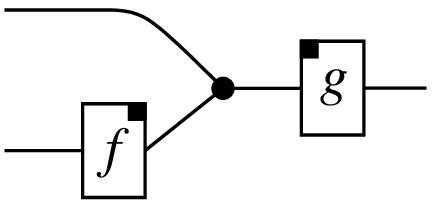}}}
\]

\begin{theorem}[Coherence for dagger biproduct categories]
  A well-formed equation between morphism terms in the language of
  dagger biproduct categories follows from the axioms of dagger biproduct
  categories if and only if it holds in the graphical language, up to
  isomorphism of diagrams, the diagrammatic manipulations shown
  in Table~\ref{tab-product-examples}, and their duals.
\end{theorem}

\begin{proof}
  By reduction to biproduct categories, as in the proofs of
  Theorems~\ref{thm-coherence-dagger-categories}
  and~\ref{thm-coherence-dagger-monoidal}. The axioms
  $\Delta\da_A=\nabla_A$ and $\Diamond\da_A=\Box_A$ allow $\dagger$ to
  be removed from anywhere but a morphism variable.\eot
\end{proof}

Finally, there is an obvious notion of {\em dagger traced biproduct
  category} (which is really a dagger traced dagger biproduct
category), with graphical language derived from traced biproduct
categories.

\section{Bicategories}

A bicategory {\cite{Ben67}} is a generalization of a monoidal
category. In addition to objects $A,B,\ldots$ and morphisms
$f,g,\ldots$, one now also considers {\em 0-cells}
$\alpha,\beta,\ldots$, which we can visualize as {\em colors}. For
example, consider the following diagram. It is a standard diagram for
monoidal categories, except that the areas between the wires have been
colored.
\[ \resizebox{1.6in}{!}{\includegraphics{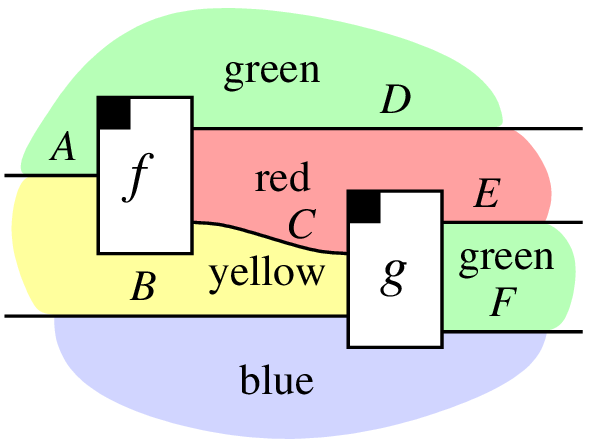}}
\]
As usual, we have objects $A,B,C,D,E,F$ and morphisms $f:A\ii C\x D$
and $g:B\x C\ii F\x E$. But now there are also 0-cells called green,
red, yellow, and blue. In such diagrams, each object has a {\em
  source}, which is the 0-cell just above it, and a {\em target},
which is the 0-cell just below it. For example, we have
$A:\mbox{green}\ii\mbox{yellow}$, $B:\mbox{yellow}\ii\mbox{blue}$, and
so on. It is now clear that, to be consistently colored, such diagrams
have to satisfy some coloring constraints. The constraints are:
\begin{itemize}
\item The tensor $B\x A$ of two objects may only be formed if the
  target of $A$ is equal to the source of $B$. In symbols, for any
  0-cells $\alpha,\beta,\gamma$, if $A:\alpha\ii\beta$ and
  $B:\beta\ii\gamma$, then $B\x A:\alpha\ii\gamma$.
\item If $f:A\ii B$ is a morphism, then $A$ and $B$ must have a common
  source and a common target. In symbols, if $f:A\ii B$ and
  $A:\alpha\ii\beta$, then $B:\alpha\ii\beta$.
\item One also requires a unit object $I_{\alpha}:\alpha\ii \alpha$
  for every color $\alpha$.
\end{itemize}

As an illustration of the second property, consider $f:A\ii C\x D$ in
the above example, where $A:\mbox{green}\ii\mbox{yellow}$ and $C\x
D:\mbox{green}\ii\mbox{yellow}$. Subject to the above coloring
constraints, a bicategory is then required to satisfy exactly the same
axioms as a monoidal category.  Notice, for example, that if $f:A\ii
B$ and $g:B\ii C$ and $f,g$ are well-colored, then so is $g\cp f:A\ii
C$. Also, the identity maps $\id_A:A\ii A$, the associativity map
$\alpha_{A,B,C}:(A\x B)\x C\catarrow{\iso} A\x(B\x C)$, and the other
structural maps are well-colored.  In particular, a monoidal category
is the same thing as a one-object bicategory.

To give a detailed account of bicategories and their graphical
languages is beyond the scope of this paper. We have already discussed
over 30 different flavors of monoidal categories, and the reader can
well imagine how many possible variations of bicategories there are,
with 2-, 3-, and 4-dimensional graphical languages, once one considers
bicategorical versions of braids, twists, adjoints, and traces. There
are even more variations if one considers tricategories and beyond.
We refer the reader to {\cite{Ben67}} for the definition and basic
properties of bicategories, and to {\cite{Str95}},
{\cite[Sec.~7]{BD95}} for a taste of their graphical languages.

\section{Beyond a single tensor product}\label{sec-star-autonomous}

All the categorical notions that we have considered in this paper have
just a single tensor product, which we represented as juxtaposition in
the graphical languages. For notions of categories with more than one
tensor product, the graphical languages get much more complicated.
The details are beyond the scope of this paper, so we just outline the
basics and give some references. 

Examples of categories with more than one tensor are linearly
distributive categories {\cite{CS97}} and *-autonomous categories
{\cite{Bar79}}. Both of these notions are models of multiplicative
linear logic {\cite{Gir87a}}. These categories have two tensors, often
called ``tensor'' and ``par'', and written
\[     A\x B
\sep\mbox{and}\sep
A\llamp B.
\]
The two tensors are related by some morphisms, such as $A\x(B\llamp
C)\ii (A\x B)\llamp C$, while other similar morphisms, such as $(A\x
B)\llamp C\ii A\x(B\llamp C)$, are not present.

To make a graphical language for more than one tensor product, one
must label the wires by morphism terms, rather than morphism
variables. One must also introduce special tensor and par nodes as
shown here:
\[ \mnew{\xymatrix@R-4ex{
  \ar[dr]^<>(.3){B} \\
  & \circbox{\x} \ar[r]^{A\x B} &\\
  \ar[ur]_<>(.3){A}
}}
\ssep
\mnew{\xymatrix@R-4ex{
  &&\\
  \ar[r]^{A\x B} & \circbox{\x} \ar[ur]^{B} \ar[dr]_{A}\\
  &&
}}
\ssep
\mnew{\xymatrix@R-4ex{
  \ar[dr]^<>(.3){B} \\
  & \circbox{\llamp} \ar[r]^{A\llamp B} &\\
  \ar[ur]_<>(.3){A}
}}
\ssep
\mnew{\xymatrix@R-4ex{
  &&\\
  \ar[r]^{A\llamp B} & \circbox{\llamp} \ar[ur]^{B} \ar[dr]_{A}\\
  &&
}},
\]
along with similar nodes for the units. Equivalence of diagrams must
be taken up to axiomatic manipulations, such as the following, which
is called {\em cut elimination} in logic:
\[ \mnew{\xymatrix@R-4ex{
  \ar[dr]^<>(.3){B} &&&\\
  & \circbox{\x} \ar[r]^{A\x B} & \circbox{\x} \ar[ur]^{B} \ar[dr]_{A}\\
  \ar[ur]_<>(.3){A}&&&
}}
\sep=\sep
\mnew{\xymatrix@R-4ex{
    \ar[rr]^{B}&&\\
    \ar[rr]_{A}&&
}}.
\]
Finally, one must state a {\em correctness criterion}, to explain why
certain diagrams, such as the left one following, are well-formed,
while others, such as the right one, are not well-formed.
\[ \xymatrix@R-4ex{
  &&\\
  \ar[r]^{A\x B} & \circbox{\x} \ar@/^3ex/[rr]^{B} \ar@/_3ex/[rr]_{A}&&
  \circbox{\x} \ar[r]^{A\x B} &\\
  &&
}
\sep
 \xymatrix@R-4ex{
  &&\\
  \ar[r]^{A\llamp B} & \circbox{\llamp} \ar@/^3ex/[rr]^{B} \ar@/_3ex/[rr]_{A}&&
  \circbox{\x} \ar[r]^{A\x B} &\\
  &&
}
\]
The resulting theory is called the theory of {\em proof nets}, and was
first given by Girard for unit-free multiplicative linear logic
{\cite{Gir87a}}. It was later extended to include the tensor units by
Blute et al.~{\cite{BCST96}}.

\section{Summary}\label{sec-summary}

\begin{table}
\newlength{\mywidth}
\newcommand{\fittowidth}[2]{\settowidth{\mywidth}{#2}\ifdim\mywidth>#1\resizebox{#1}{!}{#2}\else#2\fi}
\newcommand{\doubleline}[2]{\begin{tabular}{@{}l@{}}#1\\#2\end{tabular}}

\newcommand{\chartboxaux}[6]{%
\framebox{\resizebox{.55in}{!}{%
\begin{tabular}[b]{@{}p{1in}@{}}
  \fittowidth{1in}{#2} \\
  \multicolumn{1}{@{}c@{}}{$\m{\rule{0mm}{#1}}{#3}$}\\
  #4\hfill#5\hfill#6
\end{tabular}%
}}}

\newcommand{\chartbox}{\chartboxaux{.6in}}
\newcommand{\chartboxb}{\chartboxaux{.43in}}

\newbox\myboxa
\newbox\myboxb
\newcommand\fittobox[3]{\setbox\myboxa\hbox{\resizebox{#1}{!}{#3}}\setbox\myboxb\hbox{\resizebox{!}{#2}{#3}}\ifdim\ht\myboxa>\ht\myboxb\resizebox{!}{#2}{#3}\else\resizebox{#1}{!}{#3}\fi}

\newcommand{\gr}[1]{\mnew{\fittobox{1in}{0.55in}{\includegraphics{#1}}}}
\newcommand{\grb}[1]{\mnew{\fittobox{1in}{0.35in}{\includegraphics{#1}}}}
\newcommand{\eq}[2]{\resizebox{1in}{!}{\mnew{\resizebox{.4in}{!}{\includegraphics{#1}}}~\mnew{=}~\mnew{\resizebox{.4in}{!}{\includegraphics{#2}}}}}
\newcommand{\eqb}[4]{\resizebox{1in}{!}{\mnew{\resizebox{#1}{!}{\includegraphics{#3}}}~\mnew{=}~\mnew{\resizebox{#2}{!}{\includegraphics{#4}}}}}

\def\chartAA{\chartbox
 {Category}
 {\gr{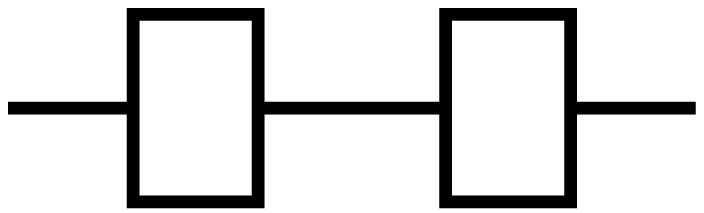}}
 {d:1}{i:1}{c:\chk}}
\def\chartAB{\chartbox
 {Planar monoidal}
 {\gr{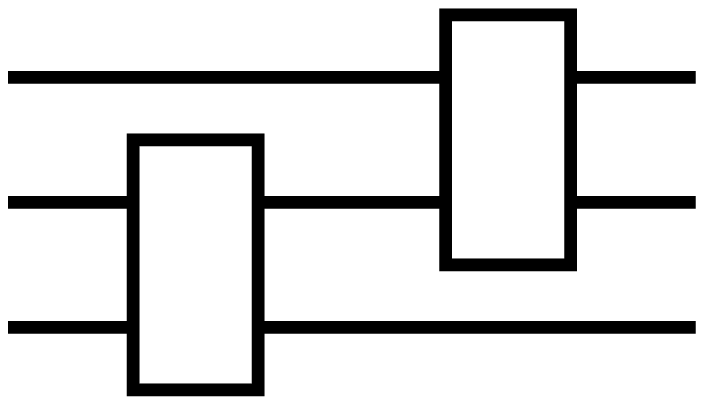}}
 {d:2}{i:2}{c:\cite{JS88,JS91}}}
\def\chartAC{\chartbox
 {Spacial monoidal}
 {\eq{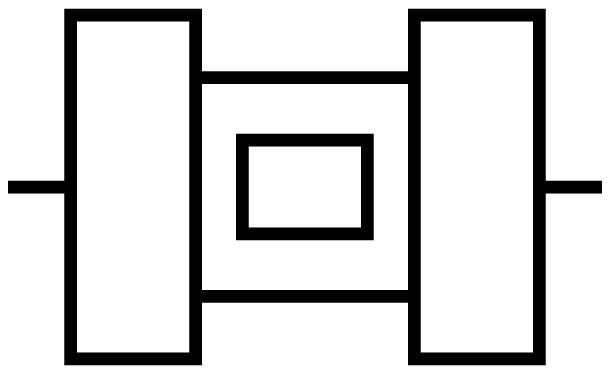}{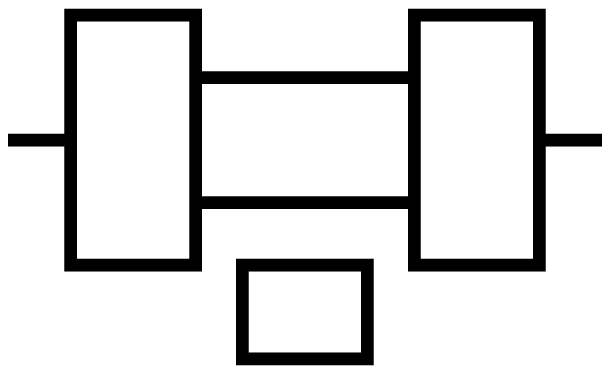}}
 {d:2}{i:3}{c:conj}}
\def\chartAD{\chartbox
 {Braided monoidal}
 {\gr{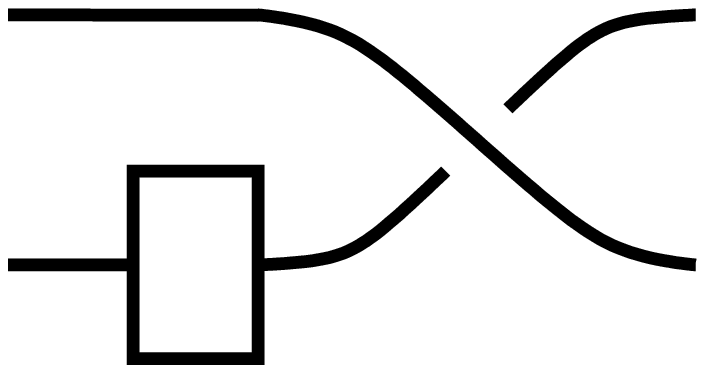}}
 {d:3}{i:3}{c:\cite{JS91}}}
\def\chartAE{\chartbox
 {Balanced monoidal}
 {\gr{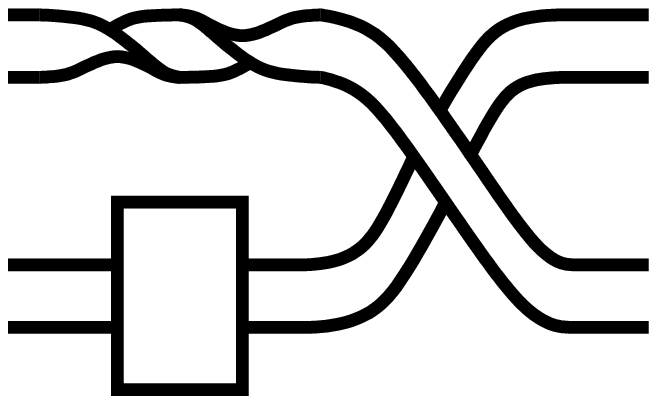}}
 {d:3f}{i:3f}{c:\cite{JS91}}}
\def\chartAF{\chartbox
 {Symmetric monoidal}
 {\gr{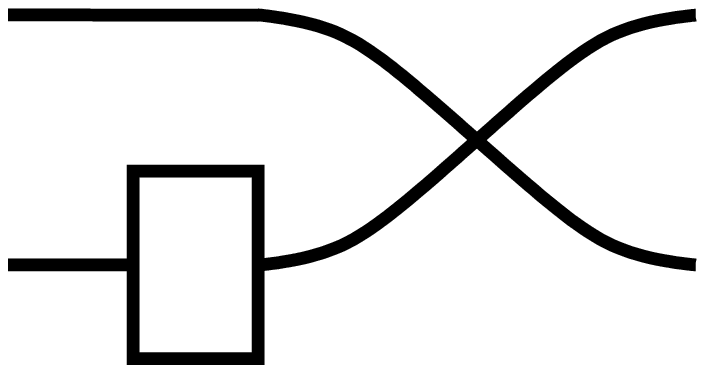}}
 {d:3}{i:4}{c:\cite{JS91}}}
\def\chartAG{\chartbox
 {Product}
 {\eq{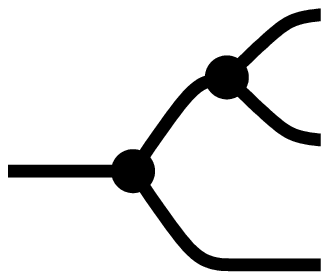}{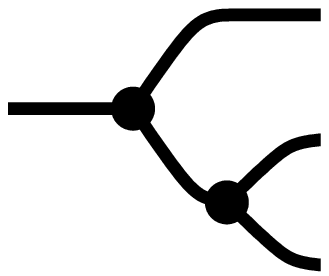}}
 {d:3}{i:eqn}{c:\chk}}
\def\chartAH{\chartbox
 {Coproduct}
 {\eq{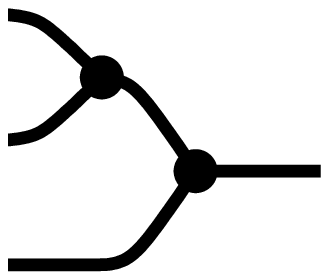}{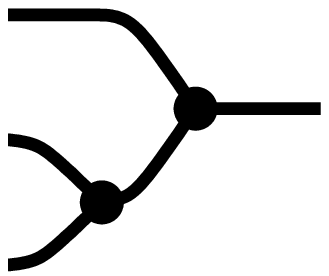}}
 {d:3}{i:eqn}{c:\chk}}
\def\chartAI{\chartbox
 {Biproduct}
 {\eqb{.35in}{.45in}{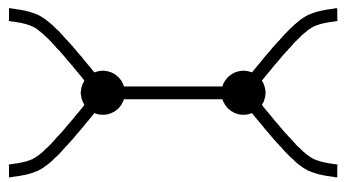}{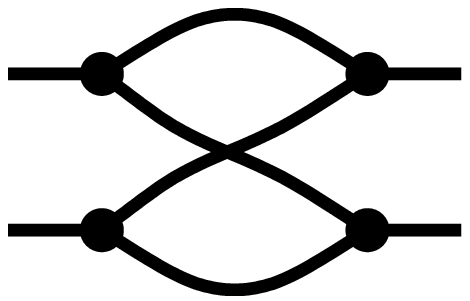}}
 {d:3}{i:eqn}{c:\chk}}

\def\chartBA{\chartbox
 {Right traced}
 {\gr{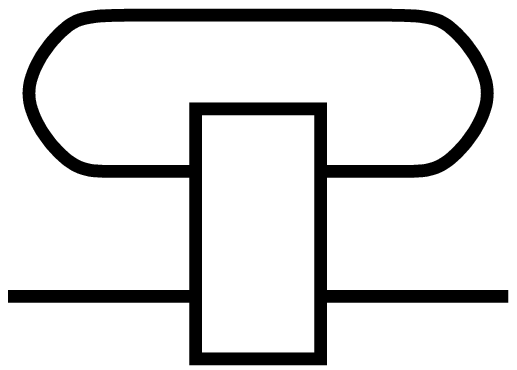}}
 {d:2}{i:2}{c:conj}}
\def\chartBB{\chartbox
 {Planar traced}
 {\gr{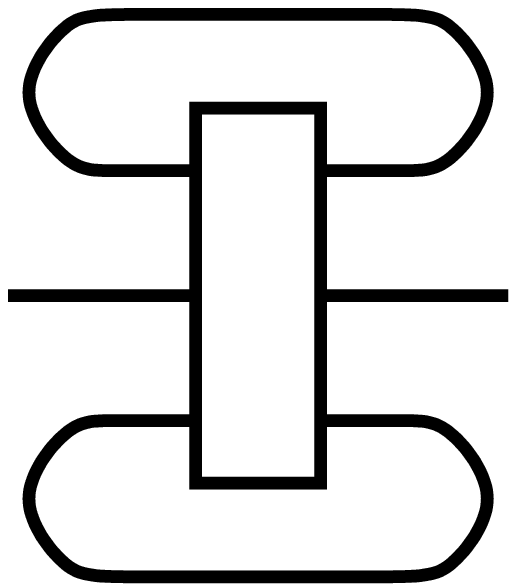}}
 {d:2}{i:2}{c:conj}}
\def\chartBC{\chartbox
 {Spacial traced}
 {\eq{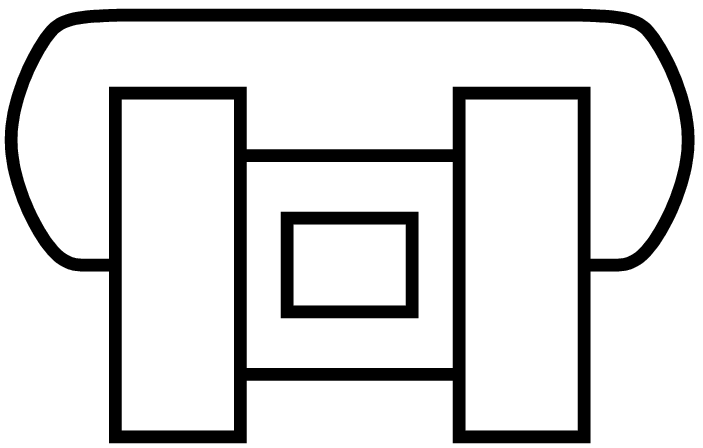}{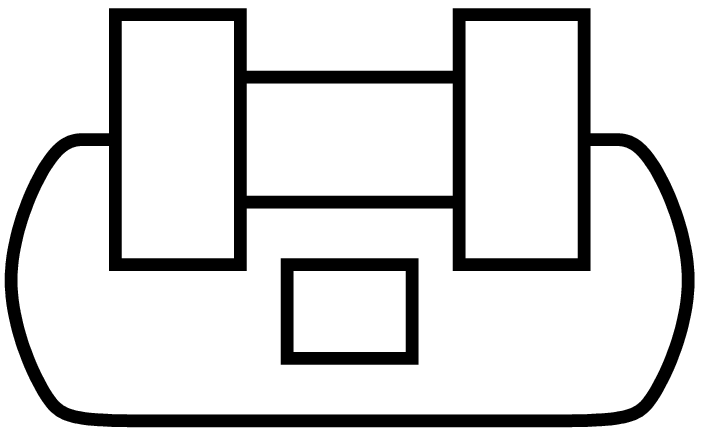}}
 {d:2}{i:3}{c:conj}}
\def\chartBD{\chartbox
 {Braided traced}
 {\gr{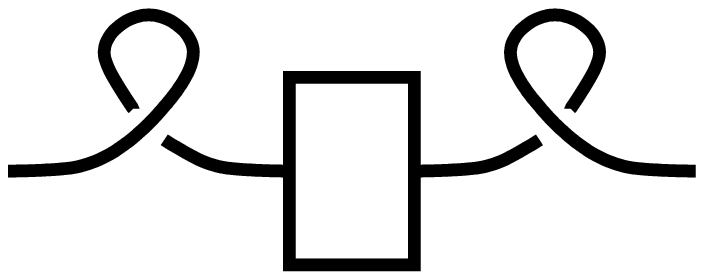}}
 {\small d:2${}^+$}{i:reg.rot}{c:int${}^*$}}
\def\chartBE{\chartbox
 {Balanced traced}
 {\doubleline{\eq{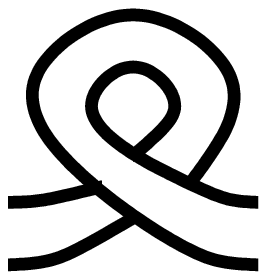}{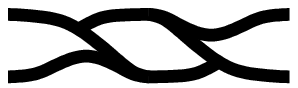}}{\eqb{.55in}{.25in}{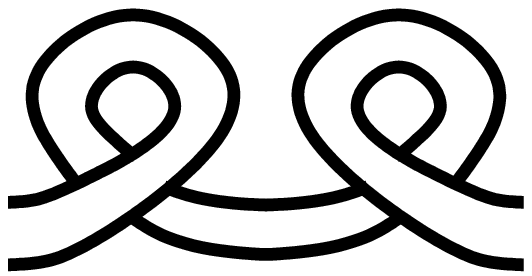}{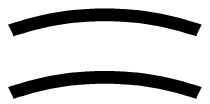}}}
 {d:3f}{i:3f}{c:int${}^*$}}
\def\chartBF{\chartbox
 {Symmetric traced}
 {\gr{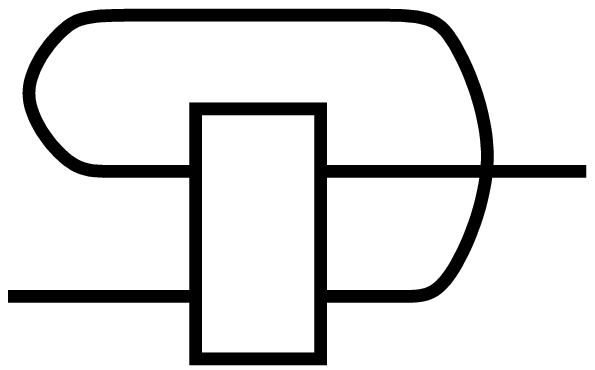}}
 {d:3}{i:4}{c:int${}^*$}}
\def\chartBG{\chartbox
 {Traced product}
 {\gr{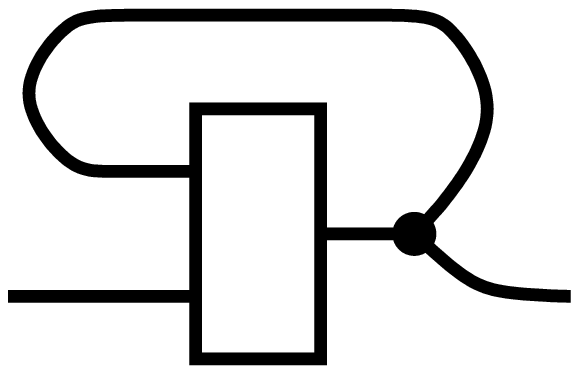}}
 {d:3}{i:eqn}{c:\chk}}
\def\chartBH{\chartbox
 {Traced coproduct}
 {\gr{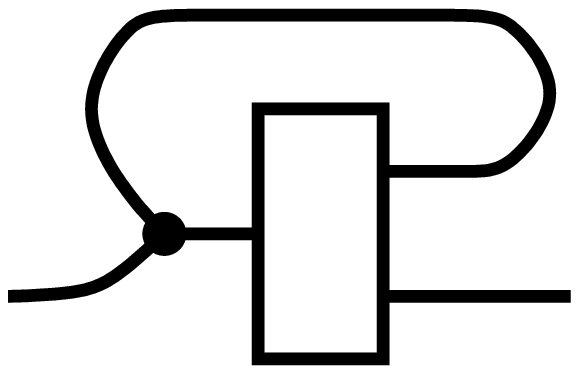}}
 {d:3}{i:eqn}{c:\chk}}
\def\chartBI{\chartbox
 {Traced biproduct}
 {\gr{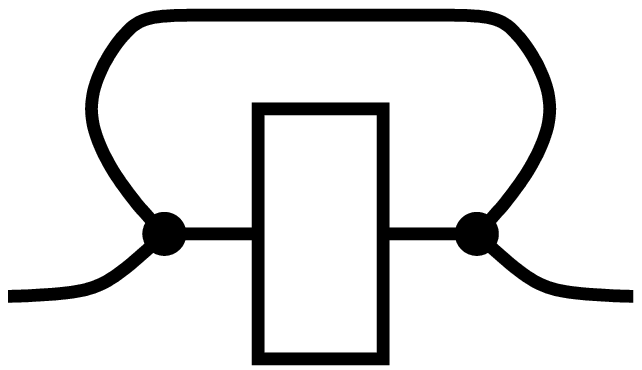}}
 {d:3}{i:eqn}{c:\chk}}

\def\chartCA{\chartboxb
 {\doubleline{Planar autonomous}{(rigid)}}
 {\grb{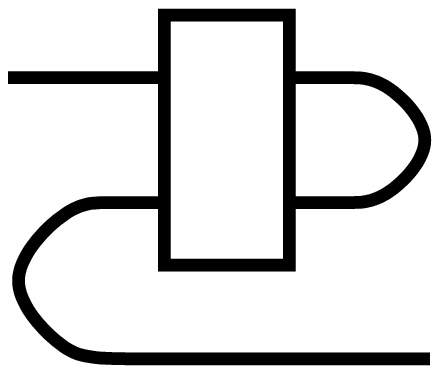}}
 {d:2}{i:2}{c:\cite{JS88}}}
\def\chartCB{\chartboxb
 {\doubleline{Planar pivotal}{(sovereign)}}
 {\grb{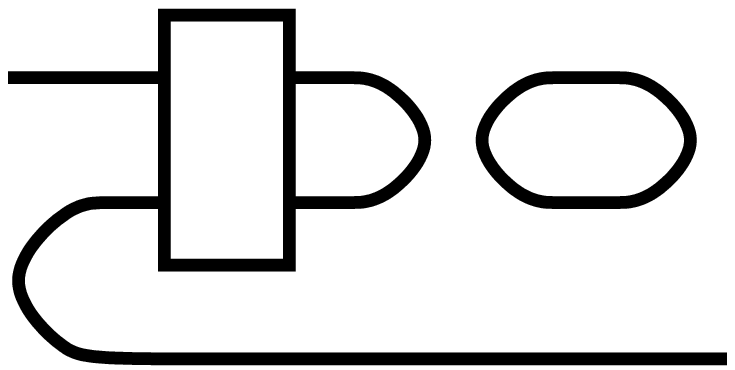}}
 {d:2}{i:2.rot}{c:\cite{FY92}${}^*$}}
\def\chartCC{\chartbox
 {Spacial pivotal}
 {\eq{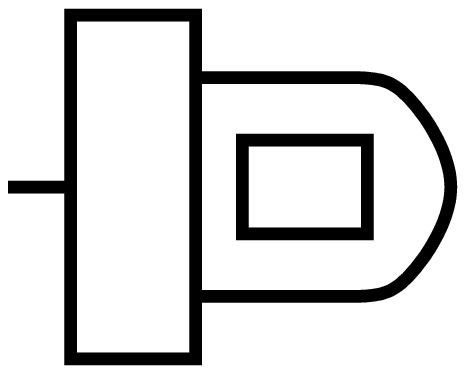}{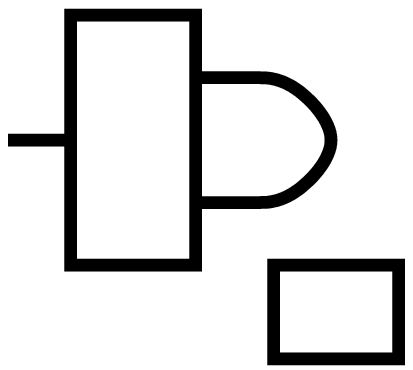}}
 {d:2}{i:3}{c:conj}}
\def\chartCD{\chartbox
 {Braided autonomous}
 {\gr{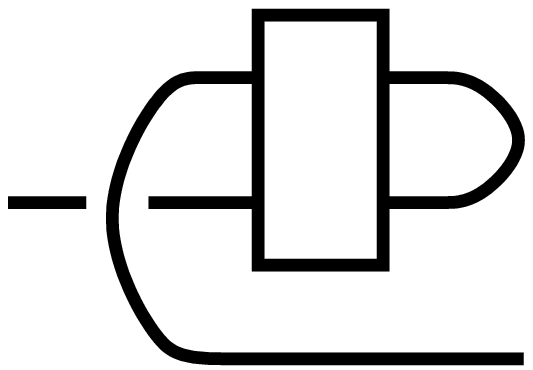}}
 {d:2${}^+$}{i:reg}{c:\cite{FY92}${}^*$}}
\def\chartCE{\chartboxb
 {\doubleline{Braided pivotal}{\fittowidth{1in}{(balanced autonomous)}}}
 {\grb{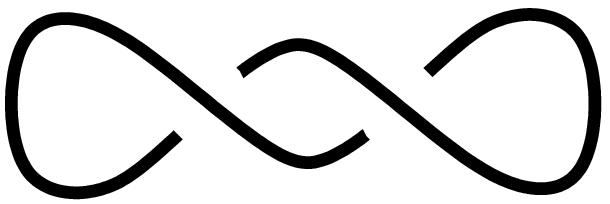}}
 {\small d:2${}^+$}{i:reg.rot}{c:\cite{FY92}${}^*$}}
\def\chartCF{\chartbox
 {Tortile (ribbon)}
 {\gr{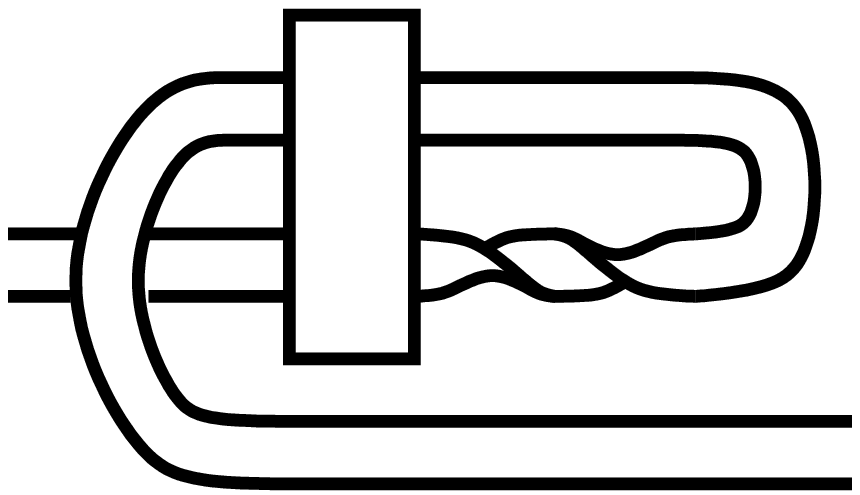}}
 {d:3f}{i:3f}{c:\cite{Shum94}${}^*$}}
\def\chartCG{\chartbox
 {Compact closed}
 {\gr{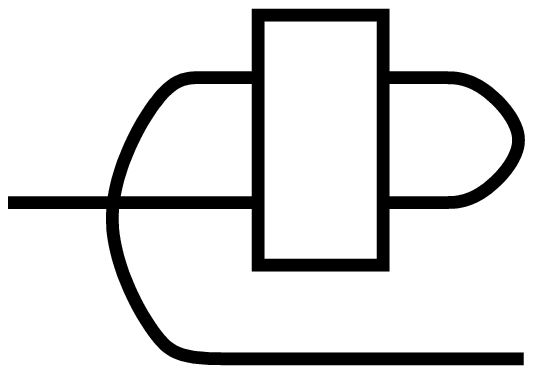}}
 {d:3}{i:4}{c:\cite{KL80}${}^*$}}

\[
\xymatrix@!R=.42in@C=.7in{
  *\txt{\bf Progressive} &
  *\txt{\bf Traced} &
  *\txt{\bf Autonomous} \\
  *{\chartAA} &
  *{\chartBA}\ar[dl] &
  *{\chartCA}\ar@/_.75in/[]!L;[dll]!UR \\
  *{\chartAB}\ar[u] &
  *{\chartBB}\ar[u]\ar[l] &
  *{\chartCB}\ar[u]\ar[l] \\
  *{\chartAC}\ar[u] &
  *{\chartBC}\ar[u]\ar[l] &
  *{\chartCC}\ar[u]\ar[l]
  \save+<1in,0cm>*{\chartCD}="box43"\ar[uu]!RD \restore \\
  *{\chartAD}\ar[u] &
  \save+<.7in,0cm>*{\chartBD}="box24"\ar[l]\ar!U;[uu]!RD \restore &
  \save+<1in,0cm>*{\chartCE}="box34"\ar"box24"\ar!UL;[uu]!RD\ar"box34";"box43" \restore \\
  *{\chartAE}\ar[u] &
  *{\chartBE}\ar[uu]\ar"box24"\ar[l] &
  *{\chartCF}\ar[uu]\ar!UR;"box34"!LD\ar[l] \\
  *{\chartAF}\ar[u] &
  *{\chartBF}\ar[u]\ar[l] &
  *{\chartCG}\ar[u]\ar[l] \\
  \save+<-.2in,0in>*{\chartAG}="box17"\ar!U+<.1in,0in>;[u]!D+<-.1in,0in> \restore &
  \save+<-.2in,0in>*{\chartBG}="box27"\ar!U+<.1in,0in>;[u]!D+<-.1in,0in>\ar"box17" \restore \\
  \save+<.2in,0in>*{\chartAH}="box18"\ar!U+<.1in,0in>;[uu]!D+<.3in,0in> \restore &
  \save+<.2in,0in>*{\chartBH}="box28"\ar!U+<.1in,0in>;[uu]!D+<.3in,0in>\ar"box18" \restore \\
  *{\chartAI}\ar!U+<.1in,0in>;"box18"!D+<-.1in,0in>\ar!U+<-.3in,0in>;"box17"!D+<-.1in,0in> &
  *{\chartBI}\ar!U+<.1in,0in>;"box28"!D+<-.1in,0in>\ar!U+<-.3in,0in>;"box27"!D+<-.1in,0in>\ar[l] \\
}
\]
\caption{Summary of monoidal notions and their graphical languages}
\label{tab-chart}
\end{table}

Table~\ref{tab-chart} summarizes the graphical languages from
Sections~\ref{sec-categories}--{\ref{sec-products}}. The name of each
class of categories is shown along with a typical diagram or equation.
The arrows indicate forgetful functors.  We have omitted spherical
categories, because they do not possess a graphical language modulo a
natural notion of isotopy.

The letter $d$ indicates the dimension of the diagrams, and the letter
$i$ indicates the dimension of the ambient space for isotopy.  If
$i>d$, then isotopy coincides with isomorphism of diagrams.  Special
cases are ``3f'' for framed diagrams and framed isotopy in 3
dimensions; ``2+'' for two-dimensional diagram with crossings (i.e.,
isotopy is taken on 2-dimensional projections, rather than on
3-dimensional diagrams); ``reg'' for regular isotopy; and ``rot'' to
indicate that isotopy includes rotation of boxes. Finally, ``eqn''
indicates that equivalence of diagrams is taken modulo equational
axioms.

The letter $c$ indicates the status of a coherence theorem. This is
usually a reference to a proof of the theorem, or ``conj'' if the
result is conjectured. A checkmark ``$\chk$'' indicates a result that
is folklore or whose proof is trivial. ``int'' indicates that the
coherence theorem follows from a version of Joyal, Street, and
Verity's Int-construction, and the corresponding coherence theorem for
pivotal categories. An asterisk ``$*$'' indicates that the result has
only been proved for simple signatures.

Dagger variants can be defined of all of the notions shown in
Table~\ref{tab-chart}, except the planar autonomous and braided
autonomous notions.  Finally, bicategories require their own
(presumably much larger) table and are not included here.


\newpage

\bibliographystyle{abbrv}
\bibliography{graphical}

\begin{thebibliography}{10}

\bibitem{AC04}
S.~Abramsky and B.~Coecke.
\newblock A categorical semantics of quantum protocols.
\newblock In {\em Proceedings of the 19th Annual IEEE Symposium on Logic in
  Computer Science, LICS 2004}, pages 415--425. IEEE Computer Society Press,
  2004.

\bibitem{BD95}
J.~C. Baez and J.~Dolan.
\newblock Higher-dimensional algebra and topological quantum field theory.
\newblock {\em Journal of Mathematical Physics}, 36(11):6073--6105, 1995.

\bibitem{Bai76}
E.~S. Bainbridge.
\newblock Feedback and generalized logic.
\newblock {\em Information and Control}, 31:75--96, 1976.

\bibitem{Bar79}
M.~Barr.
\newblock {\em *-Autonomous Categories}.
\newblock Lectures Notes in Mathematics 752. Springer, 1979.

\bibitem{BW99}
J.~W. Barrett and B.~W. Westbury.
\newblock Spherical categories.
\newblock {\em Advances in Mathematics}, 143:357--375, 1999.

\bibitem{Ben67}
J.~B\'enabou.
\newblock Introduction to bicategories, part {I}.
\newblock In {\em Reports of the Midwest Category Seminar}, Lecture Notes in
  Mathematics 47, pages 1--77. Springer, 1967.

\bibitem{10}
S.~L. Bloom and Z.~{\'E}sik.
\newblock Axiomatizing schemes and their behaviors.
\newblock {\em Journal of Computer and System Sciences}, 31:375--393, 1985.

\bibitem{BCST96}
R.~F. Blute, J.~R.~B. Cockett, R.~A.~G. Seely, and T.~H. Trimble.
\newblock Natural deduction and coherence for weakly distributive categories.
\newblock {\em Journal of Pure and Applied Algebra}, 113(3):229--296, 1996.

\bibitem{25}
V.-E. C{\u{a}}z{\u{a}}nescu.
\newblock On context-free trees.
\newblock {\em Theoretical Comput. Sci.}, 41:33--50, 1985.

\bibitem{12}
V.-E. C{\u{a}}z{\u{a}}nescu and G.~{\c{S}}tef{\u{a}}nescu.
\newblock Towards a new algebraic foundation of flowchart scheme theory.
\newblock {\em Fundamenta Informaticae}, 13:171--210, 1990.
\newblock Also appeared as: INCREST Preprint Series in Mathematics 43,
  Bucharest, 1987.

\bibitem{8}
V.-E. C{\u{a}}z{\u{a}}nescu and G.~{\c{S}}tef{\u{a}}nescu.
\newblock Feedback, iteration and repetition.
\newblock In G.~P{\u{a}}un, editor, {\em Mathematical aspects of natural and
  formal languages}, pages 43--62. World Scientific, Singapore, 1995.
\newblock Also appeared as: INCREST Preprint Series in Mathematics 42,
  Bucharest, 1988.

\bibitem{22}
V.-E. C{\u{a}}z{\u{a}}nescu and C.~Ungureanu.
\newblock Again on advice on structuring compilers and proving them correct.
\newblock Preprint Series in Mathematics~75, INCREST, Bucharest, 1982.

\bibitem{CS97}
J.~R.~B. Cockett and R.~A.~G. Seely.
\newblock Weakly distributive categories.
\newblock {\em Pure and Applied Algebra}, 114(2):133--173, 1997.

\bibitem{15}
Z.~{\'E}sik.
\newblock Identities in iterative and rational algebraic theories.
\newblock {\em Computational Linguistics and Computer Languages}, XIV:183--207,
  1980.

\bibitem{FY89}
P.~J. Freyd and D.~N. Yetter.
\newblock Braided compact closed categories with applications to low
  dimensional topology.
\newblock {\em Advances in Mathematics}, 77:156--182, 1989.

\bibitem{FY92}
P.~J. Freyd and D.~N. Yetter.
\newblock Coherence theorems via knot theory.
\newblock {\em Journal of Pure and Applied Algebra}, 78:49--76, 1992.

\bibitem{Gir87a}
J.-Y. Girard.
\newblock Linear logic.
\newblock {\em Theoretical Comput. Sci.}, 50:1--102, 1987.

\bibitem{Has97}
M.~Hasegawa.
\newblock {\em Models of Sharing Graphs: A Categorical Semantics of let and
  letrec}.
\newblock PhD thesis, Department of Computer Science, University of Edinburgh,
  July 1997.

\bibitem{JSXX}
A.~Joyal and R.~Street.
\newblock The geometry of tensor calculus {II}.
\newblock Unpublished draft, available from Ross Street's website.

\bibitem{JS86}
A.~Joyal and R.~Street.
\newblock Braided monoidal categories.
\newblock Mathematics Report 860081, Macquarie University, Nov. 1986.

\bibitem{JS88}
A.~Joyal and R.~Street.
\newblock Planar diagrams and tensor algebra.
\newblock Unpublished manuscript, available from Ross Street's website, Sept.
  1988.

\bibitem{JS91}
A.~Joyal and R.~Street.
\newblock The geometry of tensor calculus {I}.
\newblock {\em Advances in Mathematics}, 88(1):55--112, 1991.

\bibitem{JS93}
A.~Joyal and R.~Street.
\newblock Braided tensor categories.
\newblock {\em Advances in Mathematics}, 102:20--78, 1993.

\bibitem{JSV96}
A.~Joyal, R.~Street, and D.~Verity.
\newblock Traced monoidal categories.
\newblock {\em Mathematical Proceedings of the Cambridge Philosophical
  Society}, 119:447--468, 1996.

\bibitem{Kau86}
L.~H. Kauffman.
\newblock An invariant of regular isotopy.
\newblock {\em Transactions of the American Mathematical Society},
  318(2):417--471, 1990.

\bibitem{K72b}
G.~M. Kelly.
\newblock An abstract approach to coherence.
\newblock In S.~Mac~Lane, editor, {\em Coherence in Categories}, Lecture Notes
  in Mathematics 281, pages 106--147. Springer, 1972.

\bibitem{KL80}
G.~M. Kelly and M.~L. Laplaza.
\newblock Coherence for compact closed categories.
\newblock {\em Journal of Pure and Applied Algebra}, 19:193--213, 1980.

\bibitem{ML63}
S.~{Mac Lane}.
\newblock Natural associativity and commutativity.
\newblock {\em Rice University Studies}, 49:28--46, 1963.

\bibitem{ML71}
S.~{Mac Lane}.
\newblock {\em Categories for the Working Mathematician}.
\newblock Graduate Texts in Mathematics 5. Springer, 1971.

\bibitem{Pen71}
R.~Penrose.
\newblock Applications of negative dimensional tensors.
\newblock In D.~J.~A. Welsh, editor, {\em Combinatorial Mathematics and its
  Applications}, pages 221--244. Academic Press, New York, 1971.

\bibitem{Rei32}
K.~Reidemeister.
\newblock {\em Knotentheorie}.
\newblock Springer, Berlin, 1932; Chelsea, New York, 1948.
\newblock English translation: {\em Knot Theory}, BCS Associates, 1983.

\bibitem{SR72}
N.~Saavedra~Rivano.
\newblock {\em Categories Tanakiennes}.
\newblock Lecture Notes in Mathematics 265. Springer, 1972.

\bibitem{Sel05}
P.~Selinger.
\newblock Dagger compact closed categories and completely positive maps.
\newblock In {\em Proceedings of the 3rd International Workshop on Quantum
  Programming Languages}, Electronic Notes in Theoretical Computer Science 170,
  pages 139--163. Elsevier Science, 2007.

\bibitem{Shum94}
M.~C. Shum.
\newblock Tortile tensor categories.
\newblock {\em Journal of Pure and Applied Algebra}, 93:57--110, 1994.

\bibitem{5}
G.~{\c{S}}tef{\u{a}}nescu.
\newblock An algebraic theory of flowchart schemes.
\newblock In {\em Proceedings of the 11th Colloquium on Trees in Algebra and
  Programming, CAAP'86}, Lecture Notes in Computer Science 214, pages 60--73.
  Springer, 1986.

\bibitem{2a}
G.~{\c{S}}tef{\u{a}}nescu.
\newblock On flowchart theories, part {I}. the deterministic case.
\newblock {\em Journal of Computer and System Sciences}, 35(2):163--191, 1987.

\bibitem{2b}
G.~{\c{S}}tef{\u{a}}nescu.
\newblock On flowchart theories: Part {II}. the nondeterministic case.
\newblock {\em tcs}, 52:307--340, 1987.

\bibitem{6}
G.~{\c{S}}tef{\u{a}}nescu.
\newblock Feedback theories (a calculus for isomorphism classes of flowchart
  schemes).
\newblock {\em Revue Roumaine de Math\'ematiques Pures et Appliqu\'ees},
  35:73--79, 1990.
\newblock Also appeared as: INCREST Preprint Series in Mathematics 24,
  Bucharest, 1986.

\bibitem{11}
G.~{\c{S}}tef{\u{a}}nescu.
\newblock Algebra of flownomials.
\newblock Technical Report TUM-I9437, Technische Univerit\"at M\"unchen, 1994.

\bibitem{4}
G.~{\c{S}}tef{\u{a}}nescu.
\newblock {\em Network Algebra}.
\newblock Springer, 2000.

\bibitem{Str95}
R.~Street.
\newblock Low-dimensional topology and higher-order categories.
\newblock In {\em Proceedings of the International Category Theory Conference
  (CT'95)}, 1995.
\newblock Available from http://www.mta.ca/$\sim$cat-dist/ct95.html.

\bibitem{Tur94}
V.~G. Turaev.
\newblock {\em Quantum Invariants of Knots and 3-Manifolds}.
\newblock Studies in Mathematics 18. Walter De Gruyter \& Co., Berlin, 1994.

\bibitem{Yet92}
D.~N. Yetter.
\newblock Framed tangles and a theorem of {Deligne} on braided deformations of
  tannakian categories.
\newblock In M.~Gerstenhaber and J.~D. Stasheff, editors, {\em Deformation
  Theory and Quantum Groups with Applications to Mathematical Physics},
  Contemporary Mathematics 134, pages 325--349. Americal Mathematical Society,
  1992.

\end{thebibliography}
\end{document}